\documentclass[a4paper,reqno,11pt]{amsart}

\usepackage[american]{babel}

\usepackage[utf8]{inputenc}
\usepackage[T1]{fontenc}
\usepackage{microtype}
\usepackage[a4paper, scale={0.72,0.74}, marginratio={1:1}, footskip=7mm, headsep=10mm]{geometry}

\usepackage{indentfirst}

\usepackage{emptypage}  

\usepackage{perpage}

\usepackage{enumitem}

\usepackage{fixltx2e}

\usepackage{color}

\usepackage{multicol}

\usepackage{hyperref}

\usepackage{booktabs}
\usepackage{graphicx}
\usepackage{tabularx}
\usepackage{caption,subcaption}

\usepackage{longtable}
\usepackage{tikz,tikz-cd}
\usetikzlibrary{matrix, calc, arrows, decorations.markings, positioning}

\usepackage{amsmath, amssymb, amsthm}
\usepackage{stmaryrd, mathtools, mathdots}
\usepackage{eufrak}
\usepackage{wasysym}
\usepackage[vcentermath]{youngtab}
\usepackage{empheq}
\usepackage{cases}

\usepackage{accents}

\usepackage{bm,bbm}

\usepackage{braket}

\numberwithin{equation}{section}

\usepackage{subdepth}

\usepackage{comment}
\usepackage[style=alphabetic, backend=bibtex8, citestyle=alphabetic, firstinits=true, maxnames=99, maxalphanames=5, sorting=nyt]{biblatex}
\DeclareFieldFormat[article, inbook, incollection, inproceedings, patent, thesis, unpublished]{title}{#1} 
\renewbibmacro{in:}{}

\newenvironment{myitemize}{%
\begin{list}{$\bullet$}%
 	{%
	\setlength{\itemsep}{0.4em}%
	\setlength{\topsep}{0.5em}%
	\setlength\leftmargin{2.45em}%
	\setlength\labelwidth{2.05em}%
	\setlength{\labelsep}{0.4em}%
	}%
	}%
{\end{list}}
\renewenvironment{itemize}{
\begin{myitemize}}%
{\end{myitemize}}

\MakePerPage{footnote}
\makeatletter
\newcommand*{\myfnsymbolsingle}[1]{%
  \ensuremath{%
    \ifcase#1
    \or 
      \dagger
    \else 
      \@ctrerr  
    \fi
  }%
}   
\makeatother

\usepackage{alphalph}
\newalphalph{\myfnsymbolmult}[mult]{\myfnsymbolsingle}{}

\theoremstyle{plain}
\newtheorem{theorem}{Theorem}[section]

\newtheorem{proposition}[theorem]{Proposition}
\newtheorem{corollary}[theorem]{Corollary}

\theoremstyle{definition}

\newtheorem{remark}[theorem]{Remark}
\newtheorem{example}[theorem]{Example}

\makeatletter
\renewcommand{\@secnumfont}{\bfseries}
\renewcommand\section{\@startsection{section}{1}%
\z@{.7\linespacing\@plus\linespacing}{.5\linespacing}%
{\large\bfseries\scshape\centering}}
\renewcommand\subsection{\@startsection{subsection}{2}%
  \z@{.5\linespacing\@plus.7\linespacing}{-.5em}%
  {\bfseries\scshape}}
\makeatother

\makeatletter
\DeclarePairedDelimiter\abs{\lvert}{\rvert} 
\let\oldabs\abs
\def\abs{\@ifstar{\oldabs}{\oldabs*}}
\DeclarePairedDelimiterX{\norm}[1]{\lVert}{\rVert}{#1} 
\let\oldnorm\norm
\def\norm{\@ifstar{\oldnorm}{\oldnorm*}}
\DeclarePairedDelimiterX{\ceil}[1]{\lceil}{\rceil}{#1} 
\let\oldceil\ceil
\def\ceil{\@ifstar{\oldceil}{\oldceil*}}
\DeclarePairedDelimiterX{\floor}[1]{\lfloor}{\rfloor}{#1} 
\let\oldfloor\floor
\def\floor{\@ifstar{\oldfloor}{\oldfloor*}}
\makeatother

\newcommand\sfM{\mathsf M}

\newcommand\sfP{\mathsf P}
\newcommand\sfQ{\mathsf Q}

\newcommand\sfT{\mathsf T}

\newcommand\sfW{\mathsf W}

\newcommand\sfX{\mathsf X}

\newcommand\sfe{\mathsf e}

\newcommand\sfh{\mathsf h}

\newcommand\sft{\mathsf t}

\newcommand\sfw{\mathsf w}

\newcommand{\cQ}{{\ensuremath{\mathcal Q}} }
\newcommand{\cR}{{\ensuremath{\mathcal R}} }

\newcommand{\bbE}{{\ensuremath{\mathbb E}} }

\newcommand{\bbP}{{\ensuremath{\mathbb P}} }

\renewcommand{\P}{\mathbb{P}} 
\newcommand{\E}{\mathbb{E}} 
\newcommand{\1}{\mathbbm{1}} 
\newcommand{\ind}{\1}

\renewcommand{\i}{\mathrm{i}} 
\newcommand{\schur}{\mathrm{s}} 

\newcommand{\sumtwo}[2]{\sum_{\substack{#1 \\ #2}}} 
\newcommand{\sumthree}[3]{\sum_{\substack{#1 \\ #2 \\ #3}}} 

\newcommand{\id}{\mathrm{id}}

\newcommand{\rsk}{\mathsf{RSK}}
\newcommand\drsk{\mathsf{dRSK}}
\newcommand{\dtasep}{\mathsf{dTASEP}}
\newcommand{\tasep}{\mathsf{TASEP}}

\DeclareMathOperator{\sh}{\mathrm{sh}}
\DeclareMathOperator{\ledge}{\mathrm{l-edge}}

\renewcommand{\emptyset}{\varnothing}
\newcommand{\Z}{\mathbb{Z}} 
\newcommand{\R}{\mathbb{R}} 
\newcommand{\C}{\mathbb{C}} 

\renewcommand{\epsilon}{\varepsilon}
\renewcommand{\rho}{\varrho}
\renewcommand{\phi}{\varphi}

\DeclareMathSymbol{\widehatsym}{\mathord}{largesymbols}{"62}

\renewcommand{\hat}{\widehat}

\renewcommand{\tilde}{\widetilde}

\makeatletter
\newcommand{\doubletilde}[1]{{%
  \mathpalette\double@tilde{#1}%
}}
\newcommand{\double@tilde}[2]{%
  \sbox\z@{$\m@th#1\tilde{#2}$}%
  \ht\z@=.9\ht\z@
  \tilde{\box\z@}%
}
\makeatother

\newenvironment{myenumerate}{%
\renewcommand{\theenumi}{(\roman{enumi})}%
\renewcommand{\labelenumi}{\theenumi}%
\begin{list}{\labelenumi}
	{%
	\setlength{\itemsep}{0.4em}%
	\setlength{\topsep}{0.5em}%
	\setlength\leftmargin{2.45em}%
	\setlength\labelwidth{2.05em}%
	\setlength{\labelsep}{0.4em}%
	\usecounter{enumi}%
	}%
	}%
{\end{list}
}
\renewenvironment{enumerate}{
\begin{myenumerate}}%
{\end{myenumerate}}

  {\left\lbrace\begin{array}{@{}l@{}}}%
  {\end{array}\right.}

\makeatletter
\newsavebox{\mybox}\newsavebox{\mysim}
\newcommand{\distras}[1]{%
  \savebox{\mybox}{\hbox{\kern3pt$\scriptstyle#1$\kern3pt}}%
  \savebox{\mysim}{\hbox{$\sim$}}%
  \mathbin{\overset{#1}{\kern\z@\resizebox{\wd\mybox}{\ht\mysim}{$\sim$}}}%
}
\makeatother

\begin{document}

\title[Non-intersecting path constructions for TASEP with inhomogeneous rates]{\large Non-intersecting path constructions \\ for TASEP with inhomogeneous rates \\ and the KPZ fixed point}

\author[E.~Bisi]{Elia Bisi}
\address{Elia Bisi, Technische Universit\"at Wien \\
Institut f\"ur Stochastik und Wirtschaftsmathematik \\
E 105-07 \\
Wiedner Hauptstra{\ss}e 8-10 \\
1040 Wien, Austria}
\email{elia.bisi@tuwien.ac.at}

\author[Y.~Liao]{Yuchen Liao}
\address{Yuchen Liao, Mathematics Institute,
University of Warwick,
Coventry CV4 7AL,
UK}
\email{yuchen.liao@warwick.ac.uk}

\author[A.~Saenz]{Axel Saenz}
\address{Axel Saenz, Mathematics Department,
Oregon State University,
Corvallis,
OR 97331,
USA}
\email{saenzroa@oregonstate.edu}

\author[N.~Zygouras]{Nikos Zygouras}
\address{Nikos Zygouras, Mathematics Institute,
University of Warwick,
Coventry CV4 7AL,
UK}
\email{n.zygouras@warwick.ac.uk}

\begin{abstract}
We consider a discrete-time TASEP, where each particle jumps according to Bernoulli random variables with particle-dependent and time-inhomogeneous parameters.
We use the combinatorics of the Robinson--Schensted--Knuth correspondence and certain intertwining relations to express the transition kernel of this interacting particle system in terms of ensembles of weighted, non-intersecting lattice paths and, consequently, as a marginal of a determinantal point process.
We next express the joint distribution of the particle positions as a Fredholm determinant, whose correlation kernel is given in terms of a boundary-value problem for a discrete heat equation.
The solution to such a problem finally leads us to a representation of the correlation kernel in terms of random walk hitting probabilities, generalizing the formulation of 
Matetski, Quastel and Remenik (Acta Math., 2021) to the case of both particle- and time-inhomogeneous rates.
The solution to the boundary value problem in the fully inhomogeneous case appears with a finer structure than in the homogeneous case.
\end{abstract}

\maketitle
\tableofcontents

\section{Introduction}
\label{sec:intro}

\subsection{Background and literature}
\label{subsec:background}

The Kardar--Parisi--Zhang (KPZ) universality class is a large class of stochastic systems with highly correlated components that exhibit a similar statistical asymptotic behavior under space-time rescaling.
They include $(1+1)$-dimensional random growth models, interacting particle systems, eigenvalues of random matrices, and stochastic partial differential equations.
These models can be characterized by means of a space-time `height function', which typically features random non-Gaussian fluctuations depending on the initial height profile.
For certain specific initial configurations (e.g.\ `step' and `flat'), the one-point distributions are given by the Tracy--Widom laws (first introduced in the random matrix literature~\cite{tracyWidom94, tracyWidom96}) and the multi-point distributions are given by Airy processes.
Such precise asymptotics have been obtained so far only for a few integrable models, whose rich algebraic structure lead to exact formulas that are suitable for asymptotic analysis; see e.g.~\cite{borodinGorin16, zygouras22}.
Among these integrable models, the most accessible ones can be described in terms of determinantal measures: popular examples are the corner growth model, the polynuclear growth model, last passage percolation, and various types of exclusion processes.

It is conjectured that all KPZ models, under the so-called \emph{1:2:3-scaling}, converge to a universal, scale-invariant Markov process, known as the \emph{KPZ fixed point}.
Such a limiting process was constructed in~\cite{matetskiQuastelRemenik21} for the continuous-time Totally Asymmetric Simple Exclusion Process ($\tasep$), a prototypical interacting particle system on the integer line.
The transition probabilities of $\tasep$ were first shown to admit determinantal formulas in~\cite{schutz97}, using the coordinate Bethe ansatz.
Based on these formulas, \cite{sasamoto05, borodinFerrariPrahoferSasamoto07} showed that the $\tasep$ evolution is encoded by a determinantal point process; consequently, for arbitrary initial configurations, the multi-point distribution of $\tasep$ particles was given as a Fredholm determinant, whose kernel is implicitly characterized by a biorthogonalization problem.
This problem was solved in~\cite{borodinFerrariPrahoferSasamoto07} in the case of flat (2-periodic) initial configuration for the particles.
The solution to the problem of biorthogonalization for general initial configuration was given in~\cite{matetskiQuastelRemenik21}, where the kernel was expressed in terms of a functional of a random walk and its hitting times to a curve encoding the (arbitrary) initial configuration.
The KPZ fixed point was then constructed by taking a Donsker type scaling limit, under which random walk and associated hitting times turn into Brownian motion and corresponding  hitting times.

Since the seminal contribution of~\cite{matetskiQuastelRemenik21}, numerous works have considered other, not only determinantal, KPZ models with general initial configurations, obtaining similar Fredholm determinant formulas and, in certain cases, also convergence to the KPZ fixed point.
Here is a non-exhaustive list of such results.
A system of one-sided reflected Brownian motions was studied in~\cite{nicaQuastelRemenik20}.
Two variations of discrete-time $\tasep$ with geometric and Bernoulli jumps were considered in~\cite{arai20}.
Convergence to the KPZ fixed point was proved in~\cite{QuastelSarkar23} {in the case of finite range asymmetric exclusion processes and the KPZ equation, for certain classes of initial conditions}.
In~\cite{matetskiRemenik22} it was shown that the method of~\cite{matetskiQuastelRemenik21} can cover a general class of models whose multipoint distributions possess a Sch\"utz type determinantal formula.

In the present work, we consider a discrete-time $\tasep$ with inhomogeneous jump probabilities.
We provide an explicit, step-by-step route from the very definition of the model to a Fredholm determinant representation of the joint distribution of the particle positions in terms of random walk hitting times.
This is, to the best of our knowledge, the first such formulation for an interacting particle system with both particle- and time-inhomogeneous rates.
As discussed above, a result of this type, for the continuous-time and homogeneous-rate $\tasep$, was first obtained in~\cite{matetskiQuastelRemenik21}.
However, our approach differs from that of~\cite{matetskiQuastelRemenik21}:
Firstly, our starting point is not a Sch\"utz type formula, but rather the combinatorial structure of the integrable model.
Moreover, instead of solving the biorthogonalization problem, we map to ensembles of non-intersecting paths.
We work directly with the corresponding determinantal processes, exploiting some special features that emerge through mapping to path ensembles and from the expression of the path weights via local operators.
We hope this perspective can shed additional light on the structure of the KPZ fixed point formulas and may also be useful in different settings, in particular for other particle systems that can be characterized by intertwining relations.
In the next subsection we present our result and discuss our approach in detail.

\subsection{Our result and approach}
\label{subsec:ourContribution}

In all variations of $\tasep$, particles occupy sites of $\Z$ and, according to a stochastic mechanism, perform jumps in the same direction (to the right, by convention).
The interaction between particles consists in the exclusion rule: no two particles may occupy the same position at any given time.
Therefore, if a particle attempts to jump to an occupied site, the jump is suppressed.
There are several possible stochastic mechanisms inducing the dynamics.
In this article, as a working example, we consider a version of $\tasep$ where:
\begin{enumerate}
\item there is a finite number {$N\geq 2$} of particles;
\item the dynamics evolves in discrete time;
\item jump sizes are given by independent Bernoulli random variables;
\item the expected jump size may depend both on the particle ({\it particle-inhomogeneous rates}) and on time
 ({\it time-inhomogenous rates});
\item particle positions are updated sequentially from right to left.
\end{enumerate}
Let us define the version of $\tasep$ we are concerned with.
For $k=1,\dots,N$, denote by $Y_k(t) \in \mathbb{Z}$ the position of the $k$-th particle from the right at time $t \in \mathbb{Z}_{\geq 0}$.
Therefore, the configuration encoding the positions of the particles, at time $t$, will be
\[
Y(t) = (Y_1(t) > Y_2(t) > \cdots > Y_N(t) ),
\]
with {arbitrary} initial configuration 
\[
Y(0) = \bm{y} = (y_1 > y_2 > \cdots > y_N).
\]

Let $p_t$, $t\geq 1$, and $q_k$, $1\leq k\leq N$, be positive parameters.
The random dynamics is then given by sequential updates from right to left, 
i.e.\ from the particle labeled $1$ to the particle labeled $N$, as follows:
Assume we have already updated the position of the $(k-1)$-th particle at time $t$ (with $k\geq 2$).
Then, the $k$-th particle updates its position at time $t$ by jumping one step to the right with probability $p_tq_k/(1+p_tq_k)$, assuming that the neighboring site on its right 
is not occupied by the $(k-1)$-th particle; otherwise, the particle remains in its current position.
We call this model \emph{discrete-time (Bernoulli) $\tasep$} and abbreviate it as $\dtasep$.

To state our main result, we first define the following kernels of operators from $\Z$ to $\Z$.
For $1\leq j\leq k$ and $0\leq r\leq t$, let
\begin{align}
\mathcal{S}_{[j,k],(r,t]}(x,y)&:= \oint_{\Gamma_0}\frac{\mathrm{d}z}{2\pi\mathrm{i}z} \frac{\prod_{\ell=j}^{k}(q_\ell-z)\cdot \prod_{\ell=r+1}^{t}(1+p_\ell z)}{z^{x-y+k-j+1}}, \label{eq:Sintegral}\\ 
\bar{\mathcal{S}}_{[j,k],(r,t]}(x,y)&:= -\prod_{\ell=j}^{k}(q_\ell-1)\cdot\oint_{\Gamma_{\bm{q}}}\frac{\mathrm{d}z}{2\pi\mathrm{i}z} \frac{z^{y-x+k-j+1}}{\prod_{\ell=j}^{k}(q_\ell-z)\cdot \prod_{\ell=r+1}^{t}(1+p_\ell z)}, \label{eq:Sbarintegral}
\end{align}
where $\Gamma_0$ and $\Gamma_{\bm{q}}$ are simple closed contours {with counterclockwise orientation}, enclosing $0$ and $\{q_i\}_{i=1}^N$ as the only poles, respectively.
{The kernels $\mathcal{S}_{[j,k],(r,t]}$ are compositions of some random walk transition kernels and $\bar{\mathcal{S}}_{[j,k],(r,t]}$ are dual versions of $\mathcal{S}_{[j,k],(r,t]}$ with contribution coming from the poles $\{q_i\}_{i=1}^N$ instead of $0$; see Proposition \ref{prop:S-Sbar_localOp} for the precise statement.
The random walk/path interpretations of the kernels will be explained through \S2-4.}

Let $S$ be a geometric random walk (as defined in~\eqref{eq:transitionRW}).
{Let
\[
\mathrm{epi}(\bm{y}) := \{ (i,x)\colon {0\leq i\leq N-1}, \, x> y_{i+1} \}
\]
be the (discrete) strict epigraph of the (discrete) curve $(i,y_{i+1})_{0\leq i\leq N-1}$.
For $n\leq N$, let $\tau$ be the first time $\leq n$ at which the random walk $S$ hits $\mathrm{epi}(y)$}:
\begin{equation}
\tau:=\min\{m\in \{0,\dots,n\}\colon S_m>y_{m+1}\}.
\end{equation}
The kernel encoding the initial configuration {$\bm{y}=(y_1>\dots>y_N)$} is then expressed in terms of $\bar{\mathcal{S}}$ and $\tau$ as the expectation
\begin{equation}\label{eq:SbarEpi}
\bar{\mathcal{S}}_{[1,n],(r,t]}^{\mathrm{epi}(\bm{y})}(x,y):= 
\frac{\mathbb{E}_{S_0=x}[\bar{\mathcal{S}}_{[\tau+1,n],(r,t]}(S_\tau,y)\1_{\tau<n}]}{\prod_{\ell=1}^{n}(q_\ell-1)}.
\end{equation}
{Let also} $Q_i(x,y):= q_i^{y-x}\1_{\{y< x\}}$ for $1\leq i\leq N$ and $x,y\in \mathbb{Z}$ and $Q_{(m,n]}:=Q_{m+1}\circ\cdots\circ Q_n$ for $n>m$.

{Finally, for two operators $A$ and $B$ with kernels $A(x,y)$ and $B(x,y)$, $x,y\in\Z$, we define the composition operator $A\circ B$ through the kernel $(A\circ B) (x,z) :=\sum_{y\in \Z} A(x,y) B(y,z)$.}
With these notations, our main result is the following.
	
\begin{theorem}[Multipoint distributions of $\dtasep$ with particle- and time-inhomogeneous rates]\label{thm:corker_hitting}
Let $(Y(t))_{t\geq 0}=(Y_1(t)>\dots>Y_N(t))_{t\geq 0}$ be the locations of $N$ particles evolving according to the $\dtasep$ dynamics with parameters $\{p_t\}_{t\geq 1}$ and $\{q_k\}_{k=1}^N$ and initial configuration $Y(0)=\bm{y}$.
Assume that $q_kp_t<1$ for all $k,t$ and $q_k>1$ for all $k$.
Then, the joint distribution of particle locations at time $t$ is given by the Fredholm determinant
\begin{equation}\label{eq:multipoint}
\mathbb{P}\left(\bigcap_{i=1}^m \{Y_{k_i}(t)\geq s_i\} \;\middle|\; {Y(0)=\bm{y}} \right) =\det(I-\chi_s K\chi_s)_{\ell^2(\{k_1,\dots,k_m\}\times \mathbb{Z})},
\end{equation}
for any $m\in \mathbb{N}$, $1\leq k_1<k_2<\cdots<k_m\leq N$ and	$(s_1,\dots,s_m)\in {\mathbb{R}^m}$, where $\chi_s(k_i,x):= \1_{x<s_i}$ and $K$ is the kernel
\begin{equation}
\label{eq:corker_hitting}
K(m,x;n,x')=- Q_{(m,n]}{(x,x')}\1_{n>m} + \mathcal{S}_{[1,m],(0,t]}\circ 	\bar{\mathcal{S}}_{[1,n],(0,t]}^{\mathrm{epi}(\bm{y})} (x,x').
\end{equation}
\end{theorem}

{We remark that the condition $q_k p_t<1$ for all $k,t$ is equivalent to assuming that, for all $k,t$, the $k$-th particle attempts its $t$-th jump with probability $(q_k p_t)/(1+q_k p_t)<\frac{1}{2}$.
The case when all the jump probabilities are greater than $1/2$ can be analyzed using particle-hole duality (see for example~\cite{ferrari04}).
Our theorem does not cover the case where some of the jump rates are $>\frac{1}{2}$ and others are $<\frac{1}{2}$.
This seems to be a common restriction, appearing for example also in \cite[Assumption 1.1]{matetskiRemenik22}.}

The assumption $q_k>1$ is innocuous, as we will explain in Remark~\ref{rem:qless1}.

\vskip 1mm

The original work~\cite{matetskiQuastelRemenik21} dealt with homogeneous rates, while~\cite{arai20} considered time-inhomogeneous rates only.
In their recent work~\cite{matetskiRemenik22}, Matetski and Remenik expressed interest in the case of particle-inhomogeneous rates, considering it ``meaningful from a physical point of view''.
However, they only considered this general case in the preliminary part of their analysis, and did not solve the corresponding biorthogonalization problem\footnote{{In a recently appeared preprint~\cite{MatetskiRemenik23}, they have extended} the explicit biorthogonalization scheme developed in~\cite{matetskiRemenik22} to an inhomogeneous setting that allows to study, for instance, a discrete-time $\tasep$ where some particles update sequentially and some update in parallel.}.

Let us mention that shape functions and hydrodynamic limits of inhomogeneous $\tasep$ and corner growth model have been studied since the 1990s; see~\cite{SeppalainenKrug99, Emrah16, EmrahJanjigianSeppalainen21}.
Furthermore, multipoint formulas for $\tasep$ with particle- and time-inhomogeneous rates have been obtained e.g.\ in~\cite{BorodinPeche08, KnizelPetrovSaenz19, johanssonRahman22, IwaoMotegiScrimshaw22}.
However, Theorem~\ref{thm:corker_hitting} expresses, for the first time in the case of a particle- and time-inhomogeneous $\tasep$, the joint law of particle locations in terms of random walk hitting times.
As we explain later in the introduction, we obtain these formulas in a quite different way from~\cite{matetskiQuastelRemenik21, matetskiRemenik22}, overcoming the technical difficulties that such a generalization presents.

\vskip 1mm

The route we follow to prove Theorem \ref{thm:corker_hitting} is also new, in various aspects, compared to the methods used so far, as we now discuss.
The remarkable idea of expressing the multipoint law in terms of a random walk hitting problem is due to~\cite{matetskiQuastelRemenik21}.
Before that, the standard representation was in terms of a Fredholm determinant involving two families of biorthogonal functions $\{\Psi_k^n(\cdot), \Phi_k^n(\cdot) \}_{1\leq k \leq n }$, where $\{\Psi_k^n(\cdot)\}$ are typically given explicitly, but $\{\Phi_k^n(\cdot)\}$ are to be determined by solving the biorthogonalisation problem with respect to $\{\Psi_k^n(\cdot)\}$.
This formulation was first established in~\cite{sasamoto05, borodinFerrariPrahoferSasamoto07}, where the case of flat (2-periodic) initial configuration for continuous-time $\tasep$ (which corresponds to a biorthogonalisation problem for shifted Charlier polynomials) was solved. 

Starting from the determinantal formula of~\cite{borodinFerrariPrahoferSasamoto07} for $\tasep$ with general initial configuration, \cite{matetskiQuastelRemenik21} were able to solve the biorthogonalisation problem in terms of a hitting time expectation of the form~\eqref{eq:SbarEpi}.
At the core of this lay an expression of the family $\{\Phi_k^n(\cdot)\}$ in terms of a terminal-boundary value problem for a discrete heat equation.
The fact that the family $\{\Phi_k^n(\cdot)\}$ obtained by this method solves the biorthogonal problem with respect to $\{\Psi_k^n(\cdot)\}$ was achieved via a direct check.
As explained in~\cite{matetskiQuastelRemenik21} (see also \cite{Remenik22}), the intuition that led to such a guess was based on two points.
Firstly, thanks to the ``skew time reversibility'' of $\tasep$, the one-point distribution of $\tasep$ with homogeneous parameters and a general initial profile is equal to the multi-point distribution of $\tasep$ starting from the so-called \emph{step initial configuration} (which represented the first solvable example of a particle system~\cite{johansson00}).
Secondly, the multi-point distribution of $\tasep$ with the step initial condition has been known since~\cite{prahoferSpohn02} to possess the Fredholm determinant expression  
\begin{align*}
\mathbb{P}\Bigg(\bigcap_{i=1}^m \{Y_{k_i}(t)\geq s_i\}\Bigg)
=\det\big(I-K_t^{(k_m)}( I-Q^{k_1-k_m}\chi'_{s_1} Q^{k_2-k_1}\chi'_{s_2}\cdots Q^{k_m-k_{m-1}}\chi'_{s_m}) \big)_{\ell^2(\mathbb{Z})},
\end{align*}
where $Q^n$ is the $n$-step transition kernel of a homogeneous, geometric random walk, $K_t^{(n)}$ is the kernel of the one-point distribution of the model and $\chi'_{s}(k_j,x):=\ind_{x>s_j}$.
The expression $Q^{k_1-k_m}\chi'_{s_1} Q^{k_2-k_1}\chi'_{s_2}\cdots Q^{k_m-k_{m-1}}\chi'_{s_m}$ can then be interpreted as the probability that a homogeneous, geometric random walk with transition kernel $Q$ lies above $s_1,\dots,s_m$ at times $k_1,\dots,k_m$.
We remark that, due to the aforementioned skew time reversibility of $\tasep$, the levels $s_1,\dots,s_m$ are related to the initial (rather than final) positions of the particles.

\vskip 1mm

Our approach to Theorem \ref{thm:corker_hitting} differs from the above, even though our guiding principle has been the already mentioned terminal-boundary value problem and a desire to better understand its foundations.
We do not start from the determinantal formulas; instead, we work with the combinatorial foundations of discrete $\tasep$ and its links, via intertwinings and Markov functions, to determinantal point processes.
We first compute the transition kernel of $\dtasep$ by using the \emph{column insertion, dual} version of the Robinson--Schensted--Knuth ($\rsk$) correspondence, which we abbreviate as $\drsk$.
As it turns out, this combinatorial algorithm, viewed from a dynamical standpoint, encodes the $\dtasep$ dynamics as a projection.
Furthermore, the transition kernel of $\dtasep$ intertwines with the transition kernel of the evolution of the shape of the tableaux generated by the $\drsk$ dynamics and, thus, can be written as the latter kernel conjugated by an `intertwining' kernel; see~\eqref{eq:kerQ_LambdaInverse}.
The general link between $\tasep$ dynamics and $\rsk$ correspondences is of course well known; see for example~\cite{diekerWarren08} and references therein.
In particular, our approach can be regarded as a time-inhomogeneous generalization of~\cite{diekerWarren08} 
(see Case B: `Bernoulli jumps with blocking'). 

Next, we interpret all the kernels appearing in our representation of the $\dtasep$ transition kernel in terms of weights of ensembles of non-intersecting lattice paths; see~\eqref{eq:pathQ2}.
For our later goals, it is important to remark that these weights can be expressed in terms of one-step (local) transition operators.
 
The Lindstr\"om--Gessel--Viennot theorem leads, then, to a determinantal formulation for all these kernels,
thus allowing us to view the transition distribution of $\dtasep$ as a marginal of a determinantal point process; see~\eqref{eq:Q_marginalDPP}.
Using standard methods in the theory of determinantal point processes (as in~\cite{borodinFerrariPrahoferSasamoto07, johansson03}), we express the fixed-time joint distribution of $\dtasep$ particles as a Fredholm determinant; see Proposition~\ref{prop:modifiedCorKer}.
The correlation kernel {of the Fredholm determinant} involves the local operators encoding the transition weights of the path ensemble as well as the inverse of a matrix $M$; see~\eqref{eq:modifiedCorKer}-\eqref{eq:Mmatrix}.
The geometric picture that we obtain through the non-intersecting path ensembles leads us to conclude that $M$ is upper triangular and, therefore, explicitly invertible.
This crucial aspect leads to the boundary-terminal value problem (Proposition~\ref{prop:modifiedBVP}), which we next solve to arrive at the random walk hitting formula~\eqref{eq:SbarEpi}.
This task turns out to be more challenging than in~\cite{matetskiQuastelRemenik21}, since, in the case of particle-inhomogeneous rates, the solution is not spanned by polynomials; see Remark~\ref{rem:non-polynomiality}.
In particular, we develop a very careful double induction argument (see the proof of Proposition~\ref{prop:generalG}) {that involves} some subtle cancellations of inclusion-exclusion type that take place in the formulas.

\vskip 1mm

As outlined in \S\ref{subsec:background}, the KPZ fixed point was constructed in~\cite{matetskiQuastelRemenik21} as the 1:2:3 scaling limit, at large time and length scales, of the homogeneous continuous-time $\tasep$.
The present work paves the way to construct analogous processes from particle systems with variable, fast/slow, rates.
We leave this task for future work.
 
\subsection{Outline of the {article}} 
{In~\S\ref{sec:dynamics} we start by presenting some combinatorial objects and, in particular, the dual, column $\rsk$ algorithm ($\drsk$) and its link with discrete-time $\tasep$ ($\dtasep$); we also obtain an expression for the transition probability kernel of $\dtasep$ via an intertwining relation.
In~\S\ref{sec:path-constr}
we re-express the transition kernel of $\dtasep$ in terms of weights of ensembles of non-intersecting 
paths and determinantal point processes, thus arriving at an initial Fredholm determinant formula.
In~\S\ref{sec:RWformulas} we prove our main result, Theorem~\ref{thm:corker_hitting}, first formulating a terminal-boundary value problem and then solving it, to arrive at the hitting time representation for the correlation kernel of $\dtasep$ with inhomogeneous rates.
To solve the terminal-boundary value problem, we first use path representations for certain subsets of $\mathbb{Z}$, and then extend the solution to the whole space via a subtle double induction argument (see Proposition~\ref{prop:generalG}).}

\section{$\tasep$ dynamics and combinatorics}
\label{sec:dynamics}

{In this section we present the main combinatorial tools that we need for the analysis of the $\dtasep$ dynamics.
In~\S\ref{subsec:partitions}, we introduce a few standard algebraic combinatorial objects.
In~\S\ref{subsec:RSK}, we describe the dual, column $\rsk$ ($\drsk$) correspondence and its main properties.
In~\S\ref{subsec:RSK-TASEP}, we discuss the link between $\dtasep$ and $\drsk$.
Finally, in~\S\ref{subsec:shape} we establish certain intertwining relations and deduce a preliminary expression for the $\dtasep$ transition kernel.}

\subsection{Partitions, tableaux, and Schur polynomials}
\label{subsec:partitions}

A \emph{partition} $\lambda$ of $n\geq 0$ is a sequence $\lambda = (\lambda_1\geq \lambda_2 \geq \dots)$ of weakly decreasing non-negative integers, called \emph{parts} of $\lambda$, such that $\abs{\lambda}:=\sum_{i\geq 1} \lambda_i =n$.
If $\lambda$ is a partition of $n$, we write $\lambda \vdash n$ and refer to $n$ as the \emph{size} of $\lambda$.
We will also say that $\lambda$ is a partition without referring to its size.
We will denote by $\emptyset$ the only partition of $0$.
Any partition of $n$ can be graphically represented as a \emph{Young diagram} of size $n$, i.e.\ a collection of $n$ cells arranged in left-justified rows, with $\lambda_i$ cells in the $i$-th row.
Every such a cell can be identified with a pair $(i,j) \in \Z_{\geq 1}^2$ with row index $i$ and column index $j$; thus, we may alternatively write $\lambda$ as the set of such pairs:
\begin{equation}
\label{eq:partitionSet}
\lambda = \{(i,j) \colon i\geq 1, \, 1\leq j\leq \lambda_i \} .
\end{equation}
The \emph{conjugate partition} of $\lambda$, which we denote by $\lambda^{\top}$, is defined by setting $\lambda_i^{\top}$ to be the number of $k\geq 1$ such that $\lambda_k\geq i$; conjugating a partition corresponds to transposing the associated Young diagram.
The \emph{length} of $\lambda$ is the number of its non-zero parts; since it clearly coincides with the first part of the conjugate partition $\lambda^{\top}$, we denote it by $\lambda^{\top}_1$.

{We define the (discrete) Weyl chamber as
\begin{equation}
\label{eq:weylChamber}
\sfW_n := \left\{\bm{y} = (y_1,\dots,y_n)\in \Z^n\colon y_1\geq y_2\geq \cdots \geq y_n \right\}.
\end{equation}
Throughout this section, elements of $\sfW_n$ will be implicitly taken to have non-negative components and, thus, to be integer partitions of length $\leq n$.
In later sections, we will drop this assumption and consider elements of $\sfW_n$ with possibly negative components.}

For any two partitions $\lambda$ and $\mu$, we write $\mu\subseteq \lambda$ if $\mu_i\leq \lambda_i$ for all $i\geq 1$, or equivalently if $\mu$ is a subset of $\lambda$, viewing the partitions as sets as in~\eqref{eq:partitionSet}.
A \emph{skew Young diagram} $\lambda/\mu$ is the set difference between two partitions $\lambda$ and $\mu$ such that $\mu\subseteq \lambda$.
The \emph{size} of $\lambda/\mu$, denoted by $\abs{\lambda/\mu}$, is the number of its cells, which equals $\abs{\lambda} - \abs{\mu}$.
If $\lambda/\mu$ has at most one cell per column, we call it \emph{horizontal strip}; if $\lambda/\mu$ has at most one cell per row, we call it \emph{vertical strip}.
We say that two partitions $\mu$ and $\lambda$ \emph{interlace}, and write $\mu\prec \lambda$, if $\lambda /\mu$ is a horizontal strip, or equivalently if $\lambda_i \geq \mu_i \geq \lambda_{i+1}$ for all $i\geq 1$.

A \emph{Young tableau} $\sfT=\{\sfT_{i,j} \colon (i,j)\in \lambda\}$ of \emph{shape} $\lambda$ is a filling of a Young diagram $\lambda$ {with elements of an alphabet $A\subseteq \Z_{\geq 1}$}.
We write $\sh(\sfT)$ for the shape of $\sfT$.
The \emph{transpose} of $\sfT$, denoted by $\sfT^{\top}$, is the tableau of shape $\lambda^{\top}$ that is obtained from $\sfT$ by exchanging its rows with its columns.
A Young tableau $\sfT$ is called \emph{column-strict} if its entries weakly increase along rows and strictly increase down columns.
Every column-strict Young tableau $\sfT$ of shape $\lambda$ can be alternatively represented as a sequence of interlacing partitions:
\begin{equation}
\label{eq:tableauInterlacing}
\sfT \equiv \big({\emptyset =:} \lambda^{(0)} \prec \lambda^{(1)} \prec \lambda^{(2)} \prec \cdots\big) ,
\end{equation}
where each $\lambda^{(k)}$ is the shape of the Young tableau obtained from $\sfT$ by removing all the cells containing numbers $>k$.
By the column-strict property of $\sfT$, we have $\lambda^{(k)} \in \sfW_k$ for all $k$, and the partitions interlace.
Clearly, $\lambda^{(k)}$ coincides with $\lambda$ for $k$ large enough; therefore, one can think of the sequence as finite, by stopping it at any $\lambda^{(k)}$ such that $\lambda^{(k)}=\lambda$.
See Figure~\ref{fig:tableaux} for an example of a column-strict Young tableau.
Similarly, a Young tableau $\sfT$ is called \emph{row-strict} if its rows are strictly increasing and its columns are weakly increasing, or equivalently if $\sfT^{\top}$ is column-strict.

\begin{figure}
\[
\young(111233,2245,44)
\quad\equiv\quad
\left( {\emptyset \prec} (3) \prec (4,2) \prec (6,2) \prec (6,3,2) \prec (6,4,2) \right)
\]
\caption{On the left-hand side, an example of a column-strict Young tableau of shape $(6,4,2)$ and left edge $(3,2)$.
On the right-hand side, its corresponding sequence of interlacing partitions.
}
\label{fig:tableaux}
\end{figure}

We define the \emph{left edge}\footnote{This terminology comes from the triangular arrangement of $\big(\lambda^{(0)} \prec \lambda^{(1)} \prec \lambda^{(2)} \prec \cdots\big)$ as a Gelfand--Tsetlin pattern.
{Later on, we will also consider the left edge of a triangular point process, as defined in~\S\ref{subsec:DPP} and visualized in Figure~\ref{fig:paths}.}} of a column-strict tableau $\sfT = \big(\lambda^{(0)} \prec \lambda^{(1)} \prec \lambda^{(2)} \prec \cdots\big)$ to be the partition
\begin{equation}
\label{eq:leftEdge}
\ledge(\sfT) := \big(\lambda^{(1)}_1\geq \lambda^{(2)}_2\geq \cdots\big) .
\end{equation}
Notice that, as all entries of the $k$-th row of $\sfT$ are $\geq k$ by the column-strict property, $\lambda^{(k)}_k$ is simply the number of $k$'s in the $k$-th row of $\sfT$.
See again Figure~\ref{fig:tableaux} for an example.

Finally, we give two equivalent, combinatorial definitions of Schur polynomials.
Let $n\geq 1$ and $\lambda \in \sfW_n$.
The \emph{Schur polynomial} in $n$ variables of shape $\lambda$ is given by
\begin{equation}
\label{eq:schur}
\schur_{\lambda}(x_1,\dots,x_n) 
:= \sum_{\substack{\lambda^{(0)} \prec \cdots \prec \lambda^{(n)} \colon \\ \lambda^{(n)}=\lambda}}
\prod_{k=1}^n x_k^{\lvert\lambda^{(k)}/\lambda^{(k-1)}\rvert}
= \sum_{\sfT\colon \sh(\sfT)=\lambda} \,
\prod_{k=1}^n x_k^{\lvert (i,j)\colon \sfT_{i,j}=k\rvert} .
\end{equation}
The first sum is taken over any sequence $\big(\lambda^{(0)} \prec \cdots \prec \lambda^{(n)}\big)$ of interlacing partitions such that $\lambda^{(k)} \in \sfW_k$ for all $k$ and $\lambda^{(n)}=\lambda$.
The second sum is taken over any column-strict Young tableau $\sfT$ of shape $\lambda$ in the alphabet $\{1,\dots,n\}$.
It is also convenient to define $\schur_{\lambda}(x_1,\dots,x_n) :=0$ whenever the length of $\lambda$ exceeds $n$.

\subsection{Dual column $\rsk$}
\label{subsec:RSK}

As we will see, the $\tasep$ dynamics we are concerned with are encoded by a certain variation of the \emph{Robinson--Schensted--Knuth correspondence} ($\rsk$), a celebrated combinatorial algorithm~\cite{knuth70, fulton97, stanley99}.

All the $\rsk$ variations map a matrix (input) to a pair of Young tableaux (output).
They can be differentiated based on two key factors:
\begin{itemize}
\item the input may be a non-negative integer matrix or a $\{0,1\}$-matrix (in the latter case, one usually talks about \emph{dual} $\rsk$);
\item the algorithm may be based on the so-called \emph{row insertion} or \emph{column insertion}.
\end{itemize}
According to these factors, one obtains four variations of $\rsk$: row $\rsk$, column $\rsk$, dual row $\rsk$, and dual column $\rsk$.
For our purposes we need the latter variation, dual column $\rsk$, which we abbreviate as $\drsk$.
We introduce it here and refer to~\cite[A.4.3]{fulton97} for further details.

It is convenient to first define, for $j\geq 1$, a mapping
\begin{equation}
\label{eq:insertionIntoColumn}
\mathcal{I}_j \colon (\sfT,x) \mapsto (\sfT',y) ,
\end{equation}
which should be interpreted as the insertion of a number $x$ into the $j$-th column of a tableau $\sfT$.
Here:
\begin{enumerate}
\item\label{item:inputTableau} $\sfT$ is an input column-strict Young tableau $\sfT$ of shape $\lambda$, with $\lambda_1 \geq j-1$;
\item\label{item:inputInteger} $x$ is an input positive integer such that, if $j>1$ and $\lambda^{\top}_j = \lambda^{\top}_{j-1}$, then $x\leq \max_{i} \sfT_{i,j}$;
\item\label{item:outputTableau} $\sfT'$ is an output column-strict Young tableau;
\item\label{item:outputInteger} $y$ is either an output positive integer or a `stop symbol' $\bigtimes$.
\end{enumerate}
The mapping works as follows. For fixed $j\geq 1$, 
if all entries $\sfT_{i,j}$, $1\leq i\leq \lambda_j^{\top}$, of the $j$-th column of $\sfT$ are $<x$ (so that, by~\ref{item:inputInteger}, we have $\lambda^{\top}_j < \lambda^{\top}_{j-1}$), then a new cell $(\lambda^{\top}_j+1,j)$ containing $x$ is added to the column, thus yielding a new column-strict tableau $\sfT'$; the outputs are then $(\sfT',\bigtimes)$.
Otherwise, let $i$ be the smallest integer such that $x \leq \sfT_{i,j} =: y$; define $\sfT'$ to be the same tableau as $\sfT$ except for the $(i,j)$-entry $\sfT'_{i,j}:=x$; the outputs are then $(\sfT',y)$.

We now define the \emph{column insertion} algorithm as a composition of several mappings of the form~\eqref{eq:insertionIntoColumn}.
Consider the sequence
\[
(\sfT,x)=:(\sfT^{(0)},y^{(0)}) \xmapsto{\mathcal{I}_1} (\sfT^{(1)},y^{(1)}) \xmapsto{\mathcal{I}_2} \cdots \xmapsto{\mathcal{I}_k} (\sfT^{(k)},y^{(k)})=: (\sfT',\bigtimes) \, ,
\]
where $k$ is the smallest integer such that $y^{(k)}=\bigtimes$.
Notice that, by construction, every $y^{(j-1)}$, $1\leq j\leq k$, \emph{can} be inserted into the $j$-th column of $T^{(j-1)}$, in the sense that hypothesis~\ref{item:inputInteger} above is satisfied.
We then set $\sfT'$ to be the outcome of the column insertion of $x$ into the tableau $\sfT$.
Clearly, if $\sfT$ is of size $n$, then $\sfT'$ will be of size $n+1$.
See Figure~\ref{fig:columnInsertion} for a graphical representation of the column insertion algorithm.

\begin{figure}
\newcommand{\one}{\color{blue} 1}
\newcommand{\two}{\color{blue} 2}
\newcommand{\three}{\color{blue} 3}
\newcommand{\four}{\color{blue} 4}
\newcommand{\five}{\color{blue} 5}
\[
\left(\,\young(112,25,\three) \, , {\color{red} 3}\right)
\xmapsto{\mathcal{I}_1}
\left(\,\young(112,2\five,3) \, ,{\color{red} 3}\right)
\xmapsto{\mathcal{I}_2}
\left(\,\young(112,23,3)\, ,{\color{red} 5}\right)
\xmapsto{\mathcal{I}_3}
\left(\,\young(112,235,3)\, ,\bigtimes \right)
\]
\caption{Example of column insertion of an integer $x$ into a tableau $\sfT$.
We start with the pair $(\sfT,x)$, on the left-hand side, and apply the mappings $\mathcal{I}_1, \mathcal{I}_2, \dots$ until we get a pair of the form $(\sfT',\bigtimes )$.
The tableau $\sfT'$ is then the outcome of the column insertion.
At the $j$-th step, the red number is to be inserted into the $j$-th column: either it replaces the blue number (first two steps) or it is inserted in a new cell at the end of the column (third step).
In the former case, the blue number becomes the red one at the next step; in the latter case, a `stop symbol' $\bigtimes$ is returned and the procedure stops.}
\label{fig:columnInsertion}
\end{figure}

We now construct the $\drsk$ algorithm.
Given an input matrix $w=\{w_{i,j} \colon 1\leq i\leq n, \, 1\leq j\leq N \}$ with entries in $\{0,1\}$, we define a sequence
\begin{equation}
\label{eq:drskDynamics}
(\emptyset,\emptyset) =: (\sfP(0),\sfQ(0))
\mapsto
(\sfP(1),\sfQ(1))
\mapsto\dots\mapsto
(\sfP(n),\sfQ(n))=:(\sfP,\sfQ)
\end{equation}
of Young tableaux pairs starting from the pair of empty tableaux and ending at the $\drsk$ output pair $(\sfP,\sfQ)$ (for an example, see Figure~\ref{fig:drsk}).
Essentially, each $\sfP(i)$ is constructed by column inserting into $\sfP(i-1)$ the column indices $j$ that correspond to ones in the $i$-th row of $w$, whereas each $\sfQ(i)$ records the cells that are added in the construction of $\sfP(i)$.
More precisely, for all $i=1,\dots,n$, given $(\sfP(i-1),\sfQ(i-1))$, the next pair $(\sfP(i),\sfQ(i))$ is obtained as the last element of the sequence
\[
\begin{split}
(\sfP(i-1),\sfQ(i-1)) =: (\sfP(i,0),\sfQ(i,0))
&\mapsto
(\sfP(i,1),\sfQ(i,1)) \\
&\mapsto\dots\mapsto
(\sfP(i,N),\sfQ(i,N)) =: (\sfP(i),\sfQ(i)) ,
\end{split}
\]
where, for $j=1,\dots,N$:
\begin{itemize}
\item if $w_{i,j}=0$, then $\sfP(i,j)=\sfP(i,j-1)$ and $\sfQ(i,j)=\sfQ(i,j-1)$;
\item if $w_{i,j}=1$, then
\begin{itemize}
\item $\sfP(i,j)$ is the tableau obtained by column inserting $j$ into $\sfP(i,j-1)$, and
\item $\sfQ(i,j)$ is obtained from $\sfQ(i,j-1)$ by adding a cell, filled with $i$, at the same location where a cell was added in the column insertion of $j$ into $\sfP(i,j-1)$.
\end{itemize} 
\end{itemize}
By construction, for all $i,j$, $\sfP(i,j)$ and $\sfQ(i,j)$ are Young tableaux of the same shape.
Each $\sfP(i,j)$ is column-strict.
Moreover, it is not difficult to see that each $\sfQ(i,j)$ is row-strict.

\begin{figure}
\centering
\begin{minipage}{0.2\textwidth}
\centering
\[
w=
\begin{pmatrix}
1 &0 &1 &1 \\
0 &1 &1 &0 \\
1 &1 &0 &1 
\end{pmatrix}
\]
\end{minipage}
\begin{minipage}{0.8\textwidth}
\centering
\begin{align*}
\qquad\quad\emptyset=:\sfP(0) &&\mapsto &&\sfP(1)=\young(1,3,4) &&\mapsto &&\sfP(2)=\young(13,24,3) &&\mapsto &&\sfP(3)=\young(113,224,3,4) \\
\emptyset:=\sfQ(0) &&\mapsto &&\sfQ(1)=\young(1,1,1) &&\mapsto &&\sfQ(2)=\young(12,12,1) &&\mapsto &&\sfQ(3)=\young(123,123,1,3)
\end{align*}
\end{minipage}
\caption{An example of the $\drsk$ correspondence, constructed as in~\eqref{eq:drskDynamics}.
An input $\{0,1\}$-matrix $w$ yields a sequence of tableaux pairs, which terminates at the $\drsk$ output pair $(\sfP,\sfQ)=(\sfP(3),\sfQ(3))$.
}
\label{fig:drsk}
\end{figure}

In the next theorem we summarize the properties of this mapping that are useful for our purposes.
They are all either immediate from the construction or easy to prove, and can be visualized in the example of Figure~\ref{fig:drsk}.
We refer e.g.\ to~\cite[A.4.3]{fulton97} for a proof.
\begin{theorem}
\label{thm:drsk}
The dual column Robinson--Schensted--Knuth correspondence $\drsk\colon w \mapsto (\sfP,\sfQ)$ is a bijection between a matrix with entries in $\{0,1\}$ and a pair $(\sfP,\sfQ)$ of Young tableaux of the same shape such that $\sfP$ is column-strict and $\sfQ$ is row-strict.
If the input matrix is $n\times N$, then $\sfP$ is in the alphabet $\{1,\dots ,N\}$ and $\sfQ$ is in the alphabet $\{1,\dots ,n\}$, so one can identify
\[
\sfP = \big(\lambda^{(0)} \prec \cdots \prec \lambda^{(N)}\big)
\qquad\text{and}\qquad
\sfQ^{\top} = \big(\mu^{(0),\top} \prec \cdots \prec \mu^{(n),\top}\big) \, ,
\]
where $\lambda^{(k)}, \mu^{(k),\top} \in \sfW_k$ for all $k$.
Referring to the sequence of pairs~\eqref{eq:drskDynamics} that defines $\drsk$, we then have 
\begin{equation}
\label{eq:drskShapeDynamics}
\mu^{(i)}= \sh(\sfP(i)) \qquad
\text{for } i=1,\dots,n.
\end{equation}
Moreover, we have
\begin{align}
\label{eq:drskTypeCol}
\sum_{i=1}^n w_{ij} 
&= \big\lvert\lambda^{(j)}/\lambda^{(j-1)}\big\rvert
&&\text{for all } j=1,\dots,N , \\
\label{eq:drskTypeRow}
\sum_{j=1}^N w_{ij}
&=\big\lvert\mu^{(i)}/\mu^{(i-1)}\big\rvert
&&\text{for all } i=1,\dots,n .
\end{align}
\end{theorem}

\subsection{$\dtasep$ dynamics and $\drsk$}
\label{subsec:RSK-TASEP}

Let us now elaborate on the definition of $\dtasep$ given in the introduction and describe its relation to the dynamics of $\drsk$.
Recall that the $N$-particle $\dtasep$ is encoded by the discrete-time Markov chain $(Y(t))_{t\geq 0}$ of particle configurations $Y(t) = (Y_1(t) > Y_2(t) > \cdots > Y_N(t))$, where $Y_k(t)$ is the location of the $k$-th particle from the right.
We consider an arbitrary initial configuration $Y(0) = \bm{y} = (y_1 > y_2 > \cdots > y_N)$.
Let $\bm{p}=(p_t)_{t\geq 1}$ and $\bm{q}=(q_1,\dots,q_N)$ be positive parameters and let $W = \{ W_{t,k} \colon t \geq 1, \, 1\leq k\leq N\}$ be a collection of independent Bernoulli random variables with
\begin{equation}
\label{eq:bernoulliWeights}
\P\left( W_{t,k} = 0 \right)
= \frac{1}{1 + p_t q_k } ,
\qquad\qquad
\P\left( W_{t,k} = 1 \right)
= \frac{p_t q_k}{1 + p_t q_k}.
\end{equation}
The random dynamics is then given by sequential updates from right to left, i.e.\ from the particle labeled $1$ to the particle labeled $N$, driven by these random variables as follows:
\begin{equation}
\label{eq:dTASEP_recursion}
Y_k(t) : = \min \left\{Y_{k-1}(t)-1 , \,   Y_{k}(t-1) + W_{t,k} \right\},
\end{equation}
with the convention that $Y_0(t) = \infty$ for all $t \geq 0$.
These dynamics clearly preserve the ordering of the particles (exclusion rule).
We will abbreviate the $N$-particle $\dtasep$ with parameters $\bm{p}$ and $\bm{q}$ as $\dtasep(N;\bm{p},\bm{q})$.

The construction of $\drsk$ given in~\S\ref{subsec:RSK} (see in particular~\eqref{eq:drskDynamics} and Figure~\ref{fig:drsk}) is `dynamic': at the $t$-th step, the tableau pair $(\sfP(t-1),\sfQ(t-1))$ and the $t$-th row of the input matrix $w$ are used to generate a new tableau pair $(\sfP(t),\sfQ(t))$ through the column insertion algorithm.
We will now see how this dynamic procedure encodes the evolution of the $\dtasep$, if one interprets the input matrix entries as the Bernoulli random variables governing the particle jumps.

Notice that the collection $W = \{ W_{t,k} \colon t \geq 1, \, 1\leq k\leq N\}$ of Bernoulli random variables defined in~\eqref{eq:bernoulliWeights} can be seen as a (random) matrix with infinitely many rows.
For all $t\geq 0$, let $(\sfP(t),\sfQ(t))$ be the tableau pair obtained by applying $\drsk$ to the (random) matrix $\{ W_{i,j} \colon 1\leq i\leq t, \, 1\leq j\leq N\}$ consisting of the first $t$ rows of $W$.
As each $\sfP(t)$ is a column-strict tableau, we may write $\sfP(t) = \big(\lambda^{(0)}(t) \prec \lambda^{(1)}(t) \prec \cdots \prec \lambda^{(N)}(t)\big)$ as a sequence of interlacing partitions; see~\eqref{eq:tableauInterlacing}.

Recall from~\eqref{eq:leftEdge} that the left edge of $\sfP(t)$ is the partition $\ledge(\sfP(t)) = (\lambda^{(1)}_1 \geq \cdots \geq \lambda^{(N)}_N)$, where each $\lambda^{(k)}_k(t)$ is the number of $k$'s in the $k$-th row of $\sfP(t)$.
It is a consequence of the column insertion algorithm that the time evolution of $\ledge(\sfP(t))$ is autonomous from any additional information carried by $\sfP(t)$.
To see this, suppose that the first $t-1$ rows of $W$ have been inserted, yielding a $\sfP$-tableau $\sfP(t-1)$.
If $W_{t,1}=1$, then a $1$ is column inserted into $\sfP(t-1)$, thus yielding a new tableau $\sfP(t,1)$ (according to the notation of~\S\ref{subsec:RSK}) that contains one more $1$ in the first row.
As the subsequent insertion of any $k>1$ does not affect the cells containing $1$'s, we have $\lambda^{(1)}_1(t) = \lambda^{(1)}_1(t-1)+W_{t,1}$.
Suppose now that, at time $t$, for some $k\geq 2$, the numbers $<k$ have been sequentially inserted, thus yielding a tableau $\sfP(t,k-1)$.
Now, if $W_{t,k}=1$, then $\sfP(t,k)$ is generated by column inserting a $k$ into $\sfP(t,k-1)$: if there are more $(k-1)$'s in the $(k-1)$-th row than $k$'s in $k$-th row, then the `new' $k$ will end up in the $k$-th row; however, if there are as many $(k-1)$'s in the $(k-1)$-th row as there are $k$'s in the $k$-th row, then the `new' $k$ will end up in the $j$-th row, for some $j<k$.
Again, since the cells containing $k$ are not affected by subsequent insertions of larger numbers, we conclude that
\begin{equation}
\label{eq:leftEdgeRecursion}
\lambda^{(k)}_k(t) =
\begin{cases}
\lambda^{(k)}_k(t-1) + W_{t,k} &\text{if } \lambda^{(k-1)}_{k-1}(t) > \lambda^{(k)}_k(t-1) , \\
\lambda^{(k)}_k(t-1) &\text{if } \lambda^{(k-1)}_{k-1}(t) = \lambda^{(k)}_k(t-1) .
\end{cases}
\end{equation}
The latter formula is also valid for $k=1$, if we adopt the convention $\lambda^{(0)}_0(t)=\infty$ for all $t\geq 0$.
Notice that the update rules~\eqref{eq:leftEdgeRecursion} must be applied sequentially, from $k=1$ to $k=N$.
It is then straightforward to check that the $N$-tuple 
\[
\big(\lambda^{(k)}_k(t)-k\colon 1\leq k\leq N\big)
=\big(\lambda^{(1)}_1(t)-1 > \lambda^{(2)}_2(t)-2 > \cdots > \lambda^{(N)}_N(t)-N \big)
\]
satisfies the same recursion equations~\eqref{eq:dTASEP_recursion} that the $\dtasep$ satisfies.

{
\begin{remark}
Integer partitions coming from Young tableaux have of course nonnegative parts.
As a result, the transition kernels arising in the $\drsk$ dynamics that will be computed in the next subsection will be acting, in principle, on {elements of $\sfW_N$ with nonnegative components}.
On the other hand, $\dtasep$ particles may occupy any site of $\Z$.
However, this is not an issue: The kernels coming from $\drsk$ can be extended to {elements of $\sfW_N$ with components} of any sign, just by shifting all the parts by the same (integer) amount.
\end{remark}
}

\subsection{Transition probabilities for $\drsk$ and $\dtasep$}
\label{subsec:shape}

We now study the evolution of the $\sfP$- and $\sfQ$-tableaux under the $\drsk$ dynamics considered in~\S\ref{subsec:RSK-TASEP}.
This will yield useful formulas for the transition probabilities of $\dtasep(N;\bm{p},\bm{q})$, as defined in~\eqref{eq:bernoulliWeights}-\eqref{eq:dTASEP_recursion}.

By~\eqref{eq:bernoulliWeights}, the joint probability distribution of the Bernoulli weights up to time $t$ is
\begin{align*}
\P(W_{i,j}=w_{i,j}\colon 1\leq i\leq t, 1\leq j\leq N)
= \frac{1}{Z^{\bm{p},\bm{q}}_{(0,t]}}
\prod_{i=1}^t  p_i^{\sum_{j=1}^N w_{i,j}}
\prod_{j=1}^N q_j^{\sum_{i=1}^t w_{i,j}} ,
\end{align*}
where $\{w_{i,j}\colon 1\leq i\leq t, 1\leq j\leq N\}$ is any $\{0,1\}$-matrix and, for $0\leq r< t$,
\begin{equation}
\label{eq:normalization}
Z^{\bm{p},\bm{q}}_{(r,t]}:=  \prod_{i=r+1}^t \prod_{j=1}^N (1+ p_i q_j) .
\end{equation}
Then, by Theorem~\ref{thm:drsk}, the pushforward law of the tableaux under the $\drsk$ bijection at time $t$ is given by
\begin{equation}
\label{eq:drskPushforward}
\begin{split}
&\P\Big( \sfP(t)=\big(\lambda^{(0)}\prec\cdots\prec \lambda^{(N)}\big),\, \sfQ(t)^\top=\big(\mu^{(0),\top} \prec\cdots\prec\mu^{(t),\top}\big)  \Big) \\
= \;&\ind_{\lambda^{(N)}=\mu^{(t)}}
\frac{1}{Z^{\bm{p},\bm{q}}_{(0,t]}}
\prod_{i=1}^t p_i^{\lvert\mu^{(i)}/\mu^{(i-1)}\rvert}
\prod_{j=1}^N q_j^{\lvert\lambda^{(j)} / \lambda^{(j-1)}\rvert} ,
\end{split}
\end{equation}
where $\big(\lambda^{(0)}\prec\cdots\prec \lambda^{(N)}\big)$ and $\big(\mu^{(0),\top}\prec\cdots\prec\mu^{(t),\top}\big)$ are any sequences of interlacing partitions such that $\lambda^{(k)}, \mu^{(k),\top} \in \sfW_k$ for all $k$.

It follows from~\eqref{eq:drskPushforward} and from the definition~\eqref{eq:schur} of Schur polynomials that the marginal law of the common shape of $\sfP(t)$ and $\sfQ(t)$ is given by a \emph{Schur measure}:
\begin{align}
\label{eq:shapeMarginal}
\P\big( \sh(\sfP(t)) = \sh(\sfQ(t)) = \lambda \big) = \frac{1}{Z^{\bm{p},\bm{q}}_{(0,t]}} \, 
\schur_{\lambda^{\top}}\big(\bm{p}_{[1,t]} \big) \, \schur_{\lambda} (\bm{q}) ,
\end{align}
where $\bm{p}_{[1,t]}:= (p_1,\dots,p_t)$.
By summing the above probabilities over all partitions $\lambda$, one obtains the so-called \emph{dual Cauchy identity} (see e.g.~\cite[\S7.14]{stanley99}):
\[
\sum_{\lambda}
\schur_{\lambda^\top}\big(\bm{p}_{[1,t]}\big) \,
\schur_{\lambda}(\bm{q})
= \prod_{\substack{ 1\leq i\leq t \\ 1\leq j\leq N }} (1+p_iq_j) .
\]

Recalling~\eqref{eq:drskShapeDynamics} and taking a marginal of~\eqref{eq:drskPushforward}, we see that the joint distribution of the shapes of the $\sfP$-tableaux up to time $t$ is given by
\begin{align}
\label{eq:QtableauMarginal}
\begin{split}
&\P\Big( \big(\sh(\sfP(0)),\dots,\sh(\sfP(t)) \big) = \big( \mu^{(0)},\dots,\mu^{(t)}\big)  \Big) 
=\P\Big(\sfQ(t)^\top=\big(\mu^{(0),\top} \prec\cdots\prec\mu^{(t),\top}\big)  \Big) \\
&=\sumtwo{\lambda^{(0)} \prec\cdots \prec \lambda^{(N)}\colon}{\lambda^{(N)}=\mu^{(t)}}
\frac{1}{Z^{\bm{p},\bm{q}}_{(0,t]}}
\prod_{i=1}^t p_i^{\lvert\mu^{(i)}/\mu^{(i-1)}\rvert}
\prod_{j=1}^N q_j^{\lvert\lambda^{(j)} / \lambda^{(j-1)}\rvert}
=  \frac{1}{Z^{\bm{p},\bm{q}}_{(0,t]}}
\prod_{i=1}^t p_i^{\lvert\mu^{(i)}/\mu^{(i-1)}\rvert}
\schur_{\mu^{(t)}}(\bm{q}) .
\end{split}
\end{align}

For $0\leq r<t$, define now the kernels $\hat{\cR}_{(r,t]}$ and $\cR_{(r,t]}$ by setting
\begin{equation}
\label{eq:kerR}
\hat{\cR}_{(r,t]}(\mu,\lambda) 
:= \frac{\schur_{\lambda}(\bm{q}) }{\schur_{\mu}(\bm{q}) } 
\cR_{(r,t]}(\mu,\lambda)
:= \frac{\schur_{\lambda}(\bm{q}) }{\schur_{\mu}(\bm{q}) } 
\frac{1}{Z^{\bm{p},\bm{q}}_{(r,t]}}
\sumtwo{\nu^{(r)} \prec \nu^{(r+1)} \prec \cdots \prec \nu^{(t)}\colon }{\nu^{(r)} = \mu^\top, \, \nu^{(t)}=\lambda^\top} 
\prod_{i=r+1}^t p_i^{\lvert\nu^{(i)} / \nu^{(i-1)} \rvert}
\end{equation}
for any $\mu,\lambda\in\sfW_N$, where $\nu^{(k)} \in \sfW_k$ for all $k$.
From the first equality, we see that $\hat{\cR}_{(r,t]}$ can be interpreted as a Doob $h$-transform of $\cR_{(r,t]}$, with Schur polynomials as $h$-functions {(for a precise account of Doob's $h$-transforms, see e.g.~\cite{RogersWilliams00,Doob01})}.
It follows immediately from~\eqref{eq:QtableauMarginal} that $\hat{\cR}_{(r,t]}$ is the transition kernel of the shape of the $\sfP$-tableau from time $r$ to time $t$:
\begin{equation}
\label{eq:kerR_prob}
\begin{split}
\hat{\cR}_{(r,t]}(\mu,\lambda)
&= \P\left( \sh(\sfP(t)) =\lambda \;\middle|\; \sh(\sfP(r)) =\mu, \sh(\sfP(r-1)), \dots, \sh(\sfP(0)) \right) \\
&= \P\left( \sh(\sfP(t)) =\lambda \;\middle|\; \sh(\sfP(r)) =\mu \right) .
\end{split}
\end{equation}

Next, define the kernel $\hat{K}$ and $K$ by setting
\begin{align}
\label{eq:kerK}
\hat{K}(\lambda,\bm{y})
:= \frac{1}{\schur_\lambda(\bm{q})}
K(\lambda,\bm{y})
:= \frac{1}{\schur_\lambda(\bm{q})}
\sumtwo{\lambda^{(0)}\prec \cdots \prec \lambda^{(N)}=\lambda\colon}{\big(\lambda^{(1)}_1,\dots,\lambda^{(N)}_N\big) = \bm{y}}
\prod_{j=1}^N q_j^{\lvert\lambda^{(j)} / \lambda^{(j-1)}\rvert}
\end{align}
for any $\lambda, \bm{y} \in \sfW_N$, where, as usual, $\lambda^{(k)} \in \sfW_k$ for all $k$.
Notice that $K$ is an unnormalised version of $\hat{K}$, which is a probability kernel.
It follows from~\eqref{eq:drskPushforward} and~\eqref{eq:QtableauMarginal} that, for all $t\geq 0$,
\begin{align}
\label{eq:kerK_prob}
\begin{split}
\hat{K}(\lambda,\bm{y})
&= \P\left( \ledge(\sfP(t)) = \bm{y} \;\middle|\; \sh(\sfP(t)) =\lambda, \sh(\sfP(t-1)), \dots, \sh(\sfP(0)) \right) \\
&= \P\left( \ledge(\sfP(t)) = \bm{y} \;\middle|\; \sh(\sfP(t)) =\lambda \right) .
\end{split}
\end{align}

Finally, recall from~\S\ref{subsec:RSK-TASEP} that the left edge of $\sfP$ evolves as a Markov chain in its own filtration (i.e., autonomously from the rest of $\sfP$).
Thus, we may write its transition kernel from time $r$ to time $t$, for $0\leq r<t$, as
\begin{equation}
\label{eq:kerQ}
\begin{split}
\cQ_{(r,t]}(\bm{y}, \bm{y}')
:= \;&\P\left( \ledge(\sfP(t)) = \bm{y} ' \;\middle|\; \ledge(\sfP(r)) = \bm{y} \right) \\
= \;&\P\left( \ledge(\sfP(t)) = \bm{y}' \;\middle|\; \ledge(\sfP(r)) = \bm{y}, \sfP(r),\dots,\sfP(0) \right)
\end{split}
\end{equation}
for $\bm{y}, \bm{y}' \in \sfW_N$.
{We will soon derive explicit expressions for the kernel $Q_{(r,t]}$ defined above.}

The $\cQ$- and $\hat{\cR}$-kernels are transition kernels of the left edge and of the shape of the $\sfP$-tableau, respectively; on the other hand, $\hat{K}$ encodes the conditional law of the left edge of $\sfP$ given its shape at any given time.
Therefore, from the theory of Markov functions (see e.g.~\cite{rogersPitman81}), we expect these kernels to satisfy intertwining relations, and this is indeed the case.
We state the result in the next proposition and, for completeness, we also provide a proof, following~\cite{diekerWarren08}.
\begin{proposition}\label{prop:intertwine}
For $0\leq r<t$, the following intertwining relations between operators from $\sfW_N$ to $\sfW_N$ hold:
\begin{align}
\label{eq:intertw}
\hat{\cR}_{(r,t]} \hat{K}
= \hat{K} \, \cQ_{(r,t]}
\qquad \text{and} \qquad
\cR_{(r,t]} K
= K \, \cQ_{(r,t]} .
\end{align}
\end{proposition}
\begin{proof}
Let $\bm{y}\in \sfW_N$.
By~\eqref{eq:kerK_prob} and~\eqref{eq:kerR_prob}, we have
\[
\begin{split}
&\P\big(\ledge (\sfP(t)) = \bm{y} \;\big| \sh(\sfP(0)), \dots, \sh(\sfP(r)) \big) \\
= \;& \E\big[ \P\big(\ledge(\sfP(t)) = \bm{y} \;\big| \sh(\sfP(0)), \dots, \sh(\sfP(t)) \big) \;\big| \sh(\sfP(0)), \dots, \sh(\sfP(r)) \big] \\
= \;& \E\big[ \hat{K}(\sh(\sfP(t)), \bm{y}) \;\big| \sh(\sfP(0)), \dots, \sh(\sfP(r)) \big]
= \sum_{\lambda} \hat{\cR}_{(r,t]}( \sh(\sfP(r)), \lambda) \, \hat{K}(\lambda, \bm{y}) .
\end{split}
\]
{We point out that, by the definition of $\hat{\cR}_{(r,t]}$ in~\eqref{eq:kerR}, $\hat{\cR}_{(r,t]}( \sh(\sfP(r)), \lambda)$ is non-zero only for a finite number of partitions $\lambda$.}

On the other hand, by~\eqref{eq:kerQ} and~\eqref{eq:kerK_prob}, we have
\[
\begin{split}
&\P\big(\ledge(\sfP(t)) = \bm{y} \;\big| \sh(\sfP(0)), \dots, \sh(\sfP(r)) \big) \\
= \;& \E\big[ \P\big(\ledge(\sfP(t)) = \bm{y} \;\big|\; \sfP(0), \dots, \sfP(r) \big) \;\big| \sh(\sfP(0)), \dots, \sh(\sfP(r)) \big] \\
= \;& \E\big[ \cQ_{(r,t]}(\ledge (\sfP(r)), \bm{y}) \;\big| \sh(\sfP(0)), \dots, \sh(\sfP(r)) \big]
= \sum_{\lambda} \hat{K}(\sh(\sfP(r)), \lambda) \, \cQ_{(r,t]}(\lambda,\bm{y}) .
\end{split}
\]
Comparing the two expressions above leads to the first intertwining relation in~\eqref{eq:intertw}.
The second one follows from the
first, by using~\eqref{eq:kerR} and~\eqref{eq:kerK} and noting that the Schur polynomials cancel out.
\end{proof}

\begin{remark}
Notice that, by construction, $\cQ_{(r,t]}(\bm{y}, \bm{y}')$ equals zero unless $y_j\leq y_j'$ for all $1\leq j\leq N$.
\end{remark}

Define now the modified kernels
\begin{align}
\label{eq:kerLambda}
\Lambda(\lambda,\bm{y})
:= q_1^{-y_1} \cdots q_N^{-y_N} K(\lambda, \bm{y} ),
\qquad
\hat{\cQ}_{(r,t]}(\bm{y},\bm{y}'):= \Bigg( \prod_{j=1}^N q_j^{y_j-y_j'}\Bigg) \cQ_{(r,t]}(\bm{y},\bm{y}') .
\end{align}
It is immediate to see that the second intertwining relation in~\eqref{eq:intertw} still holds when replacing $K$ with $\Lambda$ and $\cQ_{(r,t]}$ with $\hat{\cQ}_{(r,t]}$.
Moreover, it was proven in~\cite[Prop.~3]{diekerWarren08} that $\Lambda$ is invertible (with an explicit inverse).
This provides an explicit expression for the kernel $\hat{\cQ}_{(r,t]}$.
We summarize these facts in the next proposition.
\begin{proposition}\label{prop:intertwining}
The intertwining relation $\cR_{(r,t]} \Lambda = \Lambda \, \hat{\cQ}_{(r,t]}$ holds.
Moreover, the operator $\Lambda$ is invertible, so that
\begin{equation}
\label{eq:kerQ_LambdaInverse}
\hat{\cQ}_{(r,t]} = \Lambda^{-1} \cR_{(r,t]} \, \Lambda .
\end{equation}
\end{proposition}

In the next section we will interpret~\eqref{eq:kerQ_LambdaInverse} in terms of weights of non-intersecting paths.

\section{Path ensembles and determinantal point processes}
\label{sec:path-constr}

{This section concerns the non-intersecting path constructions that lie at the core of our approach.
In~\S\ref{subsec:localOperators}, we introduce certain `local' Toeplitz operators that we will use throughout this work.
In~\S\ref{subsec:pathEnsembles}, we provide a non-intersecting path interpretation of the $\dtasep$  transition kernel in terms of these local operators.
In~\S\ref{subsec:DPP}, we deduce an expression for the law of $\dtasep$ in terms of a determinantal point process, which we then study in~\S\ref{subsec:modified-kernel}, obtaining an initial expression for its correlation kernel in terms of biorthogonal functions and local operators.}
	
\subsection{Local operators}
\label{subsec:localOperators}

We will express the weights of the path ensembles in terms of convolutions of local operators, which we now introduce. 
Let us first define the conventions for the operator formalism.
For an operator $A$ defined through a kernel	$(A(x,y) \colon x\in X, y\in Y)$ on suitable spaces $X,Y$ and a function (or vector) $f=(f(y) \colon y\in Y)$, we define the function $(A\circ f)(x):=\sum_{y\in Y} A(x,y) f(y)$, {whenever the sum is absolutely convergent.}
Similarly, for a function $g=(g(x) \colon x\in X)$, we define the function $(g\circ A)(y):=\sum_{x \in X} g(x) A(x,y)$. 
For two operators $A$ and $B$ with kernels $(A(x,y) \colon x\in X, y\in Y\big)$ and $(B(y,z) \colon y\in Y, z\in Z)$, respectively, we define the operator $A\circ B$ through the kernel $(A\circ B) (x,z) :=\sum_{y\in Y} A(x,y) B(y,z)$.
Finally, we define the adjoint of $A$ as the operator $A^*$ with kernel $A^*(x,y):= A(y,x)$.
	
Let us now introduce some specific local operators we are concerned with.
Recall from~\S\ref{subsec:RSK-TASEP} that we fixed positive parameters $\bm{p}=(p_i)_{i\geq 1}$ and $\bm{q}=(q_1,\dots,q_N)$.

The first family of operators encode geometric jumps \emph{weakly} to the right: for $i=1,\dots,N$, let
\begin{align}
\label{eq:Qdagger}
Q^\dagger_i(x,y):= q_i^{y-x}\ind_{\{y\geq x\}}, \qquad x,y\in\Z.
\end{align}
We also define a family of operators encoding geometric jumps \emph{strictly} to the left:
\begin{equation}
\label{eq:Q}
Q_i(x,y):= q_i^{y-x}\1_{\{y< x\}},\qquad x,y\in \mathbb{Z}.
\end{equation}
We note that, under the hypothesis $q_i>1$ {(which we will always assume, without explicitly mentioning, from now on)}, the kernel $Q_i(x,y)$ defines a bounded operator on $\ell^1(\Z)$ with a well-defined inverse:
\begin{equation}
\label{eq:Qinverse}
Q_i^{-1}(x,y):= -\1_{y=x}+q_i\1_{y=x+1},\qquad x,y\in \mathbb{Z}.
\end{equation}
Finally, for $i\geq 1$, we define the operators
\begin{align}
\label{eq:R_operator}
R_i(x,y):=\ind_{\{y=x\}}+p_i\ind_{\{y=x+1\}}.
\end{align}

For $1\leq m\leq n\leq N$, we will use the compact notations 
\begin{align}
\label{eq:Q_composition}
Q_{(m-1,n]}
&:=Q_{[m,n]}
:= Q_{m}\circ\cdots\circ  Q_n, \\
\label{eq:Qinverse_composition}
Q^{-1}_{(m-1,n]}
&:=Q^{-1}_{[m,n]}:=  Q_{n}^{-1}\circ\cdots\circ  Q_m^{-1}.
\end{align}
We will abuse the notation slightly by defining
\begin{equation}
\label{eq:Q_compositionExt}
Q_{[m,n]}:=  Q_{m}\circ\cdots\circ  Q_N\circ Q_N^{-1}\circ\cdots  Q_{n+1}^{-1},
\end{equation}
which makes sense even for $m>n$, in which case $Q_{[m,n]}=Q_{m-1}^{-1} \circ \dots \circ Q_{n+1}^{-1}$.
{In particular, we have $Q_{(n,n]}=Q_{[n+1,n]}:=I$ for $1\leq n\leq N$.}
We will use similar conventions for the $Q^\dagger$- and $R$-operators.

Certain convolutions of the operators defined above may be expressed in terms of symmetric functions.
Given indeterminates $x_1,\dots,x_N$, let 
\[
h_n(x_1,\dots,x_N):=\sum_{1\leq i_1\leq\cdots \leq i_n \leq N} x_{i_1}\cdots x_{i_n}, \qquad
e_{n}(x_1,\dots,x_N):=\sum_{1\leq i_1<\cdots <i_n\leq N} x_{i_1}\cdots x_{i_n}
\]
be the complete symmetric polynomial of degree $n$ and the elementary symmetric polynomial of degree $n$, respectively. 
{By convention, we set $h_0=e_0:=1$ and $h_n=e_n:=0$ for all $n<0$.}
Then, it is not difficult to check that
\begin{align}\label{eq:Qdagger_h}
Q^\dagger_{(i,N]}(x,y)
& = h_{y-x}(0,\dots,0,q_{i+1},\dots,q_N), \\
\intertext{and}
\label{eq:Qinverse_e}
\begin{split}
(-1)^{N-i} Q_{(i,N]}^{-1}(x,y)
&= e_{y-x}(0,\dots,0, -q_{i+1},\dots,-q_N) \\
&= (-1)^{y-x} e_{y-x}(0,\dots,0, q_{i+1},\dots,q_N) ,
\end{split}
\end{align}
for $x,y\in\Z$, where the first $i$ indeterminates of the symmetric polynomials are set to be $0$.
{These identities follow from the definitions of the symmetric functions $h_n$ and $e_n$, the form of the operators
$Q^\dagger$ and $Q^{-1}$ in \eqref{eq:Qdagger} and \eqref{eq:Qinverse} and the definition of the operation $\circ$}.
	
Observe that the values of $Q^\dagger_i(x,y)$, $Q_i(x,y)$, $Q_i^{-1}(x,y)$ and $R_i(x,y)$ only depend on $y-x$.
The operators with such a property are known as (bi-infinite) \emph{Toeplitz operators}.
To each Toeplitz operator $T$ with kernel $T(x,y)$ on $\mathbb{Z}\times \mathbb{Z}$, we associate a formal Laurent series $\varphi_{T}(z)$, known as the symbol of $T$, defined by
\[
\varphi_{T}(z):= \sum_{x\in \mathbb{Z}} T(0,x)z^{-x}.
\]
Inside its domain of convergence, which is a (possibly empty) annulus $\{r<|z|<R\}$, the function $\varphi_T(z)$ is analytic in $z$.
We summarize some standard properties of Toeplitz operators that will be used later; the proofs are elementary, so we omit them.
From now on, all contours will be implicitly taken to have a counterclockwise orientation.
	
\begin{proposition}
\begin{enumerate}
\item Let $T$ be a (bi-infinite) Toeplitz operator whose symbol $\varphi_T(z)$ is analytic in a non-empty annulus $\{r<|z|<R\}$.
Then, the entries $T(x,y)$ can be computed through the contour integral
\begin{equation}
\label{eq:invertingGF}
T(x,y)=\oint_{|z|=r_1}\frac{\mathrm{d}z}{2\pi \i z}z^{y-x}\cdot \varphi_T(z),
\end{equation}
for any $r<r_1<R$.
\item Let $T$ and $S$ be two (bi-infinite) Toeplitz operators whose symbols $\varphi_T(z)$ and $\varphi_S(z)$ are both analytic inside a common non-empty annulus $\{r<|z|<R\}$.
Then, the convolutions $T\circ S$ and $S\circ T$ both converge, with
\begin{equation}
\varphi_{T\circ S}(z)
= \varphi_{T}(z)\varphi_{S}(z)
= \varphi_{S\circ T}(z)
\end{equation} 
for all $z$ on the annulus.
In particular, $T$ and $S$ commute.
Assuming that $T$ is invertible with $T^{-1}=S$, we have 
\begin{equation*}
1=\varphi_{\id}(z)
=\varphi_T(z)\varphi_{T^{-1}}(z),
\end{equation*}
or equivalently 
\[
\varphi_{T^{-1}}(z)=\varphi_{T}(z)^{-1}.
\]
\end{enumerate}
\label{prop:Toeplitz}
\end{proposition}

For example, the symbols of the $Q^\dagger$- and $Q$-operators defined above are given by
\begin{align}
\label{eq:Qdagger_symbol}
\phi_{Q^\dagger_i}(z) &=\frac{z}{z-q_i} &&\text{for } |z|>q_i, \\
\label{eq:Q_symbol}
\phi_{Q_i}(z) &= \frac{z}{q_i-z} &&\text{for } |z|<q_i, \\
\label{eq:Qinverse_symbol}
\phi_{Q_i^{-1}}(z) &= \frac{q_i-z}{z} &&\text{for } |z|>0.
\end{align}
As a consequence of Proposition~\ref{prop:Toeplitz} {and the fact that the series are absolutely convergent in the domains considered below}, we then obtain the contour integral representations
\begin{align}\label{eq:Qdagger_integral}
Q^\dagger_{[m,n]}(x,y)
&=(-1)^{n-m+1}\oint_{|z|=R} \frac{\mathrm{d}z}{2\pi \mathrm{i} z} \frac{z^{y-x+n-m+1}}{\prod_{\ell=m}^{n}(q_\ell-z)}
&&\text{for } R>\max\{q_i\}_{i=m}^n, \\
\label{eq:Q_integral}
Q_{[m,n]}(x,y)
&=\oint_{|z|=r} \frac{\mathrm{d} z}{2\pi \mathrm{i}z} \frac{z^{y-x+n-m+1}}{\prod_{\ell=m}^{n}(q_\ell -z)}
&&\text{for } 0<r<\min\{q_i\}_{i=m}^n, \\
\label{eq:Qinverse_integral}
Q^{-1}_{[m,n]}(x,y)
&=\oint_{|z|=r} \frac{\mathrm{d} z}{2\pi \mathrm{i}z} z^{y-x+m-n-1}\prod_{\ell=m}^{n}(q_\ell-z)
&&\text{for } r>0.
\end{align}

\subsection{Non-intersecting path ensembles}
\label{subsec:pathEnsembles}

We now define two ensembles of paths, which we call \emph{$h$-paths} $\sfh\Pi$ (related to the complete symmetric polynomials $h$) and \emph{$e$-paths} $\sfe\Pi$ (related to the elementary symmetric polynomials $e$).

Let $(y_1,i_1),\dots,(y_n,i_n),(x_1,j_1),\dots,(x_n,j_n) \in \Z^2$.
We denote by ${\sfh\Pi}_{\{(y_1,i_1),\dots,(y_n,i_n)\}}^{\{(x_1,j_1),\dots,(x_n,j_n)\}}$ the (possibly empty) ensemble of all of $n$-tuples $(\pi_1,\dots,\pi_n)$ of non-intersecting paths in $\Z^2$, such that each path $\pi_k$ starts from $(y_k,i_k)$, ends at $(x_k,j_k)$, and moves either straight up or straight to the right at each step; namely, from a point $(x,j)$ the path moves either to $(x,j+1)$ or to $(x+1,j)$.
We also denote by $\sfh\Pi_{\{(y_1,i_1),\dots,(y_n,i_n)\}, \uparrow}^{\{(x_1,j_1),\dots,(x_n,j_n)\}}$ the subset of $\Pi_{\{(y_1,i_1),\dots,(y_n,i_n)\}}^{\{(x_1,j_1),\dots,(x_n,j_n)\}}$ of all $(\pi_1,\dots,\pi_n)$ such that the first step of each path $\pi_k$ is vertical, upwards.
	
We also denote by ${\sfe\Pi}_{\{(y_1,i_1),\dots,(y_n,i_n)\}}^{\{(x_1,j_1),\dots,(x_n,j_n)\}}$ the ensemble of all $n$-tuples $(\pi_1,\dots,\pi_n)$ of non-intersecting paths in $\Z^2$, such that each path $\pi_k$ starts from $(y_k,i_k)$, ends at $(x_k,j_k)$, and moves either straight up or diagonally up-right at each step; namely, from a point $(x,i)$ the path moves either to $(x,i+1)$ or to $(x+1,i+1)$.
Finally, we denote by $({\sfe\Pi}^*)_{\{(y_1,i_1),\dots,(y_n,i_n)\}}^{\{(x_1,j_1),\dots,(x_n,j_n)\}}$ a similar $e$-path ensemble, where the allowed diagonal steps are up-left, instead of up-right.
	
In the following, both $h$- and $e$-paths will be assigned weights, based on suitable weights $\sfw\sft(\sfe)$ assigned to each edge $\sfe$.
The rules are as follows.
The weight of a path $\pi$ with edges $\sfe_1,\sfe_{2},\dots$ is defined as $\sfw\sft(\pi)= \sfw\sft(\sfe_1) \sfw\sft(\sfe_2)\cdots$. 
The total weight of an $n$-tuple $(\pi_1,\dots,\pi_n)$ of paths is defined as $\sfw\sft(\pi_1,\dots,\pi_n):=\prod_{i=1}^n\sfw\sft(\pi_i)$.
Finally, the weight of an \emph{ensemble} $\Pi$ of $n$-tuples of paths is defined as $\sfw\sft(\Pi):= \sum_{(\pi_1,\dots,\pi_n)\in\Pi} \sfw\sft(\pi_1,\dots,\pi_n)$.

We are now ready to provide the path and local operator representations of the kernels appearing in~\eqref{eq:kerQ_LambdaInverse}.
\begin{proposition}[Path and local operator representation of $\Lambda$]
\label{prop:pathLambda}
Let a vertical edge connecting $(x,i)$ to $(x,i+1)$ be assigned weight $1$ and a horizontal edge connecting $(x,i)$ to $(x+1,i)$ be assigned weight $q_i$, for $x\in \Z$ and $1\leq i\leq N$.
For $\lambda=(\lambda_1\geq\cdots \geq\lambda_N)$ and $\bm{y}'=(y_1'\geq\dots\geq y_N')$ in $\sfW_N$, the kernel $ \Lambda(\lambda,\bm{y}')$ defined in~\eqref{eq:kerLambda} can be written as
\begin{align}
\Lambda(\lambda,\bm{y}')
&= \sfw\sft \Big({\sfh\Pi}_{\{ (y_i'-i,i) \colon 1\leq i\leq N\}, \uparrow}^{\{(\lambda_i-i,N)\colon 1\leq i\leq N\}} \Big)
\label{eq:pathLambda1}\\
&= \det \Big( \sfw\sft\Big({\sfh\Pi}_{(y_i'-i,i),\uparrow}^{(\lambda_j-j,N)} \Big)\Big)_{1\leq i,j\leq N}
\label{eq:pathLambda2}\\
&= \det \Big( Q^\dagger_{(i,N]}(y_i'-i, \lambda_j-j)\Big)_{1\leq i,j\leq N}. 
\label{eq:pathLambda3}
\end{align}
\end{proposition}
\begin{proof}
The first equality is a rewriting of the definition of $\Lambda$ (see~\eqref{eq:kerLambda} and~\eqref{eq:kerK}) in terms of weights of non-intersecting path ensembles; see e.g.~\cite[Section 4]{fulmekKrattenthaler97} for a description of the connection between tableaux and non-intersecting lattice paths.
The second equality is an application of the Lindstr\"om--Gessel--Viennot theorem.
Note now that, by the definition of path weights and by the form of the complete symmetric polynomials, we have
\[
{\sfw\sft}\Big({\sfh\Pi}_{(y,i),\uparrow}^{(x,N)} \Big)
= h_{x-y}(0,\dots,0,q_{i+1},\dots,q_N).
\]
Combining this with~\eqref{eq:Qdagger_h}, we arrive at the third equality.
Notice that this proposition can be also seen as a reformulation of~\cite[Prop.~2]{diekerWarren08}.
\end{proof}
As stated in Proposition~\ref{prop:intertwining}, the operator $\Lambda$ is invertible:
We now provide a determinantal and 	path representation of its inverse.
\begin{proposition}[Path and local operator representation of $\Lambda^{-1}$]\label{prop:pathLambdaInv}
Let a vertical edge connecting $(x,i-1)$ to $(x,i)$ be assigned weight $1$ and a diagonal up-right edge connecting $(x,i-1)$ to $(x+1,i)$ be assigned weight $-q_{N-i+1}$, for $x\in \Z$ and $1\leq i\leq N-1$.
For $\bm{y}=(y_1\geq \cdots\geq y_N)$ and $\mu=(\mu_1\geq \cdots\geq \mu_N)$ in $\sfW_N$, the kernel $\Lambda^{-1}(\bm{y},\mu)$ can be written as
\begin{align}
\Lambda^{-1}(\bm{y},\mu) 
&= \det \Big( (-1)^{N-j} Q_{(j,N]}^{-1}(\mu_i-i, y_j-j)  \Big)_{1\leq i,j \leq N}
\label{eq:pathLambdaInv3} \\
&= \det \Big( \sfw\sft\Big(\sfe\Pi^{(y_j-j,N-j)}_{(\mu_i-i,0)} \Big)\Big)_{1\leq i,j \leq N} 
\label{eq:pathLambdaInv2}  \\
&= \sfw\sft\Big(\sfe\Pi^{\{(y_i-i,N-i) \colon 1\leq i\leq N\}}_{\{(\mu_i-i,0) \colon 1\leq i\leq N\}} \Big) .
\label{eq:pathLambdaInv1}
\end{align}
\end{proposition}
\begin{proof}
It was proved in~\cite[Prop.~3]{diekerWarren08} that $\Lambda$ is invertible, with an inverse given by
\begin{align}
\label{eq:LambdaInv_elem}
\Lambda^{-1}(\bm{y},\,\mu)
=\det \Big( (-1)^{y_j-\mu_i-j+i} e_{y_j-\mu_i-j+i}(0,\dots,0, q_{j+1},\dots,q_N) \Big)_{1\leq i,j \leq N}.
\end{align}
This, together with~\eqref{eq:Qinverse_e}, yields the first equality.
On the other hand, from the form of the elementary symmetric functions, it is easy to see that
\begin{align*}
e_{y-x}(0,\dots,0, -q_{j+1},\dots,-q_N)
= \sfw\sft\Big(\sfe\Pi^{(y,N-j)}_{(x,0)} \Big).
\end{align*}
The latter, together with~\eqref{eq:LambdaInv_elem} and~\eqref{eq:Qinverse_e}, yields the second equality.
Finally, the third equality follows from the Lindstr\"om--Gessel--Viennot theorem.
\end{proof}
	
The proof of the following proposition follows the same lines as the proofs of Propositions~\ref{prop:pathLambda} and~\ref{prop:pathLambdaInv}, so we {will be brief}.
\begin{proposition}[Path and local operator representation of $\cR_{(r,t]}$]\label{prop:pathR}
Let $0\leq r< t$.
Let a vertical edge connecting $(x,i)$ to $(x,i+1)$ be assigned weight $1$ and a diagonal up-right edge connecting $(x,i)$ to $(x+1,i+1)$ be assigned weight $p_{i+1}$, for $x\in \Z$ and $r\leq i\leq t-1$.
Then, for $\lambda=(\lambda_1\geq \cdots \geq \lambda_N)$ and $\mu=(\mu_1\geq \cdots \geq \mu_N)$ in $\sfW_N$ with $\mu\subseteq\lambda$, the kernel $\cR_{(r,t]}(\mu,\lambda)$ defined in~\eqref{eq:kerR} can be written as
\begin{align}
Z^{\bm{p},\bm{q}}_{(r,t]} \cdot  \cR_{(r,t]}(\mu,\lambda) 
&= \sfw\sft\Big( \sfe\Pi_{\{(\mu_i-i,r)\colon 1\leq i\leq N\}}^{\{(\lambda_i-i,t) \colon 1\leq i\leq N\}} \Big) \label{eq:pathR1}\\
&= \det\Big( \sfw\sft\Big( \sfe\Pi_{(\mu_i-i,r)}^{(\lambda_j-j,t)} \Big) \Big)_{1\leq i,j\leq N} \label{eq:pathR2}\\
&=\det\Big( R_{(r,t]} (\mu_i-i,\lambda_j-j)\Big)_{1\leq i,j\leq N}, \label{eq:pathR3}
\end{align}
where the $R$-operators are defined in~\eqref{eq:R_operator}.
\end{proposition}
\begin{proof}
{The first equality follows from the definition of $\cR_{(r,t]}(\mu,\lambda)$ in~\eqref{eq:kerR} and its representation in terms of weights of non-intersecting lattice paths.
The second equality is then a consequence of the Lindstr\"om--Gessel--Viennot theorem, while the last equality is a direct consequence of the representation of the weight of a single $e$-path ensemble in terms of the $R$-operator.}
\end{proof}
	
The following proposition provides a path representation of the transition kernel of $\dtasep(N;\bm{p},\bm{q})$, {obtained by path concatenation.} 
The graphical depiction of this result is shown in Figure~\ref{fig:paths}.
{Let us first explain what we mean by path concatenation, again referring to Figure~\ref{fig:paths} for an illustration.
Let $\Sigma^{(1)}$ and $\Sigma^{(2)}$ be two lattice path ensembles, each consisting of $N$-tuples of paths.
Suppose that, for all $(\pi^{(1)}_1,\dots,\pi^{(1)}_N)\in \Sigma^{(1)}$ and $(\pi^{(2)}_1,\dots,\pi^{(2)}_N)\in \Sigma^{(2)}$ and for all $j$, the endpoint of $\pi^{(1)}_j$ equals the starting point of $\pi^{(2)}_j$.
Then, we define the path concatenation $\Sigma_1\sqcup\Sigma_2$ to be the ensemble consisting of all paths $(\pi^{(1)}_1 \cup \pi^{(2)}_1,\dots,\pi^{(1)}_N \cup \pi^{(2)}_N)$, for some $(\pi^{(1)}_1,\dots,\pi^{(1)}_N)\in \Sigma^{(1)}$ and $(\pi^{(2)}_1,\dots,\pi^{(2)}_N)\in \Sigma^{(2)}$ (here, union of paths is understood in terms of both edges and vertices).}
\begin{proposition}[Path representation of $\dtasep$ transition kernel]\label{prop:pathTASEP}
Let $\bm{y}, \bm{y}'\in\sfW_N$ with $\bm{y} \subseteq \bm{y}'$.
The transition kernel of $\dtasep(N;\bm{p},\bm{q})$, encoding the probability that particles starting from locations $(y_1-1>y_2-2>\dots>y_N-N)$ at time $r$ end up at locations $(y_1'-1>y_2'-2>\dots>y_N'-N)$ at time $t$, admits the following weighted path representation:
\begin{equation}\label{eq:pathQ1}
\begin{split}
\cQ_{(r,t]}(\bm{y},\bm{y}') 
&= \frac{\big( \prod_{i=1}^N q_i^{y'_i-y_i}\big)}{Z^{\bm{p},\bm{q}}_{(r,t]}}  
\sum_{\substack{ \lambda, \mu \in \sfW_N\colon \\ \mu \subseteq \lambda, \, \mu\subseteq \bm{y}, \, \bm{y}'\subseteq \lambda }} \sfw\sft \bigg(	\sfh\Pi_{\{(y'_i-i,i) \colon 1\leq i\leq N\}, \uparrow}^{\{(\lambda_i-i,N) \colon 1\leq i\leq N\}} \, \bigsqcup \\
& \quad\, \bigsqcup \, (\sfe\Pi_r^*)_{\{(\lambda_i-i,N) \colon 1\leq i\leq N\}}^{\{(\mu_i-i,N+t-r) \colon 1\leq i\leq N\}} \,
\bigsqcup \,\sfe\Pi^{\{(y_i-i,2N+t-r-i) \colon 1\leq i\leq N\}}_{\{(\mu_i-i,N+t-r) \colon 1\leq i\leq N\}}  \bigg).
\end{split}
\end{equation}
The weights assigned to the edges are as follows:
\begin{itemize}
\item All vertical edges are assigned weight $1$.
\item Horizontal edges between $(x,i)$ and $(x+1,i)$ for $x\in\Z$ and $1\leq i\leq N$ are assigned weight $q_i$.
\item Diagonal, up-left edges between $(x,i)$ and $(x-1,i+1)$, for $x\in\Z$ and $N\leq i\leq N+t-r-1$, are assigned weight $p_{N+t-i}$.
\item Finally, diagonal, up-right edges between $(x,i)$ and $(x+1,i+1)$ for $x\in\Z$ and $N+t-r\leq i\leq 2N+t-r-2$, are assigned weight ${-q_{2N+t-r-i}}$.
\end{itemize}
Choose now $x_0$ so that $x_0-1< y_N-N$ and consider an auxiliary vector $\bm{x}^{(0)}=(x^{(0)}_1,\dots,x^{(0)}_{N})$, with $x^{(0)}_i=x_0-i$ for $1\leq i\leq N$.
Then, with the same assignment of weights and setting {$\widehat Z^{\bm{p},\bm{q}}_{(r,t]}:=Z^{\bm{p},\bm{q}}_{(r,t]} \prod_{i=1}^N q_i^{y_i-x_0}$}, we also have
\begin{equation}\label{eq:pathQ2}
\begin{split}
\cQ_{(r,t]}(\bm{y},\bm{y}') 
&= \frac{1}{\widehat Z^{\bm{p},\bm{q}}_{(r,t]}}
\sum_{\substack{ \lambda, \mu \in \sfW_N\colon \\ \mu \subseteq \lambda, \, \mu\subseteq \bm{y}, \, \bm{y}'\subseteq \lambda}} \sfw\sft \bigg(	{\sfh\Pi}^{\{(y'_i-i,i) \colon 1\leq i\leq N\}}_{\{(x_{i}^{(0)},i) \colon 1\leq i\leq N \}}
\,\bigsqcup \, \sfh\Pi_{\{(y'_i-i,i) \colon 1\leq i\leq N\} ,\uparrow}^{\{(\lambda_i-i,N) \colon 1\leq i\leq N\}} \,\bigsqcup \\
&\quad \,\bigsqcup\, ({\sfe\Pi}^*_r)_{\{(\lambda_i-i,N) \colon 1\leq i\leq N\}}^{\{(\mu_i-i,N+t-r) \colon 1\leq i\leq N\}}
\, \bigsqcup \,\sfe\Pi^{\{(y_i-i,2N+t-r-i) \colon 1\leq i\leq N\}}_{\{(\mu_i-i,N+t-r) \colon 1\leq i\leq N\}}  \bigg) .
\end{split}
\end{equation}
\end{proposition}

\begin{figure}
\begin{tikzpicture}[scale=0.45, font=\small]
			\foreach \i in {0,1,...,19}{
				\draw[dotted] (0, \i  ) grid (28, \i );
			}
			\draw[thin] (0,15)--(28,15);
			\draw[thin] (0,5)--(28,5);
			\node[draw,red, circle,inner sep=1.5pt,fill] at (8, 5) {}; \node[draw,red, circle,inner sep=1.5pt,fill] at (5, 15) {};
			\node at (8,4.5) {$y_5'\!-\!5=\lambda_5\!-\!5$};
			\node at (5,15.5) {$y_5\!-\!5=\mu_5\!-\!5$};
			\draw[ultra thick, red]  (0,5)--(8,5)--(8,6)--(8,7)--(8,8)--(8,9)--(7,10)--(7,11)--(7,12)--(7,13)--(6,14)--(5,15);
			\draw[ultra thick, dashed, red] (0,5)--(0,0);
			\node[draw,blue, circle,inner sep=1.5pt,fill] at (11, 4) {}; \node[draw,blue, circle,inner sep=1.5pt,fill] at (15, 5) {};
			\node[draw,blue, circle,inner sep=1.5pt,fill] at (10, 15) {};\node[draw,blue, circle,inner sep=1.5pt,fill] at (10, 16) {};
			\node at (15,4.5) {$\lambda_4\!-\!4$}; \node at (11,3.5) {$y_4'\!-\!4$};
			\node at (10, 16.5) {$y_4\!-\!4$}; \node at (8.8, 14.5) {$\mu_4\!-\!4$};
			\draw[ultra thick, blue]  (1,4)--(11,4)--(11,5)--(15,5)--(14,6)--(13,7)--(13,8)--(13,9)--(12,10)--(12,11)--(11,12)--(10,13)--(10,14)--(10,15)--(10,16); \draw[ultra thick, dashed, blue] (1,4)--(1,0);
			\node[draw,orange, circle,inner sep=1.5pt,fill] at (16, 3) {}; \node[draw,orange, circle,inner sep=1.5pt,fill] at (17, 4) {};
			\node[draw,orange, circle,inner sep=1.5pt,fill] at (18, 5) {}; \node[draw,orange, circle,inner sep=1.5pt,fill] at (13, 15) {};
			\node[draw,orange, circle,inner sep=1.5pt,fill] at (14, 17) {}; \node at (16, 2.5) {$y_3'\!-\!3$}; \node at (18.5, 4.5) {$\lambda_3\!-\!3$};
			\node at (14, 17.5) {$y_3\!-\!3$}; \node at (11.9, 14.5) {$\mu_3\!-\!3$};
			\draw[ultra thick, orange]  (2,3)--(16,3)--(16,4)--(17,4)--(17,5)--(18,5)--(18,6)--(17,7)--(16,8)--(15,9)--(15,10)--(14,11)--(14,12)--(14,13)--(13,14)--(13,15)--(14,16)--(14,17); \draw[ultra thick, dashed, orange] (2,3)--(2,0);
			\node[draw,brown, circle,inner sep=1.5pt,fill] at (18, 2) {}; \node at (18, 1.5) {$y_2'\!-\!2$}; \node at (22.2, 4.5) {$\lambda_2\!-\!2$};
			\node[draw,brown, circle,inner sep=1.5pt,fill] at (21, 4) {}; \node[draw,brown, circle,inner sep=1.5pt,fill] at (19, 3) {};  
			\node[draw,brown, circle,inner sep=1.5pt,fill] at (22, 5) {}; \node[draw,brown, circle,inner sep=1.5pt,fill] at (18, 18) {}; 
			\node[draw,brown, circle,inner sep=1.5pt,fill] at (17, 15) {};  \node at (18, 18.5) {$y_2\!-\!2$}; \node at (16, 14.5) {$\mu_2\!-\!2$};
			\draw[ultra thick, brown]  (3,2)--(18,2)--(18,3)--(19,3)--(19,4)--(21,4)--(21,5)--(22,5)--(21,6)--(20,7)--(20,8)--(20,9)--(20,10)--(19,11)--(19,12)--(18,13)--(18,14)--(17,15)--(17,16)--(17,17)--(18,18); 
			\draw[ultra thick, dashed, brown] (3,2)--(3,0);
			\node[draw,purple, circle,inner sep=1.5pt,fill] at (22, 1) {};  \node[draw,purple, circle,inner sep=1.5pt,fill] at (25, 5) {}; 
			\node[draw,purple, circle,inner sep=1.5pt,fill] at (25, 4) {}; \node[draw,purple, circle,inner sep=1.5pt,fill] at (25, 3) {}; 
			\node[draw,purple, circle,inner sep=1.5pt,fill] at (23, 2) {};
			\node[draw,purple, circle,inner sep=1.5pt,fill] at (20, 15) {};   \node[draw,purple, circle,inner sep=1.5pt,fill] at (21, 19) {};
			\node at (21,19.5) {$y_1\!-\!1$};  \node at (19,14.5) {$\mu_1\!-\!1$};
			\node at (22,0.5) {$y_1'-1$}; \node at (26.3 ,4.5) {$\lambda_1\!-\!1$};
			\draw[ultra thick, purple]  (4,1)--(22,1)--(22,2)--(23,2)--(23,3)--(25,3)--(25,5)--(24,6)--(23,7)--(23,8)--(22,9)--(21,10)--(20,11)--(20,12)--(20,13)--(20,14)--(20,15)--(20,16)--(21,17)--(21,19);
			\draw[ultra thick, dashed, purple] (4,1)--(4,0);
			\node at (0.5,18.5) {$-q_2$}; \node at (0.5,17.5) {$-q_3$}; \node at (0.5,16.5) {$-q_4$}; \node at (0.5,15.5) {$-q_5$};
			\node at (0.5,14.5) {$p_{r+1}$}; \node at (0.5,13.5) {$p_{r+2}$}; \node at (0.5,6.5) {$p_{t-1}$}; \node at (0.5,5.5) {$p_t$};
			\node at (0.5,12.5) {$\vdots$}; \node at (0.5,11.5) {$\vdots$}; \node at (0.5,10.5) {$\vdots$};   \node at (0.5,9.5) {$\vdots$}; \node at (0.5,8.5) {$\vdots$}; \node at (0.5,7.5) {$\vdots$}; 
			\node at (-1,5) {$q_5$}; \node at (-1,4) {$q_4$}; \node at (-1,3) {$q_3$}; \node at (-1,2) {$q_2$}; \node at (-1,1) {$q_1$};
			\node at (-1,0) {$0$}; \node at (4,-0.8) {$x^{(0)}_1$}; \node at (0,-0.8) {$x^{(0)}_N$};
			\node at (2,-0.8) {$\cdots\cdots$};
\foreach \j in {1,...,4}{
			\node[draw,red,circle,inner sep=1.5pt,fill] at (0, \j) {};}
\foreach \j in {1,...,3}{
			\node[draw,blue,circle,inner sep=1.5pt,fill] at (1, \j) {};}
\foreach \j in {1,...,2}{
			\node[draw,orange, circle,inner sep=1.5pt,fill] at (2, \j) {};}
\foreach \j in {1,...,1}{
			\node[draw,brown, circle,inner sep=1.5pt,fill] at (3, \j) {};}
\end{tikzpicture}
\caption{Non-intersecting path representation of the transition kernel of $\dtasep$ particles, as in~\eqref{eq:pathQ2}.
The figure refers to the transition probability of five $\dtasep$ particles, starting from locations $(y_1-1>y_2-2>\dots>y_5-5)$ at time $r$ and ending at locations $(y_1'-1>y_2'-2>\dots>y_5'-5)$ at time $t$. 
Note that the paths do \emph{not} depict the actual trajectories of the particles.
The weights assigned to all vertical steps are equal to $1$.
In the bottom part of the figure, paths move either vertically up or horizontally to the right, with horizontal weights $q_1,\dots,q_5$, as shown on the left-hand side of the figure.
In the middle part, paths move vertically up or diagonally up-left, with diagonal weights $p_t, p_{t-1}, \dots, p_{r+2},p_{r+1}$, as shown.
In the top part, paths move either vertically up or diagonally up-right, with diagonal weights $q_5,\dots,q_2$, as shown.
The solid paths can be extended in such a way that they all start at level zero and include the dashed colored lines in the bottom-left part of the picture; due to the non-intersecting property and the fact that vertical weights are assigned weight $1$, such an extension does not change the weight of the ensemble (provided that the horizontal edges on the bottom level are assigned weight $0$).
The bullets in the bottom part of the figure refer to the point processes $\sfX_N$ and $\overline{\sfX}_N$ from~\S\ref{subsec:DPP}.}
\label{fig:paths}
\end{figure}
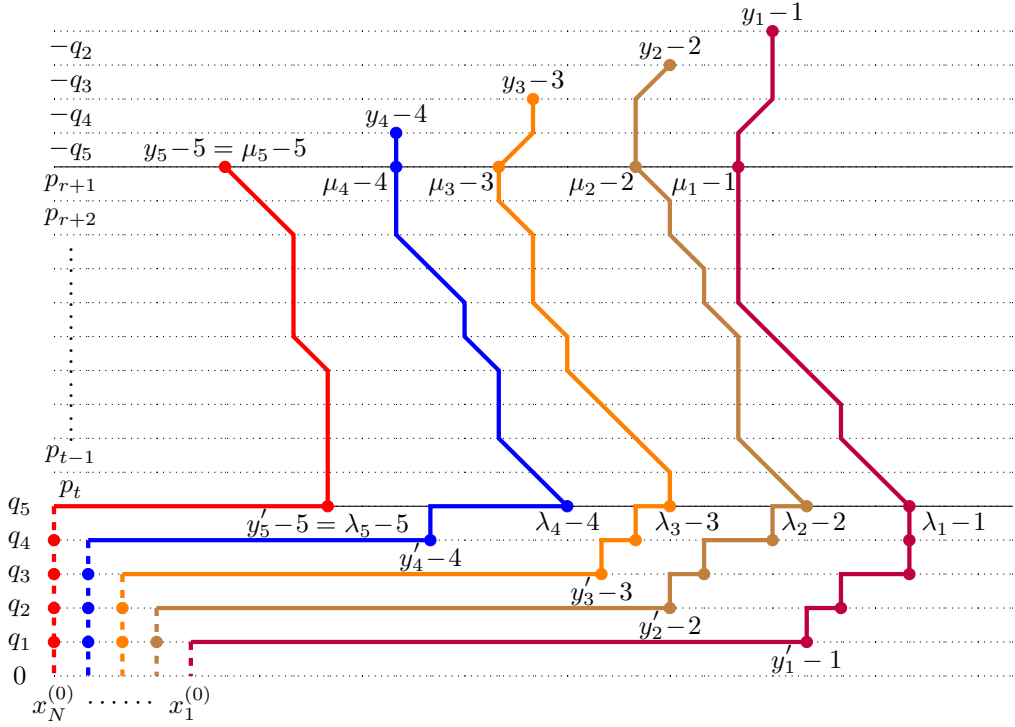

\begin{proof}
By Proposition~\ref{prop:intertwining}, we have
\begin{equation}\label{eq:Q_threeParts}
\begin{split}
\cQ_{(r,t]}(\bm{y},\bm{y}') 
&= \Bigg( \prod_{i=1}^N q_i^{y_i'-y_i}\Bigg)
\big(\Lambda^{-1} \cR_{(r,t]}  \Lambda \big)(\bm{y}, \bm{y}') \\
&= \Bigg( \prod_{i=1}^N q_i^{y_i'-y_i}\Bigg)
\sum_{\lambda, \mu}  \Lambda^{-1}(\bm{y},\mu) \, \cR_{(r,t]}(\mu,\lambda) \, \Lambda(\lambda,\bm{y}'),
\end{split}
\end{equation}
where the summation is over all partitions $\lambda,\mu\in \sfW_N$ such that $\mu \subseteq\lambda$, $\mu\subseteq \bm{y}$ and $\bm{y}' \subseteq \lambda$.
We now concatenate the non-intersecting paths corresponding to the operators $\Lambda^{-1}$, $\cR_{(r,t]}$ and $\Lambda$, as given in Propositions~\ref{prop:pathLambdaInv}, \ref{prop:pathR} and \ref{prop:pathLambda}, respectively, and as shown in Figure~\ref{fig:paths}.
The only point to notice is that, compared to Proposition~\ref{prop:pathR}, the paths corresponding to the operator $\cR_{(r,t]}$ are flipped upside down, resulting in the `reverse' ensemble $({\sfe\Pi}^*_r)_{\{(\lambda_i-i,N) \colon 1\leq i\leq N\}}^{\{(\mu_i-i,N+t-r) \colon 1\leq i\leq N\}}$, in which paths move either upwards or in the up-left direction at each step, starting from the points $\lambda_i-i$, $1\leq i\leq N$, at level $N$ and ending at the points $\mu_i-i$, $1\leq i\leq N$, at level $N+t-r$.
Thus, using equations~\eqref{eq:pathLambda1}-\eqref{eq:pathLambdaInv1}-\eqref{eq:pathR1} and the weights defined in the proposition, we obtain
\[
\begin{split}
\cQ_{(r,t]}(\bm{y},\bm{y}'\, ) 
&= \frac{\big( \prod_{i=1}^N q_i^{y'_i-y_i}\big)}{Z^{\bm{p},\bm{q}}_{(r,t]}}  
\sum_{\lambda,\mu } 
\sfw\sft \Big( \sfh\Pi_{\{(y'_i-i,i) \colon 1\leq i\leq N\}, \uparrow}^{\{(\lambda_i-i,N) \colon 1\leq i\leq N\}} \Big) \cdot \\
&\quad\, \cdot \sfw\sft\Big( (\sfe\Pi_r^*)_{\{(\lambda_i-i,N) \colon 1\leq i\leq N\}}^{\{(\mu_i-i,N+t-r) \colon 1\leq i\leq N\}} \Big)
\cdot \sfw\sft\Big(\sfe\Pi^{\{(y_i-i,2N+t-r-i) \colon 1\leq i\leq N\}}_{\{(\mu_i-i,N+t-r) \colon i=1,\dots,N\}}  \Big),
\end{split}
\]
which is the same as~\eqref{eq:pathQ1}.

We now rewrite~\eqref{eq:pathQ1}, by multiplying and dividing by {$\prod_{i=1}^N q_i^{-x_0}$}, as
\[
\begin{split}
\cQ_{(r,t]}(\bm{y},\bm{y}') 
&=  \frac{\prod_{i=1}^N q_i^{y'_i-x_0}}{Z^{\bm{p},\bm{q}}_{(r,t]} \prod_{i=1}^N q_i^{y_i-x_0} } 
\sum_{\lambda, \mu} \sfw\sft \bigg(\sfh\Pi_{\{(y'_i-i,i) \colon 1\leq i\leq N\}, \uparrow}^{\{(\lambda_i-i,N) \colon 1\leq i\leq N\}} \,\bigsqcup \\
&\quad\, \bigsqcup \, (\sfe\Pi_r^*)_{\{(\lambda_i-i,N) \colon 1\leq i\leq N\}}^{\{(\mu_i-i,N+t-r) \colon 1\leq i\leq N\}}
\,\bigsqcup\,\sfe\Pi^{\{(y_i-i,2N+t-r-i) \colon 1\leq i\leq N\}}_{\{(\mu_i-i,N+t-r) \colon 1\leq i\leq N\}}  \bigg).
\end{split}
\]
The denominator equals $\widehat Z^{\bm{p},\bm{q}}_{(r,t]}$.
Moreover, the product {$\prod_{i=1}^N q_i^{y'_i-x_0}=\prod_{i=1}^N q_i^{(y'_i-i)-x^{(0)}_i}$} equals the total weight of the path ensemble ${\sfh\Pi}^{\{(y'_i-i,i) \colon 1\leq i\leq N\}}_{\{(x_{i}^{(0)},i) \colon 1\leq i\leq N \}}$, {since the $h$-weight of a single path from $(x_i^{(0)},i)$ to $(y_i'-i,i)$ (see bottom-left end of the path depictions in Figure~\ref{fig:paths}) is $q_i^{(y'_i-i)-x^{(0)}_i}$.}
Notice that such a path ensemble is nonempty, due to the hypothesis $x_0-1< y_N-N$, which implies
\[
x_i^{(0)}=x_0-i< y_N-N\leq y_N'-N\leq y'_i-i
\]
for $1\leq i\leq N$.
This readily yields~\eqref{eq:pathQ2}.
\end{proof}

\subsection{Determinantal point processes}
\label{subsec:DPP}
	
We now describe the law of $\dtasep(N;\bm{p},\bm{q})$ as a marginal of a determinantal point process, which we now build out of the above path construction.
The configuration space will be integer arrays 
\begin{align}\label{eq:triangArray}
\sfX_N:= \{x^{(i)}_j\}_{ 1 \leq j\leq i \leq N} \qquad \text{with}\qquad  x^{(i+1)}_{j+1} < x^{(i)}_j\leq x^{(i+1)}_j.
\end{align}
Such a point process naturally arises from the non-intersecting path construction given in Proposition~\ref{prop:pathTASEP} (see in particular~\eqref{eq:pathQ2}) and illustrated in the bottom part of Figure~\ref{fig:paths}.
For $1\leq j\leq i\leq N$, the point $x_j^{(i)}$ of $\sfX_N$ is identified with the rightmost point on the horizontal line $\{(x,i)\colon x\in\Z\}$ of the $j$-th path (enumerating the paths from right to left).
We will consider~\eqref{eq:triangArray} as a point process on $\Z\times \{1,\dots,N\}$, such that the line $\{i\}\times \Z$ has exactly $i$ points $x^{(i)}_1,\dots,x^{(i)}_i$, for $1\leq i\leq N$.
However, for brevity and when there is no ambiguity, we will usually write 	$x^{(i)}_j$ instead of $(x^{(i)}_j,i)$.
We note that the non-intersecting property of the paths enforces the inequalities in~\eqref{eq:triangArray}.
It will be useful to extend the above triangular array to a square array with additional frozen points on every line:
\begin{equation}\label{eq:triangArrayExt}
\begin{split}
\overline{\sfX}_N:= \{&x^{(i)}_j\}_{i,j=1}^ {N} \qquad \text{with}\qquad  x^{(i+1)}_{j+1} < x^{(i)}_j\leq x^{(i+1)}_j
\qquad\text{and} \\
& x^{(i)}_j := (x_0-j, i) \qquad \text{for} \qquad 1\leq i<j\leq N.
\end{split}
\end{equation}
The auxiliary points $x^{(i)}_j$ with $ 1\leq i<j\leq N$ are illustrated as the `frozen' bullets in the bottom-left part of Figure~\ref{fig:paths}.
As in Proposition~\ref{prop:pathTASEP}, the point $x_0$ is chosen arbitrarily but	such that $x_0 -1< y_N-N$.
We will also use the notation $\bm{x}^{(0)}:=(x^{(0)}_1,x^{(0)}_2,\dots,x^{(0)}_N)$, with $x^{(0)}_i:=x_0-i$, for $i=1,\dots,N$.
The freezing is, again, due to the non-intersecting nature of the extended paths (i.e., the paths that start from $\{(x^{(0)}_i,0)\colon 1\leq i\leq N\}$ and include the dashed lines) in Figure~\ref{fig:paths}.
	
For $0\leq r<t$, $1\leq k\leq N$ and $x\in\Z$, we now define the family of functions
\begin{equation}
\label{eq:Psi_def}
\begin{split}
\Psi^{(N)}_{k} (x)
&:=\sum_{z\in\Z} Q_N^{-1}\circ\cdots\circ Q_{k+1}^{-1}(z,y_k-k) \, R_{r+1}\circ\cdots\circ R_{t}(z,x) \\
&=\sum_{z\in\Z} Q_N^{-1}\circ\cdots\circ Q_{k+1}^{-1}(z,y_k-k) \, R^*_{t}\circ\cdots\circ R^*_{r+1}(x,z) \\
&= R^*_{t}\circ\cdots\circ R^*_{r+1}\circ  Q_N^{-1}\circ\cdots\circ Q_{k+1} ^{-1}(x,y_k-k)  \\
&= R^*_{(r,t]}\circ Q^{-1}_{(k,N]} (x,y_k-k) ,
\end{split}
\end{equation}
where, as usual, $R^{*}_{(r,t]}$ denotes the adjoint of $R_{(r,t]}$, i.e.\ $R^{*}_{(r,t]}(x,y):= R_{(r,t]}(y,x)$.
Notice that the function $\Psi^{(N)}_{k} (x)$ depends implicitly on $\bm{y}=(y_1\geq \dots \geq y_N)$.
With reference to Figure~\ref{fig:paths}, {$(-1)^{N-k}\Psi^{(N)}_{k} (x)$} captures the weight of a path starting at location $(x,N)$ on the lower solid black line and ending at $(y_k-k, 2N+t-r-k)$ at the top of the figure (passing by any $z$ on the upper solid black line).
{Note also that, even though the sums in the first two lines of~\eqref{eq:Psi_def} are over the $\Z$, these sums actually have only a finite number of non-zero terms, since the values of $x$ and $y_k-k$ are fixed.}

Observe that the left edge of the triangular array $\sfX_N$, i.e.\ $\ledge( \sfX_N ):=(x^{(i)}_i\colon 1\leq i\leq N)$, coincides with the terminal positions $(y_1'-1>y_2'-2>\dots>y_N'-N)$ of the $\dtasep(N;\bm{p},\bm{q})$ particles, by Proposition~\ref{prop:pathTASEP}.
Thanks to this, we are able to obtain an initial representation of the transition kernel of $\dtasep(N;\bm{p},\bm{q})$ as	a marginal of the determinantal point process $\sfX_N$.
\begin{proposition}
\label{prop:DPPmarginal}
Let $\bm{y}\in\sfW_N$.
Let $x_0\in\Z$ with $x_0 -1< y_N-N$ and write $\bm{x}^{(0)}:=(x^{(0)}_1,x^{(0)}_2,\dots,x^{(0)}_N)$, where $x^{(0)}_i:=x_0-i$.
Define the {(signed)} determinantal measure
\begin{equation}\label{eq:DPP}
\begin{split}
\bbP\big( \sfX_N {\;\big|\;\bm{y}} \big)
:= \frac{1}{\widehat Z^{\bm{p},\bm{q}}_{(r,t]}}
&\left(\prod_{k=1}^{N} \det \Big( Q^\dagger_{k}(x^{(k-1)}_i,x^{(k)}_j) \Big)_{i,j=1}^{k} \right) \\
&\cdot \det\Big((-1)^{N-i} \,\Psi^{(N)}_{i} (x^{(N)}_j) \Big)_{i,j=1}^{N}
\end{split}
\end{equation}
on configurations $\sfX_N= \{x^{(i)}_j\}_{ 1 \leq j\leq i \leq N} $ with $x^{(i+1)}_{j+1} < x^{(i)}_j\leq x^{(i+1)}_j$, where we have set  $x^{(k-1)}_k=x^{(0)}_k$ for $1\leq k\leq N$, {$\widehat Z^{\bm{p},\bm{q}}_{(r,t]}:=Z^{\bm{p},\bm{q}}_{(r,t]} \prod_{i=1}^N q_i^{y_i-x_0}$} and $Z^{\bm{p},\bm{q}}_{(r,t]} $ as in~\eqref{eq:normalization} {(note that the right-hand side of~\eqref{eq:DPP} depends on $\bm{y}$ through the $\Psi$-functions).}
The determinantal measure does not depend on the auxiliary point $x_0$, as long as the condition $x_0-1<y_N-N$ holds.
Then, the transition kernel of $\dtasep(N;\bm{p},\bm{q})$, encoding the probability that particles starting from locations $(y_1-1>y_2-2>\dots>y_N-N)$ at time $r$ end up at locations $(y_1'-1>y_2'-2>\dots>y_N'-N)$ at time $t$, is given by
\begin{align}\label{eq:Q_marginalDPP}
&\cQ_{(r,t]}(\bm{y},\bm{y}') 
= \bbP\big( \ledge\big( \sfX_N \big) =(y_1'-1>y_2'-2>\dots>y_N'-N) \;\big|\; \bm{y}\big).
\end{align}
\end{proposition}
\begin{proof} 
In a nutshell, this result is a consequence of the path representation of $\cQ_{(r,t]}(\bm{y},\bm{y}')$, as described in Proposition~\ref{prop:pathTASEP} (see in particular~\eqref{eq:pathQ2}), as well as the identification of the point process $\sfX_N$ as the `trace' of that path ensemble on $\{(x,i)\colon x\in \Z, \, 1\leq i\leq N \}$.
Starting from~\eqref{eq:Q_threeParts} and using~\eqref{eq:pathLambdaInv3}, \eqref{eq:pathR3} and~\eqref{eq:pathLambda3}, we obtain
\begin{align*}
\cQ_{(r,t]}(\bm{y},\bm{y}') 
&=\Bigg( \prod_{i=1}^N q_i^{y'_i-y_i}\Bigg)
\sum_{\lambda, \mu}  \Lambda^{-1} (\bm{y},\mu) \,
\cR_{(r,t]} (\mu,\lambda) \,
\Lambda(\lambda,\bm{y}')\\
&=\frac{\big( \prod_{i=1}^N q_i^{y'_i-y_i}\big)}{Z^{\bm{p},\bm{q}}_{(r,t]}}
\sum_{\lambda, \mu} \det\Big( (-1)^{N-j} Q_{(j,N]}^{-1}\big(\mu_i-i, y_j-j \big)  \Big)_{ i,j =1}^N\\
&\quad\, \cdot \det\Big( R_{(r,t]} (\mu_i-i,\lambda_j-j)\Big)_{ i,j =1}^N 
\det\Big( Q^\dagger_{(i,N]}(y'_i-i, \lambda_j-j)\Big)_{ i,j =1}^N,
\end{align*}
where the summations are over all $\lambda,\mu\in \sfW_N$ such that $\mu \subseteq\lambda$, $\mu\subseteq \bm{y}$ and $\bm{y}' \subseteq \lambda$.
Using the Cauchy--Binet identity to compute the sum over $\mu$, we have
\begin{align*}
&\cQ_{(r,t]}(\bm{y},\bm{y}') 
=\frac{\big( \prod_{i=1}^N q_i^{y'_i-y_i}\big)}{Z^{\bm{p},\bm{q}}_{(r,t]}} 
\sum_{\lambda \supseteq \bm{y}'}  
\det \Big( (-1)^{N-j} \sum_{z\in\Z} Q_{(j,N]}^{-1}\big(z, y_j-j \big) R_{(r,t]} (z,\lambda_i-i) \Big)_{ i,j =1}^N \\
&\hskip 8.95cm \cdot \det \Big( Q^\dagger_{(i,N]}(y_i'-i, \lambda_j-j)\Big)_{ i,j =1}^N \\
			&=
			\frac{\big( \prod_{i=1}^N q_i^{y'_i-y_i}\big)}{Z^{\bm{p},\bm{q}}_{(r,t]}} \sum_{\lambda \supseteq \bm{y}'}  
			\det \Big((-1)^{N-j} \Psi^{(N)}_{j} (\lambda_i-i) \Big)_{ i,j =1}^N
			\det \Big( Q^\dagger_{(i,N]}(y'_i-i, \lambda_j-j)\,\Big)_{ i,j =1}^N,
\end{align*}
where the latter equality follows from definition~\eqref{eq:Psi_def}.
We next multiply and divide by {$\prod_{i=1}^N q_i^{-x_0}$}, recall that {$\widehat Z^{\bm{p},\bm{q}}_{(r,t]}=\big(\prod_{i=1}^N q_i^{y_i-x_0} \big)Z^{\bm{p},\bm{q}}_{(r,t]}$} and absorb the factor {$\prod_{i=1}^N q_i^{y_i'-x_0}$} into the second determinant, thus obtaining
\begin{align*}
\cQ_{(r,t]}(\bm{y},\bm{y}') 
=\frac{ 1}{\widehat Z^{\bm{p},\bm{q}}_{(r,t]}} \sum_{\lambda \supseteq \bm{y}'}  
&\det \Big((-1)^{N-i} \Psi^{(N)}_{i} (\lambda_j-j) \Big)_{i,j=1}^N \\
&\cdot \det \Big({q_i^{y_i'-x_0}}  Q^\dagger_{(i,N]}(y'_i-i, \lambda_j-j)\Big)_{ i,j =1}^N.
\end{align*}
Using the facts that $Q^\dagger_{(0,i-1]}(x^{(0)}_i,x^{(0)}_i )=1$ {and $q_i^{y_i'-x_0}=Q^\dagger_i(x^{(0)}_i,y_i'-i)$}, we can rewrite
\[
\begin{split}
{q_i^{y_i'-x_0}} Q^\dagger_{(i,N]}(y'_i-i, \lambda_j-j) 
&=Q^\dagger_{(0,i-1]}(x^{(0)}_i,x^{(0)}_i ) \,
Q^\dagger_{(i-1,i]}(x^{(0)}_i,y'_i-i ) \,
Q^\dagger_{(i,N]}(y'_i-i, \lambda_j-j) \\
&=\sumthree{(z_0,\dots,z_N)\in\Z^{N+1}\colon}
{z_0=z_1=\dots=z_{i-1}=x^{(0)}_i,}
{z_i=y'_i-i,\;\; z_{N}= \lambda_j-j}
\prod_{k=1}^N Q^\dagger_{k}(z_{k-1},z_{k}) .
\end{split}
\]
Using {several times} the Cauchy--Binet identity, we obtain
\[
\begin{split}
\cQ_{(r,t]}(\bm{y},\bm{y}') 
=\frac{1}{\widehat Z^{\bm{p},\bm{q}}_{(r,t]}}
\sum_{\lambda \supseteq \bm{y}'} \;
\sumthree{\overline{\sfX}_N\colon x^{(i-1)}_i=\cdots =x^{(1)}_{i}=x^{(0)}_i,}
{x^{(i)}_i=y'_i-i,\;\; x^{(N)}_i= \lambda_i-i}
{\text{for } 1\leq i\leq N}
&\det \Big((-1)^{N-i} \Psi^{(N)}_{i} (\lambda_j-j) \Big)_{i,j=1}^N \\
&\cdot\prod_{k=1}^N\det \Big(Q^\dagger_{k}(x^{(k-1)}_i,x^{(k)}_j ) \Big)_{i,j=1}^N.
\end{split}
\]
Notice that the constraints $x^{(i-1)}_i=\cdots =x^{(1)}_{i}=x^{(0)}_i$, $1\leq i\leq N$, corresponds to `freezing' the points $x^{(i)}_j$, $1\leq i<j\leq N$, as in~\eqref{eq:triangArrayExt}.
Due to the inequalities~\eqref{eq:triangArrayExt} that $\overline{\sfX}_N$ satisfies, we have $x^{(k)}_{j}<x^{(k-1)}_i$ for $i<j$, hence $Q^\dagger_{k}(x^{(k-1)}_i,x^{(k)}_j)=0$ for $i<j$.
Furthermore, due the constraints $x^{(i-1)}_i=\cdots =x^{(1)}_{i}=x^{(0)}_i$, we have $Q^\dagger_{k}(x^{(k-1)}_i,x^{(k)}_i )=1$ for $i>k$.
Therefore, the matrix $\big(Q^\dagger_{k}(x^{(k-1)}_i,x^{(k)}_j ) \big)_{i,j=1}^N$ is lower triangular with diagonal elements $Q^\dagger_{k}(x^{(k-1)}_i,x^{(k)}_i )=1$ for $i>k$. This implies 
\begin{align*}
\det \Big(Q^\dagger_{k}(x^{(k-1)}_i,x^{(k)}_j ) \Big)_{i,j=1}^N = 
\det \Big(Q^\dagger_{k}(x^{(k-1)}_i,x^{(k)}_j ) \Big)_{i,j=1}^k,
\end{align*}
which leads to 
\[
\begin{split}
\cQ_{(r,t]}(\bm{y},\bm{y}') 
=\frac{1}{\widehat Z^{\bm{p},\bm{q}}_{(r,t]}}
\;\sumtwo{\sfX_N\colon x^{(i)}_i=y'_i-i}
{\text{for } 1\leq i\leq N}
&\det \Big((-1)^{N-i} \Psi^{(N)}_{i} (x^{(N)}_j) \Big)_{i,j=1}^N \\
&\cdot \prod_{k=1}^N\det \Big(Q^\dagger_{k}(x^{(k-1)}_i,x^{(k)}_j ) \Big)_{i,j=1}^k,
\end{split}
\]
with the convention that $x^{(k-1)}_k=x^{(0)}_k$ for $1\leq k\leq N$.
This completes the proof {of~\eqref{eq:Q_marginalDPP}.}

{To show that the determinantal measure does not in fact depend on $x_0$, notice first that the $k$-th row of the matrix $Q^\dagger_{k}(x^{(k-1)}_i,x^{(k)}_j )$ has a common factor $q^{-x_0}$, for all $1\leq k\leq N$.
When factoring these terms out of the determinants, we see a cancellation with the corresponding terms in the normalizing constant $\hat{Z}^{\bm{p},\bm{q}}_{(r,t]}$.
The resulting expression has no further dependence on $x_0$.
}
\end{proof}

For the analysis that will follow in the next sections, it will be convenient to re-express the determinantal measure~\eqref{eq:DPP} in terms $Q$-operators, which represent weights of paths moving strictly to the left, rather than $Q^\dagger$-operators, which represent weights of paths moving weakly to the right. 
Towards this, the main observation is the following equality of determinants, which will be a consequence of certain path constructions. 
	
\begin{proposition}\label{prop:detEquality}
Let $\{x_j^{(i)}\}_{1\leq j\leq i\leq N}$ be a triangular array of integers and let $\{x_j^{(j-1)}\}_{j=1}^N$ and $\{x_0^{(j-1)}\}_{j=1}^N$ be auxiliary integer variables.
Assume that $\{x_j^{(i)}\}_{1\leq j\leq i\leq N}$ satisfies 
\begin{align}\label{eq:interlacingStrict}
x_{j}^{(i)}<x_{j-1}^{(i-1)}\leq x_{j-1}^{(i)}\qquad \text{for all} \qquad 1\leq i\leq N,\;  {1\leq j\leq i+1}.
\end{align}
Then we have
\begin{equation}
\label{eq:detEquality}
\begin{split}
\prod_{k=1}^{N}\left({q_k^{x_{k}^{(k-1)}}}
\det\left(Q^\dagger_k(x_{i}^{(k-1)},x_j^{(k)})\right)_{i,j=1}^{k}\right) 
= \prod_{k=1}^{N}\left({q_k^{x_0^{(k-1)}}}
\det\left( Q_k(x_{i-1}^{(k-1)},x_j^{(k)})\right)_{i,j=1}^{k}\right),
\end{split}
\end{equation}
and both sides are nonzero.
\end{proposition}
	
\begin{proof}
By the Lindstr\"om--Gessel--Viennot theorem, we may view $\det\big(Q^\dagger_k(x_{i}^{(k-1)},x_j^{(k)})\big)_{i,j=1}^{k}$ as the total weight of $k$ non-intersecting paths starting from $(x_k^{(k-1)},k-1),\dots,(x_{1}^{(k-1)},k-1)$ and ending at $(x_k^{(k)},k),\dots,(x_{1}^{(k)},k)$, with the first step upwards {(with weight $1$) and subsequent steps in the horizontal right direction (with weight $q_k$ at each step)}.
Note that such a non-intersecting path ensemble exists if and only if 
\begin{equation*}
x_{j}^{(k)}<x_{j-1}^{(k-1)}\leq x_{j-1}^{(k)}\qquad \text{for all}\qquad 2\leq j\leq k+1.
\end{equation*}
In such a case, the weight of the ensemble equals $\prod_{j=1}^{k}Q^\dagger_k(x^{(k-1)}_j,x^{(k)}_j)$.
Taking the product over $k=1,\dots,N$, we obtain the total weight of the non-intersecting paths illustrated in red in Figure~\ref{fig:redBluePaths}:
\begin{align*}
&\quad\,\prod_{k=1}^N\det\left(Q^\dagger_k(x_{i}^{(k-1)},x_j^{(k)})\right)_{i,j=1}^{k} \\
&= \begin{cases}
\prod_{k=1}^{N}\prod_{j=1}^{k}Q^\dagger_k(x_j^{(k-1)},x_j^{(k)})\quad &\text{if }	x_{j}^{(k)}<x_{j-1}^{(k-1)}\leq x_{j-1}^{(k)}\text{ for } 1\leq k\leq N,\ 2\leq j\leq k+1,\\
0 &\text{otherwise}.
\end{cases}
\end{align*}

{We now consider similar, `dual' paths, illustrated in blue in Figure~\ref{fig:redBluePaths}.
Recalling from~\eqref{eq:Q} that the $Q$-operators encode geometric jumps strictly to the left, by the Lindstr\"om--Gessel--Viennot theorem we view $\det\big( Q_k(x_{i-1}^{(k-1)},x_j^{(k)})\big)_{i,j=1}^{k}$ as the total weight of $k$ non-intersecting paths starting from $(x_{k-1}^{(k-1)},k-1),\dots,(x_{0}^{(k-1)},k-1)$ and ending at $(x_k^{(k)},k),\dots,(x_{1}^{(k)},k)$, with the first step diagonally up-left and subsequent steps in the horizontal left direction, with \emph{all} steps having weight $q_k$.
Reasoning as before, we obtain}
\begin{align*}
&\quad\,\prod_{k=1}^N\det\left( Q_k(x_{i-1}^{(k-1)},x_j^{(k)})\right)_{i,j=1}^{k} \\
&= \begin{cases}
\prod_{k=1}^{N}\prod_{j=1}^{k} Q_k(x_{j-1}^{(k-1)},x_j^{(k)})\quad &\text{if }	x_{j}^{(k)}<x_{j-1}^{(k-1)}\leq x_{j-1}^{(k)}\text{ for } 1\leq k\leq N,\ 1\leq j\leq k,\\
0 &\text{otherwise}.
\end{cases}
\end{align*}
{Thus, since by hypothesis the interlacing conditions~\eqref{eq:interlacingStrict} are satisfied, both sides of~\eqref{eq:detEquality} are nonzero.}
Moreover, we have 
\begin{equation}\label{eq:redBluePaths}
\begin{aligned}
&\prod_{k=1}^N\det\left(Q^\dagger_k(x_{i}^{(k-1)},x_j^{(k)})\right)_{i,j=1}^{k}= \sfw\sft(\text{red paths}),\\
&\prod_{k=1}^N\det\left( Q_k(x_{i-1}^{(k-1)},x_j^{(k)})\right)_{i,j=1}^{k}= \sfw\sft(\text{blue paths}).\\
\end{aligned}
\end{equation}
Moreover, up to inverting the weights of the blue paths, the total weight of red and blue paths simply equals the weight of all the horizontal paths from $(x_{k}^{(k-1)}, k)$ to $(x_0^{(k-1)},k)$, $k=1,\cdots,N$; in other words, we have
\begin{equation}\label{eq:concatenateBlueRedPaths}
\sfw\sft(\text{red paths})\cdot \sfw\sft(\text{blue paths})^{-1}= \prod_{k=1}^{N}q_k^{x_0^{(k-1)}-x^{(k-1)}_k}.
\end{equation}
Combining~\eqref{eq:concatenateBlueRedPaths} with~\eqref{eq:redBluePaths}, we readily arrive at~\eqref{eq:detEquality}.
\end{proof}

	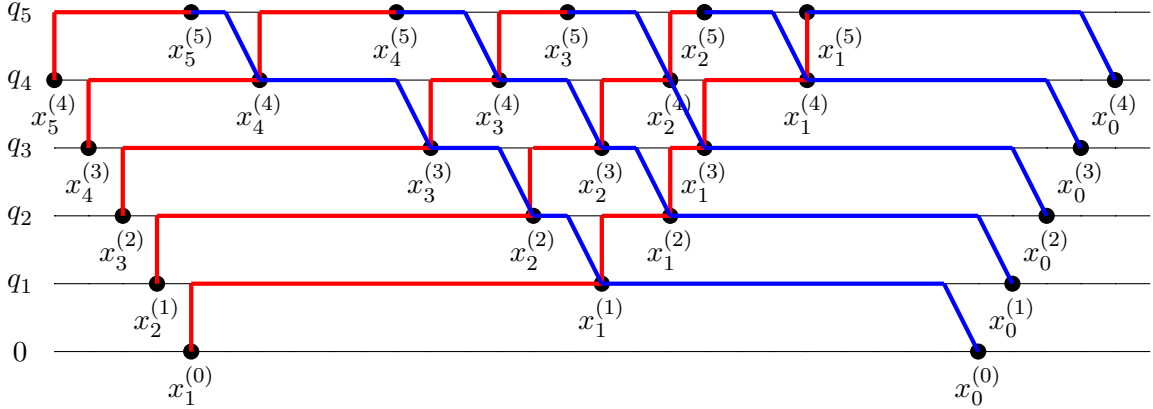
\begin{figure}
		\begin{tikzpicture}[scale=0.45]
			\foreach \i in {0,1,...,5}{
				\draw (0, 2*\i  ) grid (32, 2*\i );
			}
			\foreach \i in {0,...,4}{
				\node[draw,circle,inner sep=2pt,fill] at (4-\i, 2*\i) {};}
			\node[draw,black, circle,inner sep=2pt,fill] at (4, 10) {}; 
			\node at (4,9) {$x_5^{(5)}$}; \node at (0,7) {$x_5^{(4)}$};
			\draw[ultra thick, red]  (0,10)--(4,10);
			\draw[ultra thick,red] (0,8)--(0,10);
			\node[draw,black, circle,inner sep=2pt,fill] at (6, 8) {}; \node[draw,black, circle,inner sep=2pt,fill] at (10, 10) {};
			\node at (10,9) {$x_4^{(5)}$}; \node at (6,7) {$x_4^{(4)}$}; \node at (1,5) {$x_4^{(3)}$};
			\draw[ultra thick, red]  (1,8)--(6,8);\draw[ultra thick, red](6,10)--(10,10);
			\draw[ultra thick, red] (1,6)--(1,8); \draw[ultra thick, red](6,8)--(6,10);
			\node[draw,black, circle,inner sep=2pt,fill] at (11, 6) {};\node[draw,black, circle,inner sep=2pt,fill] at (13, 8) {}; 
			\node[draw,black, circle,inner sep=2pt,fill] at (15, 10) {};
			\node at (15,9) {$x_3^{(5)}$}; \node at (13,7) {$x_3^{(4)}$}; \node at (11,5) {$x_3^{(3)}$};\node at (2,3) {$x_3^{(2)}$};
			\draw[ultra thick, red]  (2,6)--(11,6); \draw[ultra thick, red](11,8)--(13,8);\draw[ultra thick, red](13,10)--(15,10);
			\draw[ultra thick, red](2,4)--(2,6);\draw[ultra thick, red](11,6)--(11,8);\draw[ultra thick, red](13,8)--(13,10);
			\node[draw,black, circle,inner sep=2pt,fill] at (18, 4) {}; 
			\node[draw,black, circle,inner sep=2pt,fill] at (14, 4) {};
			\node[draw,black, circle,inner sep=2pt,fill] at (18, 8) {}; \node[draw,black, circle,inner sep=2pt,fill] at (16, 6) {};  
			\node[draw,black, circle,inner sep=2pt,fill] at (19, 10) {}; 
			\node at (19,9) {$x_2^{(5)}$}; \node at (18,7) {$x_2^{(4)}$}; \node at (16,5) {$x_2^{(3)}$};\node at (14,3) {$x_2^{(2)}$};\node at (3,1) {$x_2^{(1)}$};
			\draw[ultra thick, red]  (3,4)--(14,4);\draw[ultra thick, red]  (14,6)--(16,6);\draw[ultra thick, red]  (16,8)--(18,8);\draw[ultra thick, red]  (18,10)--(19,10);
			\draw[ultra thick,red] (3,2)--(3,4);\draw[ultra thick,red] (13.9,4)--(13.9,6);\draw[ultra thick,red] (16,6)--(16,8);\draw[ultra thick,red] (18,8)--(18,10);
			\node[draw,black, circle,inner sep=2pt,fill] at (16, 2) {};  \node[draw,black, circle,inner sep=2pt,fill] at (19, 10) {}; 
			\node[draw,black, circle,inner sep=2pt,fill] at (22, 8) {}; \node[draw,black, circle,inner sep=2pt,fill] at (19, 6) {}; 
			\node[draw,black, circle,inner sep=2pt,fill] at (22, 10) {};
			\node at (23,9) {$x_1^{(5)}$}; \node at (22,7) {$x_1^{(4)}$}; \node at (19,5) {$x_1^{(3)}$};\node at (18,3) {$x_1^{(2)}$};\node at (16,1) {$x_1^{(1)}$};\node at (4,-1) {$x_1^{(0)}$};
			\draw[ultra thick, red]  (4,2)--(16,2);\draw[ultra thick, red]  (16,4)--(18,4);\draw[ultra thick, red]  (18,6)--(19,6);\draw[ultra thick, red]  (19,8)--(22,8);
			\draw[ultra thick,red] (4,0)--(4,2);\draw[ultra thick,red] (16,2)--(16,4);\draw[ultra thick,red] (18,4)--(18,6);\draw[ultra thick,red] (19,6)--(19,8);\draw[ultra thick,red] (22,8)--(22,10);
			\node at (-1,10) {$q_5$}; \node at (-1,8) {$q_4$}; \node at (-1,6) {$q_3$}; \node at (-1,4) {$q_2$}; \node at (-1,2) {$q_1$};
			\node at (-1,0) {$0$};
			\foreach \i in {0,...,4}{
				\node[draw,circle,inner sep=2pt,fill] at (27+\i, 2*\i) {};}
			\node at (31,7) {$x_0^{(4)}$}; 
			\draw[ultra thick, blue]  (4,10)--(5,10);\draw[ultra thick, blue]  (10,10)--(12,10);\draw[ultra thick, blue]  (15,10)--(17,10); \draw[ultra thick, blue]  (19,10)--(21,10);\draw[ultra thick, blue]  (22,10)--(30,10);
			\draw[ultra thick, blue] (6,8)--(5,10);\draw[ultra thick, blue] (13,8)--(12,10);\draw[ultra thick, blue] (18,8)--(17,10);\draw[ultra thick, blue] (22,8)--(21,10);\draw[ultra thick,ultra thick, blue] (31,8)--(30,10);
			\node at (30,5) {$x_0^{(3)}$}; 
			\draw[ultra thick, blue]  (6,8)--(10,8);\draw[ultra thick, blue]  (13,8)--(15,8);\draw[ultra thick, blue]  (22,8)--(29,8);
			\draw[ultra thick, blue] (11,6)--(10,8);\draw[ultra thick, blue] (16,6)--(15,8);\draw[ultra thick, blue] (19,6)--(18,8);\draw[ultra thick, blue] (30,6)--(29,8);
			\node at (29,3) {$x_0^{(2)}$}; 
			\draw[ultra thick, blue]  (11,6)--(13,6);\draw[ultra thick, blue]  (16,6)--(17,6);\draw[ultra thick, blue]  (19,6)--(28,6);
			\draw[ultra thick, blue] (14,4)--(13,6);\draw[ultra thick, blue] (18,4)--(17,6);\draw[ultra thick, blue] (29,4)--(28,6);
			\node at (28,1) {$x_0^{(1)}$}; 
			\draw[ultra thick, blue]  (14,4)--(15,4);\draw[ultra thick, blue]  (18,4)--(27,4);
			\draw[ultra thick, blue] (16,2)--(15,4);\draw[ultra thick, blue] (28,2)--(27,4);
			\node at (27,-1) {$x_0^{(0)}$}; 
			\draw[ultra thick, blue]  (16,2)--(26,2);
			\draw[ultra thick, blue] (27,0)--(26,2);
		\end{tikzpicture}
		\caption{{Ensemble of non-intersecting red and blue paths used in the proof of Proposition~\ref{prop:detEquality}.}}
		\label{fig:redBluePaths}
	\end{figure}
	
The following corollary re-expresses the determinantal measure~\eqref{eq:DPP} in a form that will be more suitable to our purposes.
\begin{corollary}\label{coro:modifiedDPP}
The determinantal measure~\eqref{eq:DPP} is equal to 
		\begin{equation}\label{eq:modifiedDPP}
			\bbP\big( \sfX_N {\;\big|\; \bm{y}}\big)
=\frac{1}{\tilde Z^{\bm{p},\bm{q}}_{(r,t]}} 
			\left(\prod_{k=1}^{N} \det \Big(  Q_{k}(x^{(k-1)}_{i-1},x^{(k)}_j) \Big)_{i,j=1}^{k}\right)
			\det\Big( \Psi^{(N)}_{i} (x^{(N)}_j) \Big)_{i,j=1}^{N},
		\end{equation}
with 
		\begin{equation}\label{eq:modifiedDPP-norm}
			\tilde{Z}^{p,q}_{(r,t]}:=(-1)^{N(N-1)/2}\prod_{i=r+1}^{t}\prod_{j=1}^{N} (1+p_iq_j)\cdot \prod_{j=1}^N q_j^{y_j-j-x_0^{(j-1)}}.
		\end{equation}
	\end{corollary}
	
	The determinantal measure~\eqref{eq:modifiedDPP} is actually independent of the auxiliary variables $\{x_0^{(j-1)}\}_{j=1}^N$. 
	{Reasoning as in the proof of Proposition~\ref{prop:DPPmarginal}, the first row of the matrix $ \big(  Q_{k}(x^{(k-1)}_{i-1},x^{(k)}_j) \big)_{i,j=1}^{k}$ has a common factor $q_k^{-x_0^{(k-1)}}$, for all $1\leq k\leq N$.
	These terms, when factored out of the determinants,} cancel the corresponding terms in the normalizing constant $\tilde{Z}^{p,q}_{(r,t]}$, making the resulting expression independent of $\{x_0^{(j-1)}\}_{j=1}^N$.
	Thus, we may set these auxiliary variables to be $\infty$ and simply define 
	\begin{equation}\label{eq:Q_infinityLim}
		Q_k(x_0^{(k-1)},y):=
		q_k^{y}=\lim_{x\to\infty} q_k^{x}\cdot  Q_k(x,y), \quad y\in \mathbb{Z}.
	\end{equation}
Using these conventions and defining
	\begin{equation}\label{eq:modifiedDPP-norm2}
		\bar{Z}^{p,q}_{(r,t]} :=(-1)^{N(N-1)/2}\prod_{j=1}^{N}\prod_{i=r+1}^{t}(1+q_jp_i)\cdot \prod_{j=1}^N
		q_j^{y_j-j},
	\end{equation}
	the determinantal measure~\eqref{eq:modifiedDPP} may be written in the form
	\begin{equation}\label{eq:modifiedDPP2}
		\bbP\big( \sfX_N {\;\big|\; \bm{y}}\big)
		= \frac{1}{\bar Z^{\bm{p},\bm{q}}_{(r,t]}} 
		\prod_{k=1}^{N} \det \Big(  Q_{k}(x^{(k-1)}_{i-1},x^{(k)}_j) \Big)_{i,j=1}^{k}
		\cdot \det\Big( \Psi^{(N)}_{i} (x^{(N)}_j) \Big)_{i,j=1}^{N}.
	\end{equation}

\subsection{Correlation kernel and biorthogonal functions}

\label{subsec:modified-kernel}
{From now on we will make the additional assumption that $q_1<q_2<\cdots$, throughout this subsection and most of \S\ref{sec:RWformulas}, until when we remove it in the proof of Theorem~\ref{thm:corker_hitting} (see~\S\ref{sec:RWformulas}).
This assumption allows for a \emph{bona fide} composition of the local operators and makes all infinite sums in this subsection and in the next section well defined; it also justifies swapping sums and exchanging sums with contour integrals.
To better understand its need,} recall that in~\eqref{eq:Q_infinityLim} we defined $Q_k(x_0^{(k-1)},y):=q_k^y$, for virtual variable $x_0^{(k-1)}$ regarded as $\infty$. 
This convention might seemingly lead to issues when defining
\begin{equation}\label{eq:Qjk_virtual}
	Q_{[j,k]}(x_0^{(j-1)},x):=  
	Q_j\circ  Q_{[j+1,k]}(x_0^{(j-1)},x)=\sum_{y\in \mathbb{Z}} q_j^{y}\cdot  Q_{[j+1,k]}(y,x)
\end{equation}
for $k>j$: if, for example, $k=j+1$, the sum above equals $\sum_{y>x}q_j^{y}\cdot q_{j+1}^{x-y}$, which diverges for $q_{j+1}\leq q_j$.
However, for $q_1<q_2<\cdots$, \eqref{eq:Qjk_virtual} is well defined.
To see this, recall from~\eqref{eq:Q_integral} that
\[
Q_{[j+1,k]}(x,y)=\oint_{|z|=r} \frac{\mathrm{d}z}{2\pi \i} \frac{z^{y-x+k-j-1}}{\prod_{\ell=j+1}^{k}(q_\ell-z)},\qquad x,y\in \mathbb{Z},
\]
where $0<r<\min\{q_\ell\}_{\ell=j+1}^k$. 
To get a well-defined expression for $Q_{[j,k]}(x_0^{(j-1)},y)$, note that, since $q_j<q_k$ for all $k>j$, we can write
\begin{equation*}
	\sum_{y\in \mathbb{Z}} q_j^{y}\cdot  Q_{[j+1,k]}(y,x)= \sum_{y<0}q_j^{y}\cdot \oint_{|z|=r}\frac{\mathrm{d}z}{2\pi i} \frac{z^{x-y+k-j-1}}{\prod_{\ell=j+1}^k(q_\ell-z)}+\sum_{y\geq 0}q_j^{y}\cdot \oint_{|z|=r'}\frac{\mathrm{d}z}{2\pi i} \frac{z^{x-y+k-j-1}}{\prod_{\ell=j+1}^k(q_\ell-z)},
\end{equation*}
where $r,r'$ are chosen such that $0<r<q_j<r'<\min\{q_\ell\}_{\ell=j+1}^k$.
Both geometric series converge and computing them yields
\begin{align}\label{eq:Qjk_virtual_CI}
	\sum_{y\in \mathbb{Z}} q_j^{y}\cdot  Q_{[j+1,k]}(y,x)&=  \oint_{|z|=r}\frac{\mathrm{d}z}{2\pi i} \frac{z^{x+k-j}}{\prod_{\ell=j}^k(q_\ell-z)}-\oint_{|z|=r'}\frac{\mathrm{d}z}{2\pi i} \frac{z^{x+k-j}}{\prod_{\ell=j}^k(q_\ell-z)}\nonumber\\&= -\oint_{\gamma_{q_j}}\frac{\mathrm{d}z}{2\pi i} \frac{z^{x+k-j}}{\prod_{\ell=j}^k(q_\ell-z)},
\end{align}
where $\gamma_{q_j}$ is any simple closed contour enclosing $q_j$ as the only pole for the integrand.
This computation motivates the definition
\begin{equation}\label{eq:Qjk-convention}
	Q_{[j,k]}(x_0^{(j-1)},x)
	:=-\oint_{\gamma_{q_j}}\frac{\mathrm{d}z}{2\pi i} \frac{z^{x+k-j}}{\prod_{\ell=j}^k(q_\ell-z)}= \frac{q_j^{x+k-j}}{\prod_{\ell=j+1}^{k}(q_\ell-q_j)}
	\qquad \text{for } 1\leq j\leq k,
\end{equation}	 
where $\gamma_{q_j}$ is the contour defined above.
{We also set
	\begin{equation}\label{eq:Qjk-convention0}
		Q_{[j,k]}(x_0^{(j-1)},y):=0\qquad \text{for } 1\leq k<j.
	\end{equation}}
We now summarize some basic properties of $Q_{[j,k]}(x_0^{(j-1)},x)$, which we will use frequently: 
\begin{proposition}\label{prop:properties_virtual}
Assume that $q_1<q_2<\cdots$ and take $Q_{[j,k]}(x_0^{(j-1)},x)$ to be defined as in~\eqref{eq:Qjk-convention}-\eqref{eq:Qjk-convention0}, where $x_0^{(j-1)}$ are virtual variables regarded as $\infty$. 
	\begin{enumerate}
		\item For $k\geq 1$ and $y\in \mathbb{Z}$, we have 
		\begin{equation}\label{eq:properties_virtual1}
			Q_{[k,k]}(x_0^{(k-1)},y) =  q_k^{y}=Q_k(x_0^{(k-1)},y),
		\end{equation}
	which is consistent with \eqref{eq:Q_infinityLim}. 
		\item Given $j,k,n\geq 1$ and $y\in \mathbb{Z}$ with $j<k$ and $j<n$, we have 
		\begin{equation}\label{eq:properties_virtual2}
			Q_{[j,k]}\circ Q_{(k,n]}(x_0^{(j-1)},y)
			:= \sum_{x'\in \mathbb{Z}}Q_{[j,k]}(x_0^{(j-1)},x')\,Q_{(k,n]}(x',y)
			=Q_{[j,n]}(x_0^{(j-1)},y),
		\end{equation}
	where we allow the slight abuse of notation~\eqref{eq:Q_compositionExt} for $k\geq n$.
		\item {For any $j,k\geq 1$}, we have
		\begin{equation}\label{eq:properties_virtual3}
			 Q_{[j,k]}(x_0^{(j-1)},y) = \lim_{x\to \infty} q_j^{x}\cdot Q_{[j,k]}(x,y).
		\end{equation}
	\end{enumerate}
\end{proposition}
\begin{proof}
	\begin{enumerate}
		\item This follows immediately from~\eqref{eq:Qjk-convention}.
		\item To prove this, one can mimic~\eqref{eq:Qjk_virtual_CI}, using the integral representation~\eqref{eq:Q_integral} and splitting the sum into two geometric series with contours deformed in such a way that both series converge simultaneously.
		\item We first prove~\eqref{eq:properties_virtual3} for $1\leq j\leq k$, using again the integral representation~\eqref{eq:Q_integral}.
		Since $q_j<q_{j+1}<\cdots$, if we deform the contour $|z|=r$ to be a slightly larger circle $|z|=r'$ with $0<r<q_j<r'<q_{j+1}$, the only extra contribution of the integral comes from the residue at $z=q_j$.
		Hence, we have
		\begin{align*}
			Q_{[j,k]}(x,y)&= \oint_{|z|=r} \frac{\mathrm{d} z}{2\pi \mathrm{i}z} \cdot\frac{z^{y-x+k-j+1}}{\prod_{\ell=j}^{k}(q_\ell -z)}\\
			&= \oint_{|z|=r'} \frac{\mathrm{d} z}{2\pi \mathrm{i}z} \cdot\frac{z^{y-x+k-j+1}}{\prod_{\ell=j}^{k}(q_\ell -z)}+\frac{q_j^{y-x+k-j}}{\prod_{\ell=j+1}^{k}(q_\ell-q_j)}.
		\end{align*}
	Note that 
	\begin{equation*}
		\left\vert q_j^x\cdot \oint_{|z|=r'}\frac{\mathrm{d} z}{2\pi \mathrm{i}z} \cdot\frac{z^{y-x+k-j+1}}{\prod_{\ell=j}^{k}(q_\ell -z)}\right\vert \leq C\cdot \left(\frac{q_j}{r'}\right)^{x}\to 0,\qquad \text{as }x\to\infty,
	\end{equation*}
for some constant $C>0$.
Thus, by~\eqref{eq:Qjk-convention}, we have 
\begin{equation*}
	\lim_{x\to \infty} q_j^{x}\cdot Q_{[j,k]}(x,y)= \frac{q^{y+k-j}_j}{\prod_{\ell=j+1}^{k}(q_\ell-q_j)}=Q_{[j,k]}(x_0^{(j-1)},y).
\end{equation*}
{To prove~\eqref{eq:properties_virtual3} for $j>k$, notice that in this case, using our convention~\eqref{eq:Q_compositionExt}, we have
\begin{equation*}
	Q_{[j,k]}(x,y)= Q_{j-1}^{-1}\circ \cdots \circ Q^{-1}_{k+1}(x,y)=0,
\end{equation*}
whenever $x>y$.
By definition~\eqref{eq:Qjk-convention0}, it then follows that
\[
	\lim_{x\to \infty} q_j^{x}\cdot Q_{[j,k]}(x,y)= 0 =Q_{[j,k]}(x_0^{(j-1)},y),
\]
as desired.} \qedhere
	\end{enumerate}
\end{proof}

The following proposition is rather standard in the theory of determinantal point processes.
However, here we unveil an additional important structure, i.e.\ the triangularity of the correlation matrix $\sf M$, which can be seen from our non-intersecting paths formulation.

Define first, for $n< N$, the following generalisation of the functions $\Psi_k^{(N)}$ from~\eqref{eq:Psi_def}:
\begin{equation}\label{eq:modifiedPsi}
\Psi_k^{(n)}(x):=  Q_{(n,N]}\circ \Psi_{k}^{(N)}(x) = \sum_{z\in \Z} Q_{(n,N]}(x,z) \Psi_{k}^{(N)}(z).
\end{equation}
{We remark that the summation over $z$ is actually within the finite range $y_k-N\leq z < x$; see Figure~\ref{fig:paths} and recall the definition of the operator $Q$.}
	\begin{proposition}[Correlation kernel and Fredholm determinant]\label{prop:modifiedCorKer}
			Let the functions $\Psi^{(n)}_k$ be defined by~\eqref{eq:Psi_def} and~\eqref{eq:modifiedPsi}.
			Let $ Q_{[j,n]}(x_0^{(j-1)},y)$ be defined 
		{by~\eqref{eq:Qjk-convention}-\eqref{eq:Qjk-convention0}}.
		Under the assumption that $q_1<q_2<\cdots$, the determinantal point process~\eqref{eq:modifiedDPP} admits the Fredholm determinant representation
		\begin{align*}
			\bbE\Bigg[ \prod_{1\leq {j\leq i}\leq N} \big(1+g(i,x^{(i)}_j) \big)  \Bigg]
			=\det( I+g K)_{\ell^2(\{1,\dots,N\}\times \Z )}
		\end{align*}
		for any bounded test function $g\colon \{1,\dots,N\}\times \Z\to\R$.
The correlation kernel $K$ is given by
		\begin{align}\label{eq:modifiedCorKer}
			K(m,x;n,x')=- Q_{(m,n]}(x,x')\ind_{\{n>m\}} + 
			\sum_{i,j=1}^N \Psi^{(m)}_i(x)  \, \big[ \sfM^{-1}\big]_{i,j} \, Q_{[j,n]}(x_{0}^{(j-1)},x'),
		\end{align}
		where the matrix $\sfM$ is defined by
		\begin{align}\label{eq:Mmatrix}
			\sfM_{i,j}:= \sum_{z\in \Z}  Q_{[i,N]}(x_{0}^{(i-1)},z) \, \Psi^{(N)}_j(z)\qquad 
			\text{for}\quad 1\leq i,j\leq N.
		\end{align}
Furthermore, $\sfM$ is upper-triangular and invertible.
	\end{proposition}
\begin{proof}
		Except for the stated properties of $\sfM$, Proposition~\ref{prop:modifiedCorKer} follows from~\cite[Lemma 3.4]{borodinFerrariPrahoferSasamoto07} and the accompanying remark, see also~\cite[Proposition 2.1]{johansson03}.
		We now check that, under our assumptions that $q_1<q_2<\cdots$, the matrix $\sf M$ is indeed upper-triangular with nonzero diagonal entries, and therefore invertible.
	 By~\eqref{eq:Mmatrix} and~\eqref{eq:Psi_def}, we have 
	\begin{align}
	\label{eq:Mmatrix2}
		\sfM_{i,j} 
		= \sum_{z\in \mathbb{Z}} Q_{[i,N]}(x_0^{(i-1)},z)\cdot\big( R^*_{(r,t]}\circ Q_{(j,N]}^{-1} \big) (z,{y_j-j})
	\end{align}
for $1\leq i,j\leq N$.
{Recalling formula~\eqref{eq:Qjk-convention} for $Q_{[i,N]}(x_0^{(i-1)},z)$ and expressing the Toeplitz operator $R^*_{(r,t]}\circ Q_{(j,N]}^{-1}$ as a contour integral in the usual way (see~\eqref{eq:QinvRstar_CI} in the next section for details)}, we see that 
	\begin{align*}
		\sfM_{i,j} &= \sum_{z\in \mathbb{Z}}\left(-\oint_{\gamma_{q_i}}\frac{\mathrm{d}\xi}{2\pi \mathrm{i}} \frac{\xi^{z+N-i}}{\prod_{\ell=i}^N(q_\ell-\xi)} \right) \oint_{|w|=r} \frac{\mathrm{d} w}{2\pi \mathrm{i}} w^{{y_j-z-N-1}}\prod_{\ell=j+1}^{N}(q_\ell-w)\prod_{\ell=r+1}^{t}(1+p_\ell w)\\
		&= -\oint_{\gamma_{q_i}}\frac{\mathrm{d}\xi}{2\pi \mathrm{i}}\oint_{|w|=r} \frac{\mathrm{d} w}{2\pi \mathrm{i}} \xi^{N-i}w^{{y_j-N-1}}\frac{\prod_{\ell=j+1}^{N}(q_\ell-w)\prod_{\ell=r+1}^{t}(1+p_\ell w)}{\prod_{\ell=i}^N(q_\ell-\xi)}\sum_{z\geq {y_j-N}}\left(\frac{\xi}{w}\right)^{z} \\
		&= -\oint_{\gamma_{q_i}}\frac{\mathrm{d}\xi}{2\pi \mathrm{i}}\oint_{|w|=r} \frac{\mathrm{d} w}{2\pi \mathrm{i}} \xi^{{y_j-i}}\cdot\frac{\prod_{\ell=j+1}^{N}(q_\ell-w)\prod_{\ell=r+1}^{t}(1+p_\ell w)}{\prod_{\ell=i}^N(q_\ell-\xi)}\cdot\frac{1}{w-\xi},
	\end{align*}
	where $r>\max\{|\xi|:\xi\in \gamma_{q_i}\}$, so that the sum over $z$ converges.
	Note that in the second equality we can restrict the sum to $z\geq {y_j-N}$, since the contour integral with respect to $w$ vanishes for $z<{y_j-N}$ due to analyticity.
	The only pole inside the $w$-contour is at $w=\xi$, and calculating its residue yields
	\begin{align*}
		\sfM_{i,j} 
		=-\oint_{\gamma_{q_i}}\frac{\mathrm{d}\xi}{2\pi \mathrm{i}}\xi^{{y_j-i}}\cdot\frac{\prod_{\ell=j+1}^{N}(q_\ell-\xi)\prod_{\ell=r+1}^{t}(1+p_\ell \xi)}{\prod_{\ell=i}^N(q_\ell-\xi)}.
	\end{align*}
	Note now that, for $i>j$, the integrand is analytic at {$q_i$}, hence the integral vanishes and $\sfM_{i,j}=0$.
	On the other hand, when $i=j$, by our assumptions on the parameters we have 
	\begin{align*}
		\sfM_{i,i}=-\oint_{\gamma_{q_i}}\frac{\mathrm{d}\xi}{2\pi \mathrm{i}}\frac{\xi^{{y_i-i}}}{q_i-\xi}\cdot\prod_{\ell=r+1}^{t}(1+p_\ell \xi)=q_i^{{y_i-i}}\cdot\prod_{\ell=r+1}^{t}(1+p_\ell q_i),
	\end{align*}
	which is nonzero.
	We conclude that the matrix $\sfM$ is upper-triangular and invertible.
\end{proof}

\begin{remark}
\label{rem:M_upperTriang}
		Here we provide a more intuitive, pathwise explanation of the fact that the matrix $\sfM$ is upper-triangular.  
		By~\eqref{eq:properties_virtual3}, for $i>j$ we have
		\begin{align*}
			\sfM_{i,j} 
			&:= \sum_{z\in \Z}  Q_{[i,N]}(x_0^{(i-1)},z) \, \Psi^{N}_j(z) \\
			&= \sum_{z\in \Z}\lim_{x\to \infty}q_i^{x}\cdot \big( Q_{i}\circ\cdots\circ Q_N\big)(x,z)\big( R^*_{(r,t]}\circ Q_N^{-1}\circ\cdots\circ Q_{j+1}^{-1} \big) (z,{y_j-j})\\
			&= \lim_{x\to \infty}q_i^{x}\cdot \big( R^*_{(r,t]}\circ Q^{-1}_{[j+1,i-1]} \big) (x,{y_j-j}),
		\end{align*}
		where in the last equality we used the fact that all the operators commute.
		Viewing the operators as associated to paths, $R^*$-operators take at most one step to the left at the time, whereas $Q^{-1}$-operators take at most one step to the right at the time.
		Thus, if $i>j$, for any $x$ sufficiently large, the point ${y_j-j}$ cannot be reached from $x$ by applying the operator $R_{(r,t]}^*\circ  Q^{-1}_{[j+1,i-1]}$, hence $R_{(r,t]}^*\circ  Q^{-1}_{[j+1,i-1]}(x,{y_j-j})=0$.
		This shows that $\sfM_{i,j}=0$ for $i>j$.
\end{remark}

	{We now derive a simplified expression for the correlation kernel in terms of biorthogonal functions}.
	Define
		\begin{align}\label{eq:Phi_def}
			\Phi^{(n)}_i(x):=\sum_{j=1}^n \big[ \sfM^{-1}\big]_{i,j} \, Q_{[j,n]}(x_{0}^{(j-1)},x)
			{=\sum_{j=i}^n \big[ \sfM^{-1}\big]_{i,j} \, Q_{[j,n]}(x_{0}^{(j-1)},x)}
			,\qquad x\in \Z,
		\end{align}
where the latter equality is due to the fact that $\sfM$ {(and hence $\sfM^{-1}$)} is upper-triangular.
	\begin{proposition}\label{prop:modifiedCorKer_simpl}
		The kernel $K$ in~\eqref{eq:modifiedCorKer} can be written as
		\begin{align}\label{eq:corKer_biorthogonal}
			K(m,x;n,x')=- Q_{(m,n]}(x,x')\ind_{\{n>m\}} + 
			\sum_{i=1}^n \Psi^{(m)}_i(x)   \, \Phi^{(n)}_i(x').
		\end{align}
		Moreover, for $1\leq i,j\leq n$ and $n=1,\dots,N$, the following biorthogonality relation holds:
		\begin{align}
			\sum_{x\in\Z} \Psi^{(n)}_i(x) \, \Phi^{(n)}_j(x) &= \delta_{i,j}, \label{eq:biorthogonalRelation}
		\end{align}
		where $\delta_{i,j}$ is the Kronecker delta.
	\end{proposition}
	\begin{proof}				
		Since $\sfM$ is upper-triangular and, by~\eqref{eq:Qjk-convention0}, $Q_{[j,n]}(x_0^{(j-1)},x) =0$ whenever $j>n$, the summation in~\eqref{eq:modifiedCorKer} over $i,j$ can be restricted to $1\leq i\leq j\leq n$, yielding
		\begin{align*}
			K(m,x;n,x')=-Q_{(m,n]}(x,x')\ind_{\{n>m\}} + 
			\sum_{i=1}^n \Psi^{(m)}_i(x)  \, \sum_{j=i}^n \big[ \sfM^{-1}\big]_{i,j} \, Q_{[j,n]}(x_0^{(j-1)},x').
		\end{align*}
		Therefore, \eqref{eq:corKer_biorthogonal} follows from definition~\eqref{eq:Phi_def}.
		
		To prove the biorthogonality relation~\eqref{eq:biorthogonalRelation}, recalling the definitions~\eqref{eq:modifiedPsi} and~\eqref{eq:Phi_def} of $\Psi^{(n)}_i(x)$ and $\Phi^{(n)}_j(x)$, we compute
		\begin{align*}
			\sum_{x\in\Z} \Psi^{(n)}_i(x) \, \Phi^{(n)}_j(x) 
			&=\sum_{x\in\Z} \left\{ \sum_{z\in\Z}   Q_{(n,N]}(x,z)\Psi^{(N)}_i(z)\right\}
			\left\{  \sum_{k=1}^N \big[\sfM^{-1}\big]_{j,k}  \,  Q_{[k,n]}(x_{0}^{(k-1)},x) \right\}  \\
			&= \sum_{k=1}^N \big[\sfM^{-1}\big]_{j,k} \sum_{z\in\Z} 
			\Psi^{(N)}_i(z) \sum_{x\in\Z } Q_{(n,N]}(x,z) \,  Q_{[k,n]}(x^{(k-1)}_0,x)\\
			&\stackrel{\eqref{eq:properties_virtual2}}{=} \sum_{k=1}^N \big[\sfM^{-1}\big]_{j,k} \sum_{z\in\Z}   \Psi^{(N)}_i(z) \,  Q_{[k,N]}(x^{(k-1)}_0,z)\\
			&\stackrel{\eqref{eq:Mmatrix}}{=} \sum_{k=1}^N \big[\sfM^{-1}\big]_{j,k} \cdot \sfM_{k,i} =\delta_{i,j} .
		\end{align*}
		{Note that the exchange of the summations over $x$ and $z$ is justified by absolute convergence, due to our working assumption $q_1<q_2<\cdots$; see the discussion at the beginning of this subsection.}
	\end{proof}

	\section{{Boundary value problem and random walk hitting times}}
	\label{sec:RWformulas}
	
	The goal of this section is to prove our main result, Theorem~\ref{thm:corker_hitting}.
	Towards this, in~\S\ref{subsec:FreholmPreliminaries} we establish some equivalent formulations of the contour integrals $\mathcal{S}$ and $\bar{\mathcal{S}}$ in terms of local operators.
	 Next, in~\S\ref{subsec:boundaryProblem}, we establish a relationship between the functions $\Phi^{(n)}_i(x)$, implicitly defined in~\eqref{eq:Phi_def}, and a terminal-boundary value problem for a discrete heat equation.
	In~\S\ref{subsec:RWhitting_preliminaries} and~\S\ref{subsec:kernel_hitting}, we express the solution to this problem in terms of random walk hitting probabilities.
	We will first do so under the additional assumption that $q_1<q_2<\cdots$, and then extend the result to general parameters through an analytic continuation argument, thus completing the proof of Theorem~\ref{thm:corker_hitting}.
	
	A boundary value problem of this kind and its connection to random walk hitting problems were first formulated in~\cite{matetskiQuastelRemenik21}.
	However, {our approach emphasizes} the role of local operators and their path interpretation; this might shed some additional light on the nature of the boundary value problem itself.
	{Furthermore}, our main technical tool, Proposition~\ref{prop:generalG}, requires a completely different proof compared to~\cite{matetskiQuastelRemenik21}.
	The reason is that, in the case of inhomogeneous rates, the kernel is not a polynomial in the spatial variables, while polynomiality is crucial in the proof of the random walk hitting formulas given by~\cite{matetskiQuastelRemenik21}.
	In~\S\ref{subsec:proofRWformula}, we will build a subtle induction argument to prove Proposition~\ref{prop:generalG}.
	
	As a preliminary notational remark, notice that, in~\S\ref{sec:dynamics}-\ref{sec:path-constr}, the vector $\bm{y}$ encoding the initial configuration satisfied the weak inequalities $y_1\geq \dots\geq y_N$, according to the original $\drsk$ dynamics of~\S\ref{sec:dynamics}.
	On the other hand, in this section, we will always assume that $\bm{y}$ satisfies the strict inequalities $y_1> \dots> y_N$, matching more closely the $\dtasep$ initial configuration and the notation of Theorem~\ref{thm:corker_hitting}.
	To translate formulas from the notation of previous sections, it will suffice to replace each $y_k$ with $y_k+k$ for all $1\leq k\leq N$.

	\subsection{Preliminaries towards the Fredholm determinant}
	\label{subsec:FreholmPreliminaries}

Moving towards the Fredholm determinant formula of Theorem~\ref{thm:corker_hitting}, here we prove an alternative representation of $\mathcal{S}$ and $\bar{\mathcal{S}}$ involving local operators.
This will give a natural connection between the path constructions of~\S\ref{sec:path-constr} and the random walk hitting times.

For $1\leq j\leq k$, we define the operators $\bar{Q}_{[j,k]}$ by
\begin{equation}\label{eq:Qbar}
\begin{split}
	\bar{Q}_{[j,k]}(x,y)
	&{:= Q_{[j,k]}(x,y)+(-1)^{k-j}	Q^\dagger_{[j,k]}(x,y)} \\
	&=\begin{cases}
		\left( Q_{j}\circ\cdots\circ Q_k\right)(x,y)\quad &x>y,\\
		(-1)^{k-j}\left(Q^{\dagger}_{j}\circ\cdots\circ Q^{\dagger}_k\right)(x,y)\quad &x\leq y,
	\end{cases}
\end{split}
\end{equation}
where $x,y\in \mathbb{Z}$.
It is then straightforward to check that 
\begin{equation*}
	Q_j^{-1}\circ\bar{Q}_{[j,k]}= \bar{Q}_{[j+1,k]}
	\qquad\text{and}\qquad
	\bar{Q}_{[j,k]}\circ  Q_{k}^{-1}= \bar{Q}_{[j,k-1]},
\end{equation*}
whenever $j<k$.
{However, when $j=k$, we have
\begin{equation}
\label{eq:Qbar_j=k}
\bar{Q}_{[k,k]}(x,y)
= Q_{k}(x,y)+Q^\dagger_{k}(x,y)
= q_k^{y-x}
\qquad \text{for all } x,y\in\Z,
\end{equation}
hence we deduce from~\eqref{eq:Qinverse} that}
\begin{equation*}
	Q_k^{-1}\circ\bar{Q}_{[k,k]}=\bar{Q}_{[k,k]}\circ  Q_{k}^{-1}= 0.
\end{equation*}
Note also that $\bar{Q}_{[j,k]}\circ\bar{Q}_{[k+1,\ell]}$ {may not be} well defined.
{The operator $\bar{Q}_{[j,k]}$ is Toeplitz, but its symbol is divergent on the whole complex plane.
However, recalling~\eqref{eq:Qdagger_integral} and~\eqref{eq:Q_integral}}, we can still express it as a contour integral:
\begin{equation}\label{eq:Qbar_CI}
	\begin{aligned}
		\bar{Q}_{[j,k]}(x,y)&=\oint_{|z|=r} \frac{\mathrm{d}z}{2\pi \mathrm{i}} \frac{z^{y-x+k-j}}{\prod_{\ell=j}^{k}(q_\ell-z)}-\oint_{|z|=R} \frac{\mathrm{d}z}{2\pi \mathrm{i}} \frac{z^{y-x+k-j}}{\prod_{\ell=j}^{k}(q_\ell-z)}\\
		&=-\oint_{\Gamma_{\bm{q}}} \frac{\mathrm{d}z}{2\pi \mathrm{i}} \frac{z^{y-x+k-j}}{\prod_{\ell=j}^{k}(q_\ell-z)},
	\end{aligned}
\end{equation}
where $0<r<q_\ell<R$ for all $j\leq \ell\leq k$, and $\Gamma_{\bm{q}}$ is any simple closed contour enclosing $q_\ell$ for all $j\leq \ell\leq k$ but {not 0}.

\begin{proposition}\label{prop:S-Sbar_localOp}
The operators $\mathcal{S}$ and $\bar{\mathcal{S}}$ defined in~\eqref{eq:Sintegral} and~\eqref{eq:Sbarintegral} can be rewritten as
	\begin{align}
		&\mathcal{S}_{[j,k],(r,t]}(x,y)= \left( {Q_{[j,k]}^{-1}}\circ R_{(r,t]}^*\right) (x,y),\label{eq:S_localOp}\\ 
		&\bar{\mathcal{S}}_{[j,k],(r,t]}(x,y)= \prod_{\ell=j}^{k}(q_\ell-1)\cdot\left(\bar{Q}_{[j,k]}\circ (R_{(r,t]}^*)^{-1}\right)(x,y), \label{eq:Sbar_localOp}
	\end{align}
{where the second formula holds under the hypothesis that $p_iq_\ell<1$ for all $r<i\leq t$ and $j\leq \ell\leq k$ (this condition is equivalent to the convergence of the right-hand side of~\eqref{eq:Sbar_localOp}).}
\end{proposition}
\begin{proof}
By~\eqref{eq:Qinverse_symbol} and Proposition~\ref{prop:Toeplitz}, the symbol of the (bi-infinite) Toeplitz operator $Q_{[j,k]}^{-1}$ is $\phi_{Q_{[j,k]}^{-1}}(z)=z^{-(k-j+1)} \prod_{\ell=j}^k (q_\ell - z)$, for $|z|>0$.
On the other hand, the Toeplitz operator $R_k^{*}$ has kernel $R_k^*(x,y)=\1_{y=x}+p_k\1_{y=x-1}$ and symbol {$\phi_{{R}_k^{*}}(z)=\sum_{x\in\Z}R_k^*(0,x)z^{-x}=1+p_k z$}.
Therefore, the composition $R_{(r,t]}^{*}$ is also a Toeplitz operator with symbol {$\phi_{{R}_{(r,t]}^{*}}(z)=\prod_{\ell=r+1}^{t}(1+p_\ell z)$ for every $z\in\C$.
It then follows from Proposition~\ref{prop:Toeplitz} that
\begin{equation}
\label{eq:QinvRstar_CI}
\left( Q_{[j,k]}^{-1}\circ R_{(r,t]}^*\right) (x,y)
= \oint_{|z|=r} \frac{\mathrm{d}z}{2\pi \mathrm{i}z} z^{y-x-k+j-1}
\prod_{\ell=j}^{k}(q_\ell -z)\cdot \prod_{\ell=r+1}^{t}(1+p_\ell z),
\end{equation}
for any $r>0$.
This, combined with definition~\eqref{eq:Sintegral}, proves~\eqref{eq:S_localOp}.}

Let us now prove~\eqref{eq:Sbar_localOp}.
By Proposition~\ref{prop:Toeplitz}, we may express the inverse of $R_{(r,t]}^*$ as
\begin{equation*}
	(R_{(r,t]}^*)^{-1}(x,y) = \oint_{|z|=R} \frac{\mathrm{d}z}{2\pi \mathrm{i}z} \frac{z^{y-x}}{\prod_{\ell=r+1}^{t}(1+p_\ell z)}, \qquad x,y\in \mathbb{Z},
\end{equation*}
where $0<R<\min\{p_\ell^{-1}\}$.
Using~\eqref{eq:Qbar_CI} for $\bar{Q}_{[j,k]}$ and noting from the above expression that $(R_{(r,t]}^*)^{-1}(x,y)=0$ if $x<y$, we can formally write
	\begin{align*}
		&\left(\bar{Q}_{[j,k]}\circ (R_{(r,t]}^*)^{-1}\right)(x,y)
		=\sum_{x'\geq y} \bar{Q}_{[j,k]}(x,x') \, (R_{(r,t]}^*)^{-1}(x',y) \\
		&= -\oint_{|z|=R}\frac{\mathrm{d}z}{2\pi \mathrm{i}z}\oint_{\Gamma_{\bm{q}}}\frac{\mathrm{d}w}{2\pi \mathrm{i}} \frac{z^y w^{-x+k-j}}{\prod_{\ell=j}^{k}(q_\ell-w)\cdot \prod_{\ell=r+1}^t (1+p_\ell z)}\sum_{x'\geq y} \left(\frac{w}{z}\right)^{x'}.
	\end{align*}
In order for $\big(\bar{Q}_{[j,k]}\circ (R_{(r,t]}^*)^{-1}\big)(x,y)$ to converge, i.e.\ for the geometric series $\sum_{x'\geq y} (w/z)^{x'}$ to converge, we need to take $\Gamma_{\bm{q}}$ to be inside the circle $|z|=R$; see Figure~\ref{fig: contourSbar} for an illustration of the contours.
This is only possible if $\max_{j\leq \ell\leq k}\{q_\ell\}<\min_{r<i\leq t}\{p_{i}^{-1}\}$, or equivalently $p_iq_\ell<1$ for all $r<i\leq t$ and $j\leq \ell\leq k$.
{Under this hypothesis, we evaluate the geometric series and compute the only residue of the $z$-contour at $z=w$, thus obtaining
\begin{align*}
\left(\bar{Q}_{[j,k]}\circ (R_{(r,t]}^*)^{-1}\right)(x,y)
&= -\oint_{|z|=R}\frac{\mathrm{d}z}{2\pi \mathrm{i}}\oint_{\Gamma_{\bm{q}}}\frac{\mathrm{d}w}{2\pi \mathrm{i}} \frac{w^{y-x+k-j}}{\prod_{\ell=j}^{k}(q_\ell-w)\cdot \prod_{\ell=r+1}^t (1+p_\ell z)}\frac{1}{z-w} \\
&= -\oint_{\Gamma_{\bm{q}}}\frac{\mathrm{d}w}{2\pi \mathrm{i}} \frac{w^{y-x+k-j}}{\prod_{\ell=j}^{k}(q_\ell-w) \cdot \prod_{\ell=r+1}^t (1+p_\ell w)}.
\end{align*}
This, combined with definition~\eqref{eq:Sbarintegral}, leads to~\eqref{eq:Sbar_localOp}.}
\end{proof}
	
	\begin{figure}
		\begin{tikzpicture}
			\draw[gray, thin] (-7,0)--(4,0);
			\draw[red,fill] (-6,0) circle (0.1cm);
			\draw[red,fill] (-5.5,0) circle (0.1cm);
			\draw[red,fill] (-5,0) circle (0.1cm);
			\draw[red,fill] (-4.5,0) circle (0.1cm);
			\draw[red,fill] (-4,0) circle (0.1cm);
			\draw[blue,fill] (1.5,0) circle (0.1cm);
			\draw[blue,fill] (2,0) circle (0.1cm);
			\draw[blue,fill] (2.5,0) circle (0.1cm);
			\draw[black,fill] (0,0) circle (0.1cm);
			\node[black,text width=0.3cm] at (0.1,-0.4) {\scriptsize $0$};
			\draw[black,thick] (0,0) ellipse (3cm);
			\node[black,text width=0.3cm] at (2,-1.2) {\scriptsize $\Gamma_{\bm{q}}$};
			\draw[black,thick] (2,0) ellipse (0.8cm and 0.8cm); 
			\node[black,text width=1cm] at (1.7,-3) {\scriptsize $|z|= R$};
		\end{tikzpicture}
		\caption{The contours $\Gamma_{\bm{q}}$ and $|z|=R$.
		Here, the red dots represent the points $\{-p_\ell^{-1}\}$ and the blue dots represent the points $\{q_\ell\}$.}
		\label{fig: contourSbar}
	\end{figure}
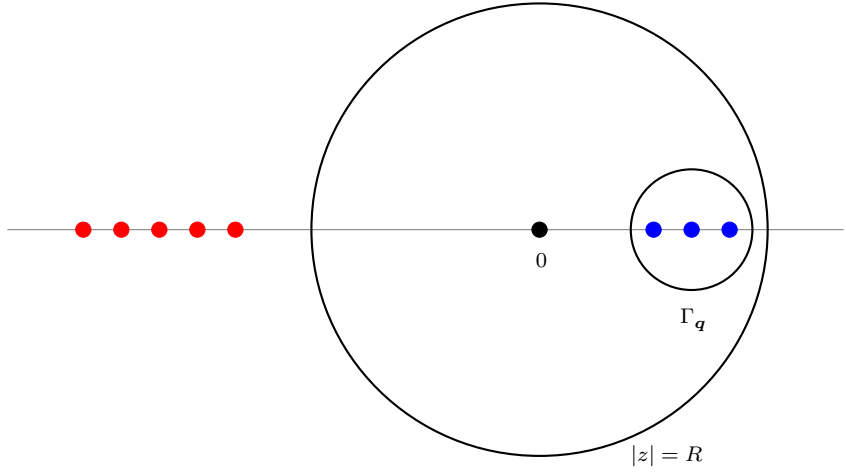

	\subsection{{Terminal-boundary value problem}}\label{subsec:boundaryProblem}
	Assume that $q_1<q_2<\cdots$. 
	\begin{proposition}\label{prop:modifiedBVP}
		Recall the definitions of $Q_k$ and its inverse in~\eqref{eq:Q}-\eqref{eq:Qinverse}.
		{Recall also that $\bm{y}=(y_1> \dots> y_N)$ is an arbitrary vector encoding the initial configuration of $\dtasep$.}
		For $n\leq N$ and $k\in\{0,\dots,n-1\}$, consider the terminal-boundary value problem
		\begin{empheq}[left=\empheqlbrace]{align}
			&h^n_k(\ell+1,x) = h_k^{n}(\ell,\cdot) \circ  Q^{-1}_{n-\ell} (x) &&x\in\Z, \; \ell <k,  \label{eq:BVP1} \\
			&h^n_k(\ell,y_{n-\ell}) = 0 &&\ell <k,  \label{eq:BVP2} \\
			&h^n_k(k,x) = q^{x-y_{n-k}}_{n-k} && x\in \Z. \label{eq:BVP3}
		\end{empheq}
		Then, for $0\leq \ell\leq k$,  the functions
		\begin{align}\label{eq:h_BVP}
			h^n_k(\ell,x):=\Phi^{(n)}_{n-k}\circ R^*_{(r,t]} \circ  Q^{-1}_{(n-\ell,n]} (x)
		\end{align}
		solve~\eqref{eq:BVP1}-\eqref{eq:BVP3}.
		In particular, we have that 
		\begin{equation}\label{modifiedPhi_hk}
			\Phi^{(n)}_{n-k}(x)= h^n_k(0,\cdot) \circ (R^*_{(r,t]})^{-1}(x).
		\end{equation}
	\end{proposition}
	
	\begin{proof}
		It is clear that $h^n_k$ defined in~\eqref{eq:h_BVP} satisfies~\eqref{eq:BVP1}, since, for $\ell<k$,
		\[
		\begin{aligned}
			 h^n_k(\ell+1,x)
			&:=  \Phi^{(n)}_{n-k}\circ R^*_{(r,t]} \circ \big\{ Q^{-1}_n \circ \cdots \circ  Q^{-1}_{n-\ell}\big\} (x) \\
			&=  \Phi^{(n)}_{n-k}\circ R^*_{(r,t]} \circ \big\{ Q^{-1}_n \circ \cdots \circ  Q^{-1}_{n-\ell+1}\big\} \circ  Q^{-1}_{n-\ell} (x) \\
			&=   h^{n}_k(\ell,\cdot)  \circ  Q^{-1}_{n-\ell} (x).
		\end{aligned}
		\]
		The boundary condition~\eqref{eq:BVP2} follows from the biorthogonality property~\eqref{eq:biorthogonalRelation} and the definition of $\Psi^{(n)}_k$ in~\eqref{eq:Psi_def} and~\eqref{eq:modifiedPsi}: for $\ell<k$,
		\begin{equation*}
		\begin{aligned}
			 h^{n}_k(\ell,y_{n-\ell}) 
			&=\big(\Phi^{(n)}_{n-k}\circ R^*_{(r,t]} \circ  Q^{-1}_{(n-\ell, n]}\big)  (y_{n-\ell}) \\
			&=\sum_{x\in\Z} \Phi^{(n)}_{n-k}(x) \cdot \big(R^*_{(r,t]}\circ  Q^{-1}_{(n-\ell,n]}\big)(x,y_{n-\ell})
			=\sum_{x\in\Z} \Phi^{(n)}_{n-k}(x)\Psi^{(n)}_{n-\ell}(x)
			=0.
		\end{aligned}
		\end{equation*}
		
		Finally, we check the terminal condition~\eqref{eq:BVP3}.
		By the definition~\eqref{eq:Phi_def} of $\Phi^{(n)}_{n-k}$, we have
		\begin{equation}
		\label{eq:h_terminal_proof}
		\begin{aligned}
			 h^n_k(k,x)
			&= \Phi^{(n)}_{n-k}\circ R^*_{(r,t]} \circ  Q^{-1}_{(n-k,n]} (x) \\
			&= \sum_{j=n-k}^{{n}}  \big[\sfM^{-1}\big]_{n-k,j} 
			\sum_{y\in\Z} Q_{[j,n]} (x^{(j-1)}_0,y) \cdot {\left(R^*_{(r,t]} \circ Q^{-1}_{(n-k,n]}\right)(y,x).}
		\end{aligned}
		\end{equation}
		{Using~\eqref{eq:properties_virtual3} and recalling that the Toeplitz operators $Q$, $Q^{-1}$ and $R^*$ all commute with each other, we have, for $j>n-k$,
		\[
		\begin{aligned}
			&\sum_{y\in\Z} Q_{[j,n]} (x^{(j-1)}_0,y) \cdot \left(R^*_{(r,t]} \circ Q^{-1}_{(n-k,n]}\right)(y,x) \\
			= \; &\sum_{y\in\Z}
			\lim_{z\to\infty}q_j^{z} \cdot Q_{[j,n]}(z,y) \cdot \left(R^*_{(r,t]} \circ Q^{-1}_{(n-k,n]}\right)(y,x)
			= \lim_{z\to\infty}q_j^{z} \left(R^*_{(r,t]} \circ Q^{-1}_{[n-k+1,j-1]}\right)(z,x).
		\end{aligned}
		\]
		The same argument of Remark~\ref{rem:M_upperTriang} shows that the latter limit is zero for $j>n-k$.
		Therefore, the only surviving summand in~\eqref{eq:h_terminal_proof} is the one corresponding to $j=n-k$.}
		We thus have
		\begin{equation*}
		\begin{aligned}
			h^n_k(k,x)&=
			\big[ \sfM^{-1} \big]_{n-k,n-k}
			\sum_{y\in\Z} Q_{[n-k,n]} (x^{(n-k-1)}_{0},y) \cdot \left(R^*_{(r,t]} \circ Q^{-1}_{(n-k,n]}\right)(y,x)  \\
			&=\frac{\sum_{y\in\Z} Q_{n-k}(x^{(n-k-1)}_{0},y) \cdot R^*_{(r,t]}(y,x)}{\sum_{y\in\Z} Q_{n-k}(x^{(n-k-1)}_{0},y) 
				\cdot R^*_{(r,t]}(y,y_{n-k})} \\
			&= \frac{q_{n-k}^{x}\cdot\sum_{y\in\Z} q_{n-k}^{y-x} \cdot R^*_{(r,t]}(y,x)}
				{q_{n-k}^{y_{n-k}}\cdot\sum_{y\in\Z}q_{n-k}^{y-y_{n-k}} \cdot R^*_{(r,t]}(y,y_{n-k})}
		\end{aligned}
		\end{equation*}
		In the second equality we used again the commutativity of the operators, the fact that $\big[{\sf M}^{-1}\big]_{n-k,n-k}=({\sf M}_{n-k,n-k})^{-1}$ (as $\sf M$ is upper-triangular), and~\eqref{eq:Mmatrix2}.
		In the third equality we used~\eqref{eq:properties_virtual1}.
		Since $R_{(r,t]}^*$ is a Toeplitz operator, the sum
		\begin{align*}
			\sum_{y\in\Z} q_{n-k}^{y-x} \cdot R^*_{(r,t]}(y,x) 
			= \sum_{y\in \Z} q_{n-k}^{y-x} \cdot R^*_{(r,t]}(y-x,0)
			= \sum_{y\in \Z} q_{n-k}^{y} \cdot R^*_{(r,t]}(y,0)
		\end{align*}
		does not depend on $x$.
		Therefore, in the latest expression of $h^n_k(k,x)$, the two sums appearing in the numerator and denominator cancel each other, and we obtain $h^n_k(k,x)=q_{n-k}^{x-y_{n-k}}$, as desired.
	\end{proof}
	
		{Recalling~\eqref{eq:Qinverse}}, equation~\eqref{eq:BVP1} can be written equivalently as
	\begin{equation}\label{eq:recurrence}
		h_k^{n}(\ell+1,x)= -h_k^{n}(\ell,x)+q_{n-\ell}\cdot h_k^{n}(\ell,x-1).
	\end{equation}
	If we solve the latter recursively in $x$ for any fixed $\ell<k$, using the boundary condition~\eqref{eq:BVP2}, we obtain a cumulative (integral) expression of $h^n_k(\ell,\cdot)$ in terms of  $h^n_k(\ell+1,\cdot)$:
	\begin{equation}\label{eq:onestep}
		h_k^n(\ell,x)=
		\begin{dcases}
			\sum_{y=x+1}^{y_{n-\ell}} q_{n-\ell}^{x-y}\cdot h_k^{n}(\ell+1,y)\quad & x\leq y_{n-\ell},\\
			-\sum_{y=y_{n-\ell}+1}^{x} q_{n-\ell}^{x-y}\cdot h_k^{n}(\ell+1,y)\quad & x\geq y_{n-\ell},
		\end{dcases}
	\end{equation}
	with the convention that the above summations equal zero when their range is empty, i.e., when $x=y_{n-\ell}$.
	
	\begin{proposition}\label{prop:polynomiality}
		If the parameters $q_{\ell}$ are equal to some value $q>1$ for all $\ell$, then {every} solution to the initial-boundary value problem~\eqref{eq:BVP1}-\eqref{eq:BVP3} satisfies the property that $q^{-x} h^n_k(\ell, x)$ is a polynomial {of degree $k-\ell$ in the spatial variable $x$}.
	\end{proposition}
	\begin{proof}
		This is trivially true for $\ell=k$, as $h^n_k(k,x)={q^{x-y_{n-k}}}$ by the initial condition~\eqref{eq:BVP3}.
		Using the recursion~\eqref{eq:onestep}, we see that $h^n_k(k-1,x)=q^{x-y_{n-k}}\cdot (y_{n-(k-1)}-x)$.
		Using this and, again, the recursion~\eqref{eq:onestep} inductively, we arrive at the claim.
	\end{proof}
	
	\begin{remark}
	\label{rem:non-polynomiality}
		Using the same procedure as in the proof of Proposition~\ref{prop:polynomiality}, one can see that the solutions to the terminal-boundary value problem~\eqref{eq:BVP1}-\eqref{eq:BVP3} {do not have an analogous polynomial property} if the parameters $q_{\ell}$, $1\leq\ell\leq N$, are not identical.
		As mentioned at the beginning of~\S\ref{sec:RWformulas}, we will need to prove Theorem~\ref{thm:corker_hitting} using different methods, compared to~\cite{matetskiQuastelRemenik21}, due to the non-polynomiality of these solutions.
	\end{remark}
	
	\subsection{Random walk hitting {probabilities}}
	\label{subsec:RWhitting_preliminaries}
	
	{
	We will now present some preliminary computations that, although not strictly essential for our final purposes, will motivate and illustrate the representation of $h^n_k$ in terms of random walk hitting probabilities.
	}
	
	When $x\leq y_{n-\ell}$ we can iterate the recursion~\eqref{eq:onestep} to obtain, for $\ell<k$,
	\begin{equation}
	\label{eq:h_hitting}
	\begin{split}
		h^n_k(\ell,x)
		= \sum_{x<x_1\leq y_{n-\ell}} \, \sum_{x_1<x_2\leq y_{n-(\ell+1)}}&\cdots \sum_{x_{k-\ell-1}<x_{k-\ell}\leq y_{n-k+1}} \\
		&q_{n-\ell}^{x-x_1} \,q_{n-(\ell+1)}^{x_1-x_2} \cdots q_{n-(k-1)}^{x_{k-\ell-1}-x_{k-\ell}}  q_{n-k}^{x_{k-\ell}-y_{n-k}}
		\end{split}
	\end{equation}
	Let $S^*$ be an {$n$-step} geometric random walk moving strictly to the right ({more precisely, a sum of independent geometric random variables with inhomogeneous parameters $q_{n}^{-1},\dots,q_1^{-1}$}) with transition probability
	\begin{equation}
	\label{eq:transitionRWstar}
		\mathbb{P}(S^*_\ell=y \;|\; S^*_{\ell-1}=x):=(q_{n-\ell}-1)q_{n-\ell}^{x-y}\1_{y>x}, \qquad 0\leq \ell \leq n-1.
	\end{equation}
	
	 Then for $x\leq y_{n-\ell}$, the one step recurrence~\eqref{eq:onestep} can be written as 
	\begin{equation*}
		h_{k}^{n}(\ell,x) = \frac{1}{q_{n-\ell}-1}\cdot\mathbb{E}_{S^*_{\ell-1}=x}\left[h_{k}^n(\ell+1,S^*_\ell)\1_{\{S^*_\ell\leq y_{n-\ell}\}}\right].
	\end{equation*} 
	{Analogously, writing~\eqref{eq:h_hitting} in terms of the law of the random walk yields}
	\begin{equation}\label{eq:indicatorrep}
		h_{k}^{n}(\ell,x) = \left(\prod_{j=\ell}^{k}\frac{1}{q_{n-j}-1}\right)
		\mathbb{E}_{S^*_{\ell-1}=x}\left[q_{n-k}^{S^*_{k-1}-y_{n-k}}\1_{\{S^*_j\leq y_{n-j},\;\ell\leq j\leq k-1\}}\right].
	\end{equation} 
	For $0\leq \ell\leq k\leq n-1$, we define the hitting time 
	\begin{equation}\label{eq:hittingStar}
		\tau^*_{\ell,n}:= \min\{m\in \{\ell,\cdots,n-1\}: S^*_m> y_{n-m}\}.
	\end{equation}
	\begin{proposition}
	With the above notation, for $x\leq y_{n-\ell}$, we have 
	\begin{equation}\label{eq:hittingtimerep}
		h_{k}^n(\ell,x)= \frac{\mathbb{P}_{S^*_{\ell-1}=x}(\tau^*_{\ell,n}=k)}{\prod_{j=\ell}^{k}({q_{n-j}}-1)}.
	\end{equation}
	\end{proposition}
	\begin{proof}
	Let us first write
	\begin{align*}
		\mathbb{P}_{S^*_{\ell-1}=x}(\tau^*_{\ell,n}=k)&= \mathbb{P}_{S^*_{\ell-1}=x}(\tau^*_{\ell,n}\geq k)-\mathbb{P}_{S^*_{\ell-1}=x}(\tau^*_{\ell,n}\geq k+1)\\
		&= \mathbb{E}_{S^*_{\ell-1}=x}\left[\1_{\{S^*_j\leq y_{n-j},\; \ell\leq j\leq k-1\}}\right]-\mathbb{E}_{S^*_{\ell-1}=x}\left[\1_{\{S^*_j\leq y_{n-j},\; \ell\leq j\leq k\}}\right].
	\end{align*}
	Here, the indicator $\1_{\{S^*_j\leq y_{n-j},\; \ell\leq j\leq k-1\}}$ is set to be $1$ when $\ell=k$.
	{Rearranging the terms and using standard properties of the conditional expectation, we obtain}
	\begin{align*}
		\mathbb{P}_{S^*_{\ell-1}=x}(\tau^*_{\ell,n}=k) =  \mathbb{E}_{S^*_{\ell-1}=x}\left[\left(1-\mathbb{E}\left[\1_{\{S^*_k\leq y_{n-k}\}} \;\middle|\; {S^*_\ell,\dots,S^*_{k-1}}\right]\right)\1_{\{S^*_j\leq y_{n-j},\;\ell\leq j\leq k-1\}}\right].
	\end{align*}
	Note that 
	\begin{align*}
		\mathbb{E}\left[\1_{\{S^*_k\leq y_{n-k}\}} \;\middle|\; {S^*_\ell,\dots,S^*_{k-1}}\right] = (q_{n-k}-1)\sum_{x=S^*_{k-1}+1}^{y_{n-k}}q_{n-k}^{S^*_{k-1}-x}= 1-q_{n-k}^{S^*_{k-1}-y_{n-k}},
	\end{align*}
	so that
	\begin{align*}
		\mathbb{P}_{S^*_{\ell-1}=x}(\tau^*_{\ell,n}=k) = \mathbb{E}_{S^*_{\ell-1}=x}\left[q_{n-k}^{S^*_{k-1}-y_{n-k}}\1_{\{S^*_j\leq y_{n-j},\; \ell\leq j\leq k-1\}}\right].
	\end{align*}
	Hence, \eqref{eq:hittingtimerep} follows from~\eqref{eq:indicatorrep}.
	\end{proof}

	\subsection{Hitting probability representation for the kernel}\label{subsec:kernel_hitting}
	We will now derive a more explicit representation of the Fredholm determinant kernel~\eqref{eq:corKer_biorthogonal}, which contains the implicit part $\sum_{i=1}^{n}\Psi_i^{(m)}(x)\Phi_i^{(n)}(x')$.
	Using~\eqref{eq:modifiedPsi}, \eqref{eq:Psi_def} and~\eqref{modifiedPhi_hk}, {we write
	\begin{equation*}
			\sum_{i=1}^n \Psi^{(m)}_i(x) \Phi^{(n)}_i(x')
			= \sum_{i=1}^{n} Q_{(m,N]}\circ R_{(r,t]}^*\circ Q_{(i,N]}^{-1}(x,y_i) \sum_{z_2\in\Z} h^n_{n-i}(0,z_2)\cdot(R^*_{(r,t]})^{-1}(z_2,x').
	\end{equation*}
Notice that, by the commutativity of Toeplitz operators,
		\begin{align*}
			Q_{(m,N]}\circ R_{(r,t]}^*\circ Q_{(i,N]}^{-1}(x,y_i) &= Q_{[1,N]}\circ Q_{[1,m]}^{-1}\circ R_{(r,t]}^*\circ Q_{[1,N]}^{-1}\circ Q_{[1,i]}(x,y_i)\\
			&= Q_{[1,m]}^{-1}\circ R_{(r,t]}^*\circ Q_{[1,i]}(x,y_i)\\
			&=  \sum_{z_1\in \mathbb{Z}}(Q_{[1,m]}^{-1}\circ R_{(r,t]}^*)(x,z_1)Q_{[1,i]}(z_1,y_i).
		\end{align*}
		Therefore, we obtain
		\begin{equation}\label{eq: kernel_decomp}
		\sum_{i=1}^n \Psi^{(m)}_i(x) \Phi^{(n)}_i(x')
		= \sum_{z_1,z_2\in\Z}\big(  Q_{[1,m]}^{-1}\circ R^*_{(r,t]} \big)(x,z_1) \cdot G(z_1,z_2) \cdot (R^*_{(r,t]})^{-1}(z_2,x'),
		\end{equation}
		where the function $G$ is defined as
	\begin{equation}\label{eq:Gfunction}
		G(z_1,z_2):= \sum_{i=1}^n Q_{[1,i]}(z_1,y_i) \cdot h_{n-i}^n(0,z_2).
	\end{equation}
}
	By the definition~\eqref{eq:Q} of the $Q$-operators and formula~\eqref{eq:h_hitting}, we have
	\begin{equation}\label{eq:Gfunction_sum}
		\begin{aligned}
			G(z_1,z_2)=\sum_{i=1}^n \, &\sum_{z_1:=x_0>x_1>x_2>\cdots >x_{i-1}>y_i} q_1^{x_1-z_1} \,q_2^{x_2-x_1}
			\cdots q_{i-1}^{x_{i-1}-x_{i-2}} \, q_{i}^{y_i-x_{i-1}} \\
			& \cdot  \sum_{\substack{ z_2=:x_n<x_{n-1}< \cdots <x_i \\
					x_{n-1}\leq y_n,\dots, x_i\leq y_{i+1}}}
			q_i^{x_{i}-y_i} \,q_{i+1}^{x_{i+1}-x_{i}}\cdots q_{n-1}^{x_{n-1}-x_{n-2}} \, q_n^{z_2-x_{n-1}}
		\end{aligned}
	\end{equation}
	for $z_2\leq y_n$.
	Up to a normalizing constant, formula~\eqref{eq:Gfunction_sum} precisely represents the probability that the geometric random walk $S^*$ defined in~\eqref{eq:transitionRWstar}, started from $z_2$, ends at $z_1$ after $n$ steps and enters the region (strictly) to the right of the curve $(y_i)_{1\leq i\leq n}$ in between; see Figure~\ref{fig:boundary} for an illustration.
	More precisely, for $z_2\leq y_n$,
	\begin{equation*}
		G(z_1,z_2)= \frac{\mathbb{P}_{S^*_{-1}=z_2}(S^*_{n-1}=z_1, \tau^*_{0,n}<n)}{\prod_{j=0}^{n-1}(q_{n-j}-1)},
	\end{equation*}
	where $\tau^*_{0,n}$ is defined in~\eqref{eq:hittingStar}.
	
		\begin{figure}
		\begin{tikzpicture}[scale=0.4]
			\foreach \i in {-1,...,10}{
				\draw (-2, \i  ) grid (26, \i );
			}
			\node[draw, red, circle, inner sep=2pt, fill] at (2,-1) {};
			\node[scale=0.8] at (2,-1.8) {$z_2=x_n$}; \node[scale=0.8] at (25.5,10.5) {$z_1=x_0$};
			\node[draw, red, circle, inner sep=2pt, fill] at (24.5,10) {};
			\node[draw, circle, inner sep=2pt, fill] at (7,0) {};
			\node[draw, circle, inner sep=2pt, fill] at (8.5,1) {};
			\node[draw, circle, inner sep=2pt, fill] at (10,2) {};
			\node[draw, circle, inner sep=2pt, fill] at (10.7,3) {};
			\node[draw, circle, inner sep=2pt, fill] at (11.5,4) {};
			\node[draw, circle, inner sep=2pt, fill] at (14,5) {};
			\node[draw, circle, inner sep=2pt, fill] at (16,6) {};
			\node[draw, circle, inner sep=2pt, fill] at (18,7) {};
			\node[draw, circle, inner sep=2pt, fill] at (19,8) {};
			\node[draw, circle, inner sep=2pt, fill] at (20,9) {};
			\node[draw, circle, inner sep=2pt, fill] at (21,10) {};
			\node[scale=0.8] at (7.7,0.4) {$y_n$};
			\node[scale=0.8] at (9.4,1.4) {$y_{n-1}$};
			\node[scale=0.8] at (21.7,10.4) {$y_{1}$};
			\draw[thick, red] (2,-1)--(2,0)--(3,0)--(3,1)--(4.5,1)--(4.5,2)--(5,2)--(6,2)--(6,3)--(8,3)--(8,4)--(9.5,4)--(9.5,5)--(14, 5);
			\draw[ultra thick, red, dashed] (14,5)--(15,5)--(15,6)--(17,6)--(17,7)--(17.5,7)--(17.5,8)--(18.5, 8)--(18.5,9)--(19.5,9)--
			(23, 9)--(23,10)--(24,10) ;
			\node[draw, red, circle, inner sep=1pt, fill] at (3,0) {};  \node[scale=0.8]  at (3.7,-0.5) {$x_{n-1}$}; 
			\node[draw, red, circle, inner sep=1pt, fill] at (4.5,1) {};  \node[scale=0.8]  at (5.2,0.5) {$x_{n-2}$}; 
			\node[draw, red, circle, inner sep=1pt, fill] at (6,2) {};  
			\node[draw, red, circle, inner sep=1pt, fill] at (8,3) {};  
			\node[draw, red, circle, inner sep=1pt, fill] at (9.5,4) {};  \node[scale=0.8]  at (9.7, 3.5) {$x_{i}$}; 
			\node[draw, red, circle, inner sep=1pt, fill] at (15, 5) {};  \node[scale=0.8]  at (15.7, 4.5) {$x_{i-1}$}; 
			\node[draw, red, circle, inner sep=1pt, fill] at (17, 6) {};  \node[scale=0.8]  at (17.7, 5.5) {$x_{i-2}$}; 
			\node[draw, red, circle, inner sep=1pt, fill] at (23, 9) {};  \node[scale=0.8]  at (23.5, 8.5) {$x_{1}$}; 
			\node[scale=0.8]  at (12.5, 4.5) {$y_{i+1}$};   \node[scale=0.8]  at (14.5, 5.5) {$y_{i}$};    \node[scale=0.8]  at (16, 6.6) {$y_{i-1}$}; 
		\end{tikzpicture}
		\caption{Path representation of the function $G(z_1,z_2)$ for $z_2\leq y_n$, as in~\eqref{eq:Gfunction}-\eqref{eq:Gfunction_sum}.
		For $1\leq i \leq n$, the solid red path depicts the path representation of $h^n_{n-i}(0,z_2)$ (see also~\eqref{eq:h_hitting}), while the dashed, red path depicts the path representation of  $(Q_1\circ\cdots\circ Q_i)(z_1,y_i)$. 
		 Concatenating the two paths gives a path of the geometric random walk $S^*$ going from $z_2$ to $z_1$, which enters the region (strictly) to the right of the (discrete) curve 	$(y_i)_{1\leq i\leq n}$ with a first entrance time at some time $1\leq i \leq n$.}
		 \label{fig:boundary}
	\end{figure}
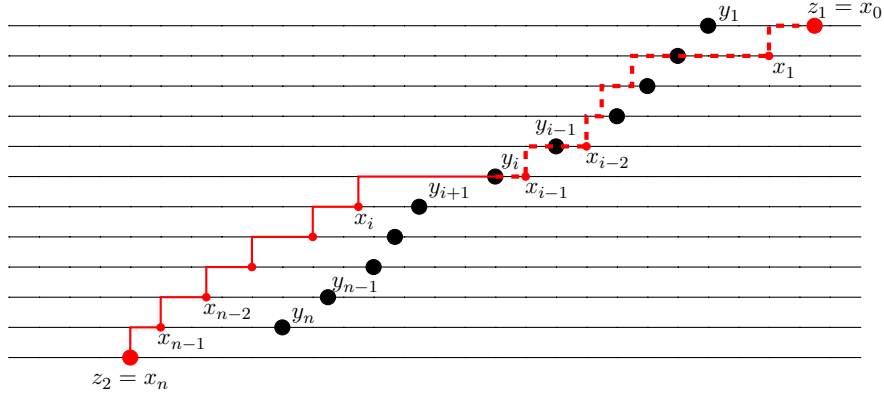
	
	The above expression of $G(z_1,z_2)$ is only valid for $z_2\leq y_n$.
	{For this special case, one only needs to use the first recurrence relation in~\eqref{eq:onestep}.
	The situation for $z_2>y_n$ is more complicated, since in this case one also has to use the second recurrence relation in~\eqref{eq:onestep}.
	Consequently, the expression for $G(z_1,z_2)$ defined in~\eqref{eq:Gfunction} will no longer be a sum of positive terms, but a sum containing both positive and negative terms; accordingly, $G(z_1,z_2)$ will not be a probability (up to normalization).	
	Nevertheless, $G(z_1,z_2)$ still possesses a probabilistic interpretation.}
	In order to extend the probabilistic interpretation of $G(z_1,z_2)$ to all $z_1,z_2\in \mathbb{Z}$, we need to introduce additional notation.
	For $1\leq j\leq k\leq n$, recall the operators $\bar{Q}_{[j,k]}$ introduced in~\eqref{eq:Qbar} and, for convenience, define a renormalized version of them as follows:
	\begin{equation}
	\label{eq:Qhat}
		\hat{Q}_{[j,k]}(x,y):= \Bigg(\prod_{\ell=j}^{k}(q_\ell-1)\Bigg) \bar{Q}_{[j,k]}(x,y).
	\end{equation}
	We now express $G(z_1,z_2)$ in terms of a random walk hitting problem involving the $\hat{Q}$-operators.
	Let $S$ be a geometric random walk {moving strictly to the left} with transition probabilities
	\begin{equation}\label{eq:transitionRW}
		\mathbb{P}(S_\ell=y \;|\; S_{\ell-1}=x):= (q_\ell-1)q_{\ell}^{y-x}\1_{y<x}
		={(q_\ell-1) Q_\ell(x,y)},\qquad \text{for }1\leq \ell\leq n.
	\end{equation}
	Define the hitting time 
	\begin{equation}\label{eq:hittingTime}
		\tau:=\min\{m\in \{0,\dots,n\}\colon S_m>y_{m+1}\},
	\end{equation} where $y_{n+1}:=-\infty$. Then we have 
	\begin{proposition}\label{prop:generalG}
		For any $z_1,z_2\in \mathbb{Z}$,
		\begin{equation}\label{eq:Gfunction_general}
			G(z_1,z_2)=\frac{\mathbb{E}_{S_0=z_1}\left[\hat Q_{[\tau+1,n]}(S_\tau,z_2)\1_{\tau<n}\right]}{\prod_{\ell=1}^{n}(q_\ell-1)}.
		\end{equation} 
	\end{proposition}
		\begin{remark}
	The special case of Proposition~\ref{prop:generalG} when $q_j=q$ for all $j$ was proved in~\cite{matetskiQuastelRemenik21}.
	{As pointed out earlier}, their proof relies crucially on the fact that $q^{-z_2}G(z_1,z_2)$ and $\hat{Q}_{[k,n]}(\cdot,z_2)$ (and hence the right hand side of~\eqref{eq:Gfunction_general}) are both polynomials in $z_2$, so one only needs to check the equality for $z_2$ in an infinite subset of $\mathbb{Z}$ (a convenient choice would then be $z_2\leq y_n$).
	{When the parameters $\{q_j\}$ are distinct,} the polynomiality no longer holds. 
\end{remark}	

The proof of Propositon~\ref{prop:generalG} is one of the main technical {novelties of this article} and will be presented in~\S\ref{subsec:proofRWformula}.
Assuming for the moment Proposition \ref{prop:generalG}, we are ready to prove Theorem~\ref{thm:corker_hitting}.
{The crucial additional information in Theorem~\ref{thm:corker_hitting} is a more explicit expression for the correlation kernel, which, in~\eqref{eq:corKer_biorthogonal}, was given implicitly through a biorthogonal relation.
To prove the result, we will first assume that the parameters $\{q_i\}$ satisfy the condition $q_1<q_2<\cdots$ and use the probabilistic representation~\eqref{eq:Gfunction_general}.
Then, we will remove the restriction on the parameters by using analytic continuation.}
\begin{proof}[Proof of Theorem~\ref{thm:corker_hitting}]
Assume first that $q_1<q_2<\cdots$.
Using~\eqref{eq:Gfunction_general}, \eqref{eq:Qhat}, \eqref{eq:Sbar_localOp} and~\eqref{eq:SbarEpi}, we obtain 
\begin{equation}\label{eq:SbarEpi2}
	\begin{aligned}
	&\sum_{{z\in\Z}} G(x,z)\cdot(R_{(r,t]}^*)^{-1}(z,y)\\
	&= 	\sum_{{z\in\Z}} \frac{\mathbb{E}_{S_0=x}\left[\prod_{\ell=\tau+1}^{n}(q_\ell-1)\cdot \bar{Q}_{[\tau+1,n]}(S_\tau,z)\1_{\tau<n}\right]}{\prod_{\ell=1}^{n}(q_\ell-1)}\cdot (R_{(r,t]}^*)^{-1}(z,y)\\
	&= \frac{\mathbb{E}_{S_0=x}\left[\prod_{\ell=\tau+1}^{n}(q_\ell-1) \left(\sum_{{z\in\Z}}\bar{Q}_{[\tau+1,n]}(S_\tau,z)\cdot(R_{(r,t]}^*)^{-1}(z,y)\right)\1_{\tau<n}\right]}{\prod_{\ell=1}^{n}(q_\ell-1)} \\
	&=\frac{\mathbb{E}_{S_0=x}\left[\bar{\mathcal{S}}_{[\tau+1,n],(r,t]}(S_{\tau},y) \1_{\tau<n}\right]}{\prod_{\ell=1}^{n}(q_\ell-1)}
	=\bar{\mathcal{S}}_{[1,n],(r,t]}^{\mathrm{epi}(\bm{y})}(x,y).
	\end{aligned}
\end{equation}

Then, by~\eqref{eq:corKer_biorthogonal}, \eqref{eq: kernel_decomp}, \eqref{eq:Gfunction}, \eqref{eq:S_localOp} and~\eqref{eq:SbarEpi2}, we have  
\begin{align*}
	K(m,x;n,x')&=-Q_{(m,n]}(x,x')\1_{n>m}+\sum_{i=1}^n\Psi_i^{(m)}(x)\cdot \Phi_i^{(n)}(x')\\
	&=-Q_{(m,n]}(x,x')\1_{n>m}+\!\!\!\sum_{{z_1,z_2\in\Z}} \!\!\! Q_{[1,m]}^{-1}\circ R_{(r,t]}^{*}(x,z_1)\cdot G(z_1,z_2)\cdot (R_{(r,t]}^*)^{-1}(z_2,x')\\
	&= -Q_{(m,n]}(x,x')\1_{n>m}+ \mathcal{S}_{[1,m],(r,t]}\circ \bar{\mathcal{S}}_{[1,n],(r,t]}^{\mathrm{epi}(\bm{y})}(x,x').
\end{align*}

Fix now $1\leq k_1<k_2<\cdots<k_m\leq N$ and $(s_1,\dots,s_m)\in {\mathbb{R}^m}$, as in the statement of Theorem~\ref{thm:corker_hitting}.
Take the starting time $r:=0$ and choose the test function
\begin{equation*}
	g\colon \{1,\dots,N\}\times \Z\to\R, \qquad
	g(k,x):=\begin{cases}
		{-\chi_s(k_i,x)},\quad &\text{if }k=k_i\ \text{for some }1\leq i\leq m,\\
		0&\text{otherwise},
	\end{cases}
\end{equation*}
{recalling that $\chi_s(k_i,x):=\1_{x<s_i}$.}
Recall now that $Y_k(t)=x_k^{(k)}$ is the left-most particle of the $k$-th row in the point process $\sfX_N$ of Proposition~\ref{prop:DPPmarginal} and Corollary~\ref{coro:modifiedDPP}.
Therefore, by Proposition~\ref{prop:modifiedCorKer} and the above expression for the kernel, we have 
\begin{align*}
	&\mathbb{P}\left(\bigcap_{i=1}^m\{Y_{k_i}(t)\geq s_i\} \;\middle|\; Y(0)=\bm{y}\right) = \mathbb{E}\left[\prod_{1\leq j{\leq }i\leq N}(1+g(i,x_j^{(i)}))\right] \\
	&= \det(I+gK)_{\ell^2(\{1,\dots,N\}\times \mathbb{Z})}
	= \det(I-\chi_sK\chi_s)_{\ell^2(\{k_1,\dots,k_m\}\times \mathbb{Z})},
\end{align*}
where $K$ is the kernel~\eqref{eq:corker_hitting}.
This proves Theorem~\ref{thm:corker_hitting} for parameters satisfying the condition $q_1<q_2<\cdots$.

Now we extend the result to parameters $q_1,q_2,\dots$ for the most general hypotheses of the theorem.
On the one hand we can write the left hand side of~\eqref{eq:multipoint} as a sum of transition probabilities over suitable configurations:
\begin{equation*}
	\mathbb{P}\left(\bigcap_{i=1}^m\{Y_{k_i}(t)\geq s_i\} \;\middle|\; Y(0)=\bm{y}\right) =\!\!\!\!\sum_{\substack{\bm{x}\in \sfW_N\colon\\ x_{k_i}{+k_i}\geq s_i, \\ 1\leq i\leq m}} \!\!\!\!\!\! \mathcal{Q}_{0,t}({(y_1+1,\dots,y_N+N),(x_1+1,\dots,x_N+N)} ),
\end{equation*}
where $\mathcal{Q}_{0,t}$ given by~\eqref{eq:Q_threeParts} is clearly analytic in $q_i$ for each $i$.
Note that the right-hand side above is a finite sum, since $\mathcal{Q}_{0,t}(\bm{y},\bm{x})= 0$ if $t<x_i-y_i$ or $x_i-y_i<0$ for some $i$.
Hence, $\mathbb{P}\left(\bigcap_{i=1}^m\{Y_{k_i}(t)\geq s_i\} \;\middle|\; Y(0)=\bm{y}\right)$ is analytic in $q_i$ for each $i$.

On the other hand, the kernels $\mathcal{S}_{[j,k],(r,t]}(x,y)$ and $\bar{\mathcal{S}}_{[j,k],(r,t]}(x,y)$ defined through the contour integral representations~\eqref{eq:Sintegral} and~\eqref{eq:Sbarintegral} are clearly analytic in $q_i$ for each $i$.
Moreover, for the geometric random walk $S$ defined in~\eqref{eq:transitionRW}, the hitting time $\tau$ defined in \eqref{eq:hittingTime} and $S_{\tau}$ have joint distribution given by a finite sum:
\begin{equation}\label{eq:jointlaw_hitting}
	\mathbb{P}_{S_0=z_1}(\tau=k, S_{\tau}= z) = \mathbbm{1}_{z>y_{k+1}}\sum_{\substack{z_1=x_0>x_1>\cdots>x_k=z\\ x_i\leq y_{i+1}, \; 0\leq i\leq k-1}} \, \prod_{i=1}^{k}\left((q_i-1)\cdot q_i^{x_i-x_{i-1}}\right),
\end{equation}
which is also analytic in $q_i$ for all $i$.
Thus the kernel $\bar{\mathcal{S}}^{\mathrm{epi}(\bm{y})}_{[1,n],(0,t]}(x,y)$, given by
\begin{align*}
	\bar{\mathcal{S}}_{[1,n],(0,t]}^{\mathrm{epi}(\bm{y})}(x,y)&:= \frac{ \mathbb{E}_{S_0=x}[\bar{\mathcal{S}}_{[\tau+1,n],(0,t]}(S_\tau,y)\1_{\tau<n}]}{\prod_{\ell=1}^{n}(q_\ell-1)}\\
	&=\frac{{\sum_{k=0}^{n-1}\sum_{z}\mathbb{E}_{S_0=x}[\bar{\mathcal{S}}_{[k+1,n],(0,t]}(z,y)\1_{\tau=k,\, S_k=z}]}}{\prod_{\ell=1}^{n}(q_\ell-1)}\\
	&= \frac{\sum_{k=0}^{n-1}\sum_{z}\bar{\mathcal{S}}_{[k+1,n],(0,t]}(z,y)\cdot\mathbb{P}_{S_0=x}(\tau=k,S_{\tau}=z)}{\prod_{\ell=1}^{n}(q_\ell-1)},
\end{align*}
is analytic in $q_i$ for all $i$ (note that the sum over $z$ is a finite sum, as $\mathbb{P}_{S_0=x}(\tau=k,S_{\tau}=z)=0$ if $z\leq y_{k+1}$ or $z>x-k$). Note that the kernel $\bar{\mathcal{S}}_{[1,n],(0,t]}^{\mathrm{epi}(\bm{y})}(x,y)$ is analytic in each $q_i$ also at $q_i=1$, since the normalizing factor $\prod_{\ell=1}^{n}(q_\ell-1)$ in the denominator cancels out with the same factor appearing in $\bar{\mathcal{S}}_{[k+1,n],(0,t]}(z,y)\cdot\mathbb{P}_{S_0=x}(\tau=k,S_{\tau}=z)$ for any $k$ and $z$ (see \eqref{eq:Sbarintegral} and \eqref{eq:jointlaw_hitting}).
Therefore, we conclude that the kernel $K(m,x;n,x')$ defined in~\eqref{eq:corker_hitting} is analytic in $q_i$ for each $i$, and so is the Fredholm determinant $\det(I-\chi_sK\chi_s)$ associated to it. 
To be more precise, we need absolute convergence of the series expansion for the Fredholm determinant, which is guaranteed by the fact that $\chi_sK\chi_s$ is a trace class operator.
The trace class property can be proved in a similar way as in~\cite[Appendix A and B]{matetskiQuastelRemenik21} (there, uniform bounds on the trace norms are obtained for a family of kernels with respect to certain scaling parameters).
The only modification needed amounts to replacing the weight function $e^{t(z-1/2)}(z-1)^{n}z^{-n}$ that appears in~\cite[(2.28) and (2.29)]{matetskiQuastelRemenik21} with $\prod_{\ell=j}^{k}(q_\ell z^{-1}-1)\cdot \prod_{\ell=r+1}^{t}(1+p_\ell z)$ for our discrete-time inhomogeneous setup.
These weight functions, coming from the contour integral representations of $\mathcal{S}_{[j,k],(r,t]}(x,y)$ and $\bar{\mathcal{S}}_{[j,k],(r,t]}(x,y)$, as shown in~\eqref{eq:Sintegral} and~\eqref{eq:Sbarintegral}, are independent of the entries $x$ and $y$ and remain uniformly bounded on the contours, so one can bound the trace norm almost identically as in~\cite{matetskiQuastelRemenik21}; we omit the details.

Now, for fixed parameters $\{p_i\}$, both sides of~\eqref{eq:multipoint} admit analytic continuation to all $q_j$ satisfying $0<q_j<\min\{p_i^{-1}\}$ for all $j$ and they agree for $q_1<q_2<\cdots$, hence they must agree for all $0<q_j<\min\{p_i^{-1}\}$, not necessarily ordered. 
\end{proof}

	\begin{remark}\label{rem:qless1}
	We have already commented in the introduction about the assumption $p_i q_j<1$, which is innocent due to a particle-hole duality.
	The second assumption of the theorem, i.e.\ $q_j>1$ for all $j$, is also innocent, as it can be removed by replacing $p_i$ with $\tilde{p}_i:=q p_i$, and $q_i$ with $\tilde{q}_i:=q_i/q$, {for some choice of $q>0$. 
	Tuning $q$, one can recover any $N$-tuple $(\tilde{q}_1,\dots,\tilde{q}_N)$ of positive parameters.}
	On the other hand, this will not change the jumping rates, {since $\tilde{p}_i \tilde{q}_j = p_iq_j$.
	Therefore}, the Fredholm determinant on the right hand side of~\eqref{eq:multipoint} does not depend on the choice of $q$.
	{This can also be seen from the fact that,} for any two choices of the renormalizing constants $q$ and $q'$, the corresponding kernels $K_{q}$ and $K_{q'}$ are off by a conjugation, which does not affect the Fredholm determinant:
	\begin{equation*}
		K_{q'}(m,x;n,x')= \left(\frac{q'}{q}\right)^{x-x'}K_{q}(m,x;n,x').
	\end{equation*} 
\end{remark}
	
	{
	\begin{example}[Step initial configuration]
		The simplest case in which the random walk hitting kernel $\bar{\mathcal{S}}_{[1,n],(0,t]}^{\mathrm{epi}(\bm{y})}$ can be explicitly written out is the step initial configuration, i.e.\ $y_i=-i$ for $i=1,\dots,N$.
		In this case, if the random walk $S$ starts at $S_0=x\leq y_1$, then it will never hit the strict epigraph of the curve $(i,y_{i+1})_{0\leq i\leq N-1}$.
		This happens because the random walk $S$ moves strictly to the left, hence
		\begin{equation*}
			S_k\leq S_0-k=x-k\leq y_1-k=y_{k+1}
		\end{equation*} 
		for all $k\geq 0$.
		Therefore, in this case, we have $\1_{\tau<n}=\1_{\tau=0}$ and
		\begin{equation}
			\bar{\mathcal{S}}_{[1,n],(0,t]}^{\mathrm{epi}(\bm{y})}(x,y)= \frac{\1_{x>y_1} \bar{\mathcal{S}}_{[1,n],(0,t]}(x,y)}{\prod_{\ell=1}^{n}(q_\ell-1)}.
		\end{equation}
		The correlation kernel for the step initial configuration takes, thus, the form 
		\begin{equation}
			K(m,x;n,x')=- Q_{(m,n]}\1_{n>m} + \sum_{z\geq 0}\mathcal{S}_{[1,m],(0,t]}(x,z)\cdot\frac{\bar{\mathcal{S}}_{[1,n],(0,t]}(z,x')}{\prod_{\ell=1}^{n}(q_\ell-1)}.
		\end{equation}
		Using the contour integral representations~\eqref{eq:Sintegral}-\eqref{eq:Sbarintegral}, it is easy to check that 
		\[
		\begin{aligned}
			K(&m,x;n,x')=- Q_{(m,n]}\1_{n>m} 
			+ \oint_{\Gamma_0}\frac{\mathrm{d}z}{2\pi \mathrm{i}}\oint_{\Gamma_{\bm{q}}}\frac{\mathrm{d}w}{2\pi \mathrm{i}} \frac{1}{z-w} \frac{z^{-x-1}\prod_{\ell=1}^m(q_\ell z^{-1}-1)}{w^{-x'}\prod_{\ell=1}^n(q_\ell w^{-1}-1)} \prod_{\ell=1}^{t}\frac{1+p_\ell z}{1+p_\ell w} ,
		\end{aligned}
		\]
		where $\Gamma_0$ and $\Gamma_{\bm{q}}$ are contours as in Theorem~\ref{thm:corker_hitting}, with the additional property that $|z|<|w|$ for all $z\in\Gamma_0$ and $w\in\Gamma_{\bm{q}}$.
	\end{example}
	}

\subsection{Proof of Proposition~\ref{prop:generalG}} \label{subsec:proofRWformula}

		In this subsection we prove Proposition~\ref{prop:generalG} by induction.
		For induction purposes, we define the following more general kernels:
		\begin{equation}\label{eq:Ggeneral}
			G_{j,k}^{(n)}(z_1,z_2):= \sum_{i=j+1}^{n-k}  Q_{[j+1,i]}(z_1,y_i)\cdot h_{n-i}^n(k,z_2),
		\end{equation}
		where $0\leq k\leq n-1$ and $0\leq j\leq n-k-1$.
		Then, by~\eqref{eq:Gfunction}, we have
		\begin{equation*}
			G(z_1,z_2) = G^{(n)}_{0,0}(z_1,z_2).
		\end{equation*}
		We will prove the following generalization of Proposition~\ref{prop:generalG}:
		\begin{equation}\label{eq:Ginduction}
			G_{j,k}^{(n)}(z_1,z_2) = \frac{\mathbb{E}_{S_{j}=z_1}\left[\hat{Q}_{[\tau^{j}+1,n-k]}(S_{\tau^{j}},z_2)\1_{\tau^{j}<n-k}\right]}{\prod_{\ell=j+1}^{n-k}(q_\ell-1)},
		\end{equation}
		for all $z_1,z_2\in \mathbb{Z}$, $0\leq k\leq n-1$ and $0\leq j\leq n-k-1$, where $S$ is the geometric random walk defined in~\eqref{eq:transitionRW} and
		\begin{equation}
		\label{eq:hittingTime_general}
			\tau^{j}:= \min\{m\in \{j,\dots,n\}: S_m>y_{m+1}\}.
		\end{equation}
		Notice that $\tau^0=\tau$, with $\tau$ defined in~\eqref{eq:hittingTime}.
		To prove~\eqref{eq:Ginduction}, we use a backward induction on $k$ and $j$.
		
		{For any $0\leq k\leq n-1$ and $j=n-k-1$, we have 
		\[
		\begin{aligned}
			G_{j,k}^{(n)}(z_1,z_2)&= Q_{n-k}(z_1,y_{n-k})\cdot h_{k}^{n}(k,z_2) \\
			&= \1_{z_1>y_{n-k}}\cdot q_{n-k}^{z_2-z_1} \\
			&= \P_{S_{n-k-1}=z_1}(\tau^{n-k-1}=n-k-1)\cdot \bar{Q}_{[n-k,n-k]}(z_1,z_2) \\
			&= \frac{\mathbb{E}_{S_{n-k-1}=z_1}\left[\hat{Q}_{[\tau^{n-k-1}+1,n-k]} (S_{\tau^{n-k-1}},z_2)\1_{\tau^{n-k-1}<n-k}\right]}{q_{n-k}-1},
		\end{aligned}
		\]
		{where the first equality follows from~\eqref{eq:Ggeneral}, the second equality from~\eqref{eq:Q} and~\eqref{eq:BVP3}, the third from~\eqref{eq:hittingTime_general} and~\eqref{eq:Qbar_j=k}, and the fourth from~\eqref{eq:Qhat}.}
		This proves~\eqref{eq:Ginduction} for $0\leq k\leq n-1$ and $j=n-k-1$.
		In particular, \eqref{eq:Ginduction} is proven for $k=n-1$ and $0\leq j\leq n-k-1$.}

		{Assume now by induction that, for some $0\leq \ell\leq n-2$, \eqref{eq:Ginduction} holds for all $k= \ell+1$ and $0\leq j\leq n-\ell-2$.
		We will show that~\eqref{eq:Ginduction} holds for $k=\ell$ and for all $0\leq j\leq n-\ell-1$, proceeding with a backward induction on $j$.
		We have already proven above the base case $k=\ell$ and $j=n-\ell-1$.}
		Assume that, for some $0\leq m\leq n-\ell-2$, \eqref{eq:Ginduction} holds for $k=\ell$ and $j= m+1$.
		We need to show~\eqref{eq:Ginduction} for $k=\ell$ and $j=m$ and, to do so, we will consider {various} cases separately.
		
		\medskip
		\noindent\textbf{Case 1}: $z_1\leq y_{m+1}$.
		In this case we have $Q_{m+1}(z_1,y_{m+1})=0$, hence by~\eqref{eq:Ggeneral} we have
		\begin{equation*}
		\begin{aligned}
			G_{m,\ell}^{(n)}(z_1,z_2)
			&=  \sum_{i=m+2}^{n-\ell}  Q_{[m+1,i]}(z_1,y_i)\cdot h_{n-i}^n(\ell,z_2)\\
			&= \sum_{y<z_1} Q_{m+1}(z_1,y)\left(\sum_{i=m+2}^{n-\ell} Q_{[m+2,i]}(y,y_i)\cdot h_{n-i}^n(\ell,z_2)\right) \\
			&= \sum_{y<z_1} \frac{\mathbb{P}_{S_{m}=z_1}[S_{m+1}=y]}{q_{m+1}-1} \cdot \frac{\mathbb{E}_{S_{m+1}=y}\left[\hat{Q}_{[\tau^{m+1}+1,n-\ell]}(S_{\tau^{m+1}},z_2)\1_{\tau^{m+1}<n-\ell}\right]}{\prod_{j=m+2}^{n-\ell}(q_j-1)},
		\end{aligned}
		\end{equation*}
		where the latter equality follows from~\eqref{eq:transitionRW} and the induction hypothesis (with $k=\ell$ and $j=m+1$).
		{Factoring out the normalization constants and using the Markov property of the random walk $S$, we obtain
		\[
			G_{m,\ell}^{(n)}(z_1,z_2)
			= \frac{\mathbb{E}_{S_{m}=z_1}\left[\hat{Q}_{[\tau^{m+1}+1,n-\ell]}(S_{\tau^{m+1}},z_2)\1_{\tau^{m+1}<n-\ell}\right]}{\prod_{j=m+1}^{n-\ell}(q_j-1)} .
		\]}
		Note now that, for $z_1\leq y_{m+1}$,
		\begin{align*}
			\tau^{m+1}\1_{S_m=z_1} = \tau^{m}\1_{S_m=z_1},
		\end{align*} 
		since the only situation when $\tau^{m}\neq \tau^{m+1}$ is $\tau^{m}=m$, which cannot happen if $S_m=z_1\leq y_{m+1}$.
		{Thus, for $z_1\leq y_{m+1}$, we have 
		\[
			G_{m,\ell}^{(n)}(z_1,z_2)
			= \frac{\mathbb{E}_{S_{m}=z_1}\left[\hat{Q}_{[\tau^{m}+1,n-\ell]}(S_{\tau^{m}},z_2)\1_{\tau^{m}<n-\ell}\right]}{\prod_{j=m+1}^{n-\ell}(q_j-1)} ,
		\]
		which proves~\eqref{eq:Ginduction} for $k=\ell$ and $j=m$ in Case~1.}
		
		\medskip
		\noindent\textbf{Case 2}: $z_1>y_{m+1}$.
		In this case, we have:
		\begin{equation}\label{eq:hitfirst}
		{\text{If}\quad S_m=z_1, \qquad\text{then}\quad \tau^m=m,}
		\end{equation}
		{since $z_1>y_{m+1}$.}
		Therefore, {recalling that $m\leq n-\ell-2$,} the right hand side of~\eqref{eq:Ginduction} for $k=\ell$ and $j=m$ reduces to $\bar Q_{[m+1,n-\ell]}(z_1,z_2)$ and it suffices to prove that
		\begin{equation}\label{eq:Ginduction_reduced}
			G_{m,\ell}^{(n)}(z_1,z_2)= \bar Q_{[m+1,n-\ell]}(z_1,z_2).
		\end{equation}
		We now check~\eqref{eq:Ginduction_reduced} in three distinct subcases, using the recurrence relation~\eqref{eq:onestep}.
		
		\medskip
		\noindent \textbf{Case 2.1}: $z_2\leq y_{n-\ell}$.
		Then, by the first recurrence relation in~\eqref{eq:onestep}, we have
		\begin{equation*}
			h_{n-i}^n(\ell,z_2) = \sum_{y=z_2+1}^{y_{n-\ell}} q_{n-\ell}^{z_2-y}\cdot h_{n-i}^n(\ell+1,y),
		\end{equation*}
		for any $m+1\leq i\leq n-\ell-1$.
		Using the latter equality and~\eqref{eq:BVP3}, we obtain
		\begin{equation}
		\label{eq:Gcase2.1}
		\begin{split}
			&G_{m,\ell}^{(n)}(z_1,z_2) = \sum_{i=m+1}^{n-\ell} Q_{[m+1,i]}(z_1,y_i) \cdot h_{n-i}^n(\ell,z_2) \\
			&= \sum_{i=m+1}^{n-\ell-1} Q_{[m+1,i]}(z_1,y_i) \left(\sum_{y=z_2+1}^{y_{n-\ell}}q_{n-\ell}^{z_2-y}\cdot h_{n-i}^n(\ell+1,y)\right) + Q_{[m+1,n-\ell]}(z_1,y_{n-\ell})\cdot h_{\ell}^n(\ell,z_2) \\
			&=\sum_{y=z_2+1}^{y_{n-\ell}}q_{n-\ell}^{z_2-y}\cdot \left(\sum_{i=m+1}^{n-\ell-1} Q_{[m+1,i]}(z_1,y_i)\cdot h_{n-i}^n(\ell+1,y)\right) + Q_{[m+1,n-\ell]}(z_1,y_{n-\ell})\cdot q_{n-\ell}^{z_2-y_{n-\ell}}.
		\end{split}
		\end{equation}
		We recognize the sum inside the big parentheses in the last line above to be $G_{m,\ell+1}^{(n)}(z_1,y)$.
		Applying the induction hypothesis with $k=\ell+1$ and $j=m\leq n-(\ell+1)-1$, we may rewrite it as
		\begin{align*}
			\sum_{i=m+1}^{n-\ell-1} Q_{[m+1,i]}(z_1,y_i)\cdot h_{n-i}^n(\ell+1,y)
			&= \frac{\mathbb{E}_{S_m=z_1}[\hat{Q}_{[\tau^{m}+1,n-\ell-1]}(S_{\tau^m},y)\1_{\tau^{m}<n-\ell-1}]}{\prod_{j=m+1}^{n-\ell-1}(q_j-1)}\\
			&= \bar{Q}_{[m+1,n-\ell-1]}(z_1,y),
		\end{align*}
		where the latter equality follows from~\eqref{eq:hitfirst} {and the fact that $m\leq n-\ell-2$.
		Notice now that, for all $z_2< y\leq y_{n-\ell}$, by the assumptions corresponding to Case~2 and Case~2.1, we have $z_2 < y\leq y_{n-\ell}\leq y_{m+1}<z_1$.
		Therefore, by~\eqref{eq:Qbar}, we have $\bar{Q}_{[m+1,n-\ell-1]}(z_1,y)= Q_{[m+1,n-\ell-1]}(z_1,y)$ and $\bar{Q}_{[m+1,n-\ell]}(z_1,z_2)= Q_{[m+1,n-\ell]}(z_1,z_2)$.
		It then follows from~\eqref{eq:Gcase2.1} that
		\[
		\begin{aligned}
		G_{m,\ell}^{(n)}(z_1,z_2)
		&=\sum_{y=z_2+1}^{y_{n-\ell}}q_{n-\ell}^{z_2-y}\cdot Q_{[m+1,n-\ell-1]}(z_1,y) + \sum_{y_{n-\ell}<y<z_1}q_{n-\ell}^{z_2-y}\cdot Q_{[m+1,n-\ell-1]}(z_1,y). \\
		&= \sum_{z_2< y<z_1} q_{n-\ell}^{z_2-y}\cdot Q_{[m+1,n-\ell-1]}(z_1,y)
		= Q_{[m+1,n-\ell]}(z_1,z_2)
		= \bar{Q}_{[m+1,n-\ell]}(z_1,z_2),
		\end{aligned}
		\]
		which proves~\eqref{eq:Ginduction_reduced} in Case~2.1.}
		
		\medskip
		\noindent \textbf{Case 2.2}: $y_{n-\ell}<z_2<z_1$.
		The computation is similar to the one in Case~2.1, except that we need to use the second recurrence relation for $h_{n-i}^n(\ell,z_2)$ in~\eqref{eq:onestep}, i.e.\
		\begin{equation*}
			h_{n-i}^{n}(\ell,z_2) = -\sum_{y=y_{n-\ell}+1}^{z_2}q_{n-\ell}^{z_2-y}\cdot h_{n-i}^n(\ell+1,y),
		\end{equation*}
		since $z_2>y_{n-\ell}$.
		{Following a similar computation as in Case~2.1, we arrive at
		\begin{equation}
		\label{eq:Gcase2.2}
		\begin{aligned}
			G_{m,\ell}^{(n)}(z_1,z_2)
			=-\!\!\!\!\sum_{y=y_{n-\ell}+1}^{z_2} \!\!\!\! q_{n-\ell}^{z_2-y}\cdot \bar{Q}_{[m+1,n-\ell-1]}(z_1,y) + Q_{[m+1,n-\ell]}(z_1,y_{n-\ell})\cdot q_{n-\ell}^{z_2-y_{n-\ell}}.
		\end{aligned}
		\end{equation}		
		Notice now that, for all $y_{n-\ell}< y\leq z_2 $, by the assumptions corresponding to Case~2 and Case~2.2, we have $y_{n-\ell}< y\leq z_2<z_1$.
		Therefore, by~\eqref{eq:Qbar}, we have $\bar{Q}_{[m+1,n-\ell-1]}(z_1,y)= Q_{[m+1,n-\ell-1]}(z_1,y)$ and $\bar{Q}_{[m+1,n-\ell]}(z_1,z_2)= Q_{[m+1,n-\ell]}(z_1,z_2)$, as in Case~2.1.
		We then deduce that
		\[
		\begin{aligned}
		G_{m,\ell}^{(n)}(z_1,z_2)
		&=-\sum_{y=y_{n-\ell}+1}^{z_2}q_{n-\ell}^{z_2-y}\cdot Q_{[m+1,n-\ell-1]}(z_1,y) + \sum_{y_{n-\ell}<y<z_1}q_{n-\ell}^{z_2-y}\cdot Q_{[m+1,n-\ell-1]}(z_1,y). \\
		&= \sum_{z_2<y<z_1}q_{n-\ell}^{z_2-y}\cdot Q_{[m+1,n-\ell-1]}(z_1,y)
		= Q_{[m+1,n-\ell]}(z_1,z_2)
		= \bar{Q}_{[m+1,n-\ell]}(z_1,z_2),
		\end{aligned}
		\]
		which proves~\eqref{eq:Ginduction_reduced} in Case~2.2.}
		
		\medskip
		\noindent\textbf{Case 2.3}: $z_2\geq z_1$.
		{Given the assumptions corresponding to Case~2 and Case~2.3, we now have $z_2\geq z_1 > y_{m+1}\geq y_{n-\ell}$.
		Therefore, similarly to~Case~2.2, we apply the second recurrence relation for $h_{n-i}^{n}(\ell,z_2)$ in~\eqref{eq:onestep} and arrive at~\eqref{eq:Gcase2.2}.
		However, this time, $\bar{Q}_{[m+1,n-\ell-1]}(z_1,y)$ takes different forms for $z_1>y>y_{n-\ell}$ and $z_2\geq y\geq z_1$.
		We then split the sum over $y$ in~\eqref{eq:Gcase2.2} accordingly and compute
		\begin{align*}
		\begin{split}
			&\sum_{y=y_{n-\ell}+1}^{z_2}q_{n-\ell}^{z_2-y}\cdot \bar{Q}_{[m+1,n-\ell-1]}(z_1,y)\\
			&= \sum_{z_1>y>y_{n-\ell}}q_{n-\ell}^{z_2-y}\cdot  Q_{[m+1,n-\ell-1]}(z_1,y) + (-1)^{n-\ell-m-2}\sum_{z_2\geq y\geq z_1} q_{n-\ell}^{z_2-y}\cdot Q^{\dagger}_{[m+1,n-\ell-1]}(z_1,y) \\
			&= Q_{[m+1,n-\ell]}(z_1,y_{n-\ell}) \cdot q_{n-\ell}^{z_2-y_{n-\ell}} - (-1)^{n-\ell-m-1}Q^{\dagger}_{[m+1,n-\ell]}(z_1,z_2) ,
			\end{split}
		\end{align*}
		where the latter equality is due to~\eqref{eq:Q} and~\eqref{eq:Qdagger}.
		Combining this with~\eqref{eq:Gcase2.2}, after a cancellation, we obtain
		\[
		G^{(n)}_{m,\ell}(z_1,z_2)
		= (-1)^{n-\ell-(m+1)} \cdot Q^{\dagger}_{[m+1,n-\ell]}(z_1,z_2)
		= \bar{Q}_{[m+1,n-\ell]}(z_1,z_2),
		\]
		where the latter equality follows from~\eqref{eq:Qbar} and the assumption $z_2\geq z_1$.
		This proves~\eqref{eq:Ginduction_reduced} in Case~2.3 and, thus, completes the proof of Proposition~\ref{prop:generalG}.}
		
	\vskip 4mm
	\noindent
	{\bf Acknowledgements}.
	We thank Konstantin Matetski, Daniel Remenik and Travis Scrimshaw for helpful comments on an earlier version of this manuscript.
	The work of E.B.\ was supported by the ERC grant 669306.
	The work of Y.L., A.S.\ and N.Z.\ was supported by the EPSRC grant EP/R024456/1.
	N.Z.\ also acknowledges hospitality and support from the Galileo Galilei Institute and from the scientific program ``Randomness, Integrability, and Universality'', during which work on this project was conducted.
	\vskip 2mm
	\noindent
	{\bf Declarations}: The authors have no relevant financial or non-financial interests to disclose. The manuscript has no associated data.

\printbibliography

\end{document}